\documentclass[11pt,a4paper]{article}
\usepackage[T1]{fontenc}
\usepackage[utf8]{inputenc}
\usepackage{lmodern}
\usepackage[margin=28mm]{geometry}
\usepackage{amsmath,amssymb,amsthm,mathtools,mathrsfs}
\usepackage{microtype}
\usepackage{enumitem,booktabs,longtable,array}
\usepackage[numbers,sort&compress]{natbib}
\usepackage{hyperref}
\hypersetup{colorlinks=true,linkcolor=blue,citecolor=blue,urlcolor=blue,
  pdftitle={Long-time action stability for perturbed KdV near finite-gap tori}}
\newcommand{\R}{\mathbb R}
\newcommand{\C}{\mathbb C}
\newcommand{\N}{\mathbb N}
\newcommand{\Z}{\mathbb Z}
\newcommand{\T}{\mathbb T}
\newcommand{\eps}{\varepsilon}
\newtheorem{theorem}{Theorem}[section]
\newtheorem*{geometricmaintheorem}{Main geometric theorem}
\newtheorem{lemma}[theorem]{Lemma}
\newtheorem{proposition}[theorem]{Proposition}
\newtheorem{corollary}[theorem]{Corollary}
\theoremstyle{definition}

\newtheorem{assumption}[theorem]{Assumption}
\theoremstyle{remark}

\numberwithin{equation}{section}
\setlist{topsep=4pt,itemsep=3pt}

\title{Long-time action stability for perturbed KdV\\
near finite-gap tori}
\author{Guan HUANG}
\date{}

\begin{document}
\maketitle
\begin{abstract}
We prove action and spectral stability near finite-gap tori of the periodic
Korteweg--de Vries equation under autonomous analytic Hamiltonian perturbations
of size $\varepsilon$. The estimates are uniform over prescribed compact
families, without small-amplitude or nonresonance assumptions, and allow nearby
initial data with infinitely many open gaps. Assuming a strong solution and a
uniform higher Sobolev bound on $[0,T_0/\varepsilon]$, we allow the normal gradient of the perturbing Hamiltonian to lose one derivative. For sufficiently small
$\varepsilon$ and initial normal Birkhoff norm $r_0$, the tangential action
variations and weighted sum of normal actions are bounded by
$C_\delta(r_0^2+\varepsilon^{2\delta})$, while the spectral-gap variation in
$\ell^2$ is bounded by $C_\delta(r_0+\varepsilon^\delta)$, for every
$0<\delta<1/4$. A key ingredient is coercivity of the negative Hessian of the
nonlinear KdV Hamiltonian on the full $\ell^2$ action space at every finite-gap
action. We combine this with resonance analysis and corrected action observables
evaluated along the solution. The hypotheses are verified for several local and
nonlocal Hamiltonian perturbations.
\end{abstract}

\section{Introduction}\label{sec:introduction}

Hamiltonian perturbations of an integrable system raise two related
questions: whether invariant tori persist and how far the unperturbed action
variables can drift over long times. KAM theory addresses the first under
suitable nondegeneracy and nonresonance assumptions, while averaging and
normal-form methods provide tools for the second. Action stability can hold
without persistence of the underlying tori or closeness in phase to an
unperturbed trajectory.

The periodic Korteweg--de Vries (KdV) equation admits finite-dimensional
invariant tori consisting of finite-gap potentials \cite{KP2003,HK2014}.
We study long-time action and spectral stability near prescribed compact
families of these tori under autonomous analytic Hamiltonian perturbations.
The reference families are fixed independently of the perturbation size and
need not have small amplitude. Our estimates are uniform in the initial
phases and require no nonresonance condition. Nearby initial data may have
infinitely many open spectral gaps. Although the full angular average of the
perturbative action drift vanishes, its time average along a resonant
unperturbed orbit need not. To control this resonant drift uniformly over
the reference family, we combine normal-form averaging with a geometric
confinement argument based on the curvature of the nonlinear KdV Hamiltonian.
\subsection{KdV, finite-gap coordinates, and initial data}
\label{subsec:intro-framework}

We work on the period-one circle $\T=\R/\Z$ and denote its real, mean-zero
Sobolev space by $H^s_0$. An autonomous Hamiltonian perturbation of periodic
KdV has the form
\[
 w_t=-w_{xxx}+6ww_x+\eps\partial_x\nabla_w\mathcal F(w).
\]
Here $\mathcal F$ is the perturbing functional, and $\nabla_w\mathcal F$
denotes its $L^2$ gradient modulo spatial constants on the mean-zero phase
space. Our main stability theorem is formulated directly in local canonical
coordinates and does not require a globally defined functional $\mathcal F$
in physical space.

The unperturbed KdV flow preserves the periodic and antiperiodic spectra of
the Hill operator $L_w=-\partial_x^2+w$ \cite{KP2003}. We write
$\gamma(w)=(\gamma_j(w))_{j\ge1}$ for its nonnegative spectral gap lengths,
indexed as in Section~\ref{sec:setup}. A potential $w$ is called \emph{finite-gap}
if only finitely many of these gaps are open ($\gamma_j(w)>0$). We measure
spectral stability by the variation of $\gamma(w(t))$ from its initial value.

Birkhoff coordinates connect this spectral description with Sobolev
regularity. For $q\in\R$, define the weighted space of real coordinate pairs
by
\begin{equation}
 h^q=\left\{(x_j,y_j)_{j\ge1}:
    \sum_{j\ge1}j^{2q}(x_j^2+y_j^2)<\infty\right\},
 \qquad
 \|(x,y)\|_q^2=\sum_{j\ge1}j^{2q}(x_j^2+y_j^2).
 \label{setup:sequence-norm}
\end{equation}
When restricting a sequence to a subset of modes, we retain the original
indices $j$ and weights $j^{2q}$. For $s\ge0$, the canonical Birkhoff map
$\mathcal B(w)=(x_j(w),y_j(w))_{j\ge1}$ defines a compatible family of
real-analytic diffeomorphisms, with physical action variables normalized by
\[
 \mathcal B:H^s_0\longrightarrow h^{s+1/2},\qquad
 I_j(w)=\frac{x_j(w)^2+y_j(w)^2}{2};
\]
see \cite[Theorem~1.3]{KST2013}. The actions are conserved by unperturbed KdV,
and $I_j(w)=0$ precisely when the $j$-th spectral gap is closed. Finite-gap
potentials therefore have finitely supported actions and are smooth. Fixing
the nonzero actions and varying the corresponding phases yields a
finite-dimensional isospectral torus \cite{KP2003,KST2013,KMMT2016}.

To set up the stability problem, fix an integer $N\ge3$, set
$\sigma=N+1/2$, and choose a finite tangential index set $A\subset\N$ together
with a nonempty compact set $\mathcal K\Subset(0,\infty)^A$. Each
$\xi\in\mathcal K$, extended by zero outside $A$, specifies a reference
finite-gap torus. We consider initial data $w_0\in H^N_0$ with tangential
actions $I_A(w_0)=\xi\in\mathcal K$. Their tangential phases are arbitrary,
and their normal actions $I_j(w_0)$, $j\notin A$, may be nonzero. Along a
solution $w(t)$ with $w(0)=w_0$, we denote the tangential action deviations
and the normal actions by
\[
 p_a(t)=I_a(w(t))-\xi_a\quad(a\in A),\qquad
 J_j(t)=I_j(w(t))\quad(j\notin A),\qquad
 |p|_1=\sum_{a\in A}|p_a|,
\]
so that $p(0)=0$. Writing $z(t)$ for the Birkhoff pairs with indices
$j\notin A$, we measure their norm and initial size by
\[
 \|z(t)\|_\sigma^2=2\sum_{j\notin A}j^{2N+1}J_j(t),
 \qquad
 r_0^2:=\|z(0)\|_\sigma^2=2\sum_{j\notin A}j^{2N+1}I_j(w_0).
\]
Thus $r_0=0$ precisely when $w_0$ lies on the reference torus, while nearby
data with $r_0>0$ may have infinitely many open normal gaps. In a sufficiently
small neighborhood of the reference family, the tangential actions are bounded
away from zero, so we use action--angle coordinates on the modes $a\in A$.
On the normal modes, we retain Cartesian coordinates $z$ to keep the chart
regular when gaps are closed. The precise local chart is given in
Section~\ref{sec:setup}.
\subsection{Main stability result and longer-time consequences}
\label{subsec:intro-results}

\paragraph{Analytic and evolution hypotheses.}
In the local coordinates of Section~\ref{sec:setup}, write
$H_\xi=h(I)+\eps f_\xi$, where $h$ is the integrable KdV Hamiltonian.
Assumption~\ref{ass:contract} imposes local analyticity and an independent
evolution hypothesis. With the finite tangential variables treated as
parameters, the represented normal gradient of $f_\xi$ satisfies
\[
 \nabla_zf_\xi:h^\sigma\longrightarrow h^{\sigma-1},\qquad
 \nabla_zf_\xi:h^{\sigma+1}\longrightarrow h^\sigma,
\]
with the uniform local holomorphic and derivative bounds specified in
Assumption~\ref{ass:contract}(1). For a fixed slow-time horizon $T_0>0$,
we also assume that each admissible initial datum generates a strong solution
in Cartesian Birkhoff coordinates on $[0,T_0/\eps]$, with the uniform bound
\[
 \sup_{0\le t\le T_0/\eps}\|z(t)\|_{\sigma+1}\le H_+.
\]
Existence on this interval and the bound $H_+$ must be supplied independently;
they do not follow from the analyticity of $f_\xi$ or the smallness of $r_0$.

\paragraph{Main result (Theorem~\ref{thm:main}).}
Under Assumption~\ref{ass:contract}, fix $0<\delta<1/4$. For sufficiently
small $r_0\ge0$ and $\eps>0$, with independent thresholds, every admissible
trajectory remains in the coordinate tube and satisfies, for all
$0\le t\le T_0/\eps$,
\begin{equation}\label{eq:intro-stability}
 \begin{gathered}
  \|\gamma(w(t))-\gamma(w_0)\|_{\ell^2}
                        \le C_\delta(r_0+\eps^\delta),\\
  |p(t)|_1+2\sum_{j\notin A}j^{2N+1}J_j(t)
                        \le C_\delta(r_0^2+\eps^{2\delta}).
 \end{gathered}
\end{equation}
The second estimate controls both the tangential action drift $|p(t)|_1$
and the squared normal norm $\|z(t)\|_\sigma^2$, whose initial value is
$r_0^2$. Since $|J_j(t)-J_j(0)|\le J_j(t)+J_j(0)$, increasing $C_\delta$ if
necessary gives the same bound for the weighted action variation
$|p(t)|_1+\sum_{j\notin A}j^{2N+1}|I_j(w(t))-I_j(w_0)|$.

The constant $C_\delta$ and the smallness thresholds depend only on $\delta$,
the reference data $(N,A,\mathcal K)$, the fixed coordinate neighborhoods and
admissible normal radius, and the uniform analytic and evolution bounds in
Assumption~\ref{ass:contract}, including $T_0$ and $H_+$. They are uniform over
$\xi\in\mathcal K$, the initial phases, and all admissible normal tails;
$C_\delta$ is independent of $\eps$ and $r_0$.

In physical space, \eqref{eq:intro-stability} implies that $w(t)$ stays within
$H^N$ distance $O(r_0+\eps^\delta)$ of the reference torus with tangential
actions $\xi=I_A(w_0)$. This is a distance estimate to the torus and does not
assert phase shadowing of a particular unperturbed orbit. The case
$A=\varnothing$ is also included: the reference torus is the zero potential,
and the initial data are small in $H^N_0$.

For shrinking initial neighborhoods, Corollary~\ref{cor:shrinking-data}
gives the bounds
\[
 \begin{gathered}
 |p(t)|_1\le\eps^{2\delta},\qquad
 \sum_{j\notin A}j^{2N+1}J_j(t)\le\eps^{2\delta},\\
 \|\gamma(w(t))-\gamma(w_0)\|_{\ell^2}\le\eps^\delta,
 \qquad 0\le t\le T_0/\eps,
 \end{gathered}
\]
whenever $r_0\le c_\delta\eps^\delta$, for a suitable fixed $c_\delta>0$ and
all sufficiently small $\eps$. The constants are absorbed by choosing $c_\delta$ sufficiently small and applying
Theorem~\ref{thm:main} with a slightly larger exponent. This includes exact
finite-gap data ($r_0=0$) and all shrinking neighborhoods with
$r_0=O(\eps^\kappa)$ for a fixed $\kappa>\delta$.

\paragraph{Longer intervals.}
For exact finite-gap data and suitable shrinking initial neighborhoods,
Section~\ref{sec:closure} gives conditional stability estimates on
logarithmically and polynomially longer intervals. The supplied strong
solution and the analytic and higher-norm bounds must hold on the proposed
interval, with the coordinate neighborhoods, angular strips, and bound
constants independent of $\eps$, the admissible initial datum, and the
slow-time length. The polynomial estimates additionally require a uniform
$h^{\sigma+2}$ bound on the normal component.
Theorems~\ref{thm:long-logarithmic} and~\ref{thm:long-polynomial} and
Corollary~\ref{cor:long-polynomial-rates} give the precise hypotheses,
parameter ranges, and dependence on the initial normal size.
\subsection{Relation to earlier work}
\label{subsec:intro-prior-work}

This work was motivated by an unpublished manuscript of Kappeler and
Kuksin \cite{KappelerKuksin2009Private}, kindly communicated to the author by
S.~Kuksin. An earlier qualitative averaging result for small finite-gap
initial data was announced by Kuksin
\cite[Section~3, Theorem~2]{Kuksin2008ThreeTheorems}. In the unpublished
manuscript, Kappeler and Kuksin state a spectral-gap stability estimate on
time intervals of order $\varepsilon^{-1}$ with the same strict exponent
range $0<\delta<1/4$, assuming exact finite-gap initial data whose $H^1$ norm
lies below a fixed $\varepsilon$-independent smallness threshold and whose
open gaps are bounded away from zero. Their framework also assumes a
two-level analytic perturbation losing at most one derivative, together
with an independently supplied solution and higher Sobolev bound on
$[0,T_0/\varepsilon]$.

Under Assumption~\ref{ass:contract}, Theorem~\ref{thm:main} establishes uniform
stability near any prescribed compact family of finite-gap tori with a fixed
open-gap set, at arbitrary fixed amplitude. It also allows nearby initial
data with infinitely many open normal gaps and tracks the contribution of
the initial normal size $r_0$. The geometric input is coercivity of the
negative Hessian of the nonlinear part of the KdV Hamiltonian on the full
action space $\ell^2$ at every finite-gap action.
Appendix~\ref{sec:finitegap-hessian} proves this property, which yields
uniform curvature bounds near the reference family.

Bobenko and Kuksin \cite{BobenkoKuksin1991} proved the nondegeneracy of
periodic finite-gap KdV families needed for KAM persistence and averaging
results. For smooth semilinear Hamiltonian perturbations of KdV with local
density $f(x,u)$, Kappeler and Montalto \cite[Theorem~1.1]{KM2021} prove
existence, uniqueness, and $O(\eps)$ orbital stability in sufficiently high
Sobolev regularity for times of order $\eps^{-2}$. Their theorem applies to
initial data within distance $O(\eps)$ of finite-gap tori of arbitrary fixed
amplitude whose frequencies satisfy quantitative Diophantine and Melnikov
conditions through order three. Where both sets of hypotheses apply and
the initial distance is $O(\eps)$, their result gives a longer stability
interval and a smaller orbital error than Theorem~\ref{thm:main}. Our
theorem covers the prescribed torus family without arithmetic exclusions
and allows a small initial normal size $r_0$ independent of $\eps$,
conditional on the supplied evolution bounds.

For the Benjamin--Ono equation, Bambusi and G\'erard
\cite[Theorem~2.2]{BambusiGerard2024} establish exponentially long action
stability near tori whose first finitely many gaps are open and whose
remaining gaps are closed. Their result includes resonant tori of arbitrary
fixed amplitude. It assumes that the perturbing Hamiltonian and its vector
field are analytic and bounded on bounded sets near the finite-gap family,
with the vector field taking values in the energy space $H^{1/2}$. Our
proof likewise combines resonant normal forms with energy confinement,
using the KdV Hessian coercivity described above. To accommodate the
permitted loss of one derivative, we use corrected action observables
along the supplied trajectory, without requiring a flow for the
normal-form generator.

Other long-time results for perturbed KdV address different regimes.
For smoothing perturbations, Huang \cite{Huang2013} proves averaging on
sets of asymptotically full measure with respect to suitable Gaussian
measures in Birkhoff coordinates. Under the hypotheses of
\cite[Definition~0.1 and Theorem~0.2, arXiv:1310.5462v2]{HuangQuasi2013},
he extends this framework to regular measures that are
$\eps$-quasi-invariant. These typical-data conclusions do not by themselves
give uniform estimates over a prescribed finite-gap family. Such a family
is a null set for every regular measure in Huang's sense: it lies in the
zero set of a nonzero transverse Birkhoff coordinate.

Small-amplitude results include the normal-form analysis of generic small
initial data by Bernier and Gr\'ebert \cite{BernierGrebert2021}, the
averaging theorem of Yuan and Zhang \cite{YZ2016}, and the stability theorem
of Yang and Zeng \cite{YangZeng2025}. These results concern the
small-amplitude regime near the zero potential; in our finite-gap setting,
the nonzero reference actions are fixed independently of $\eps$.
\subsection{Hamiltonian applications}
\label{subsec:intro-applications}

Section~\ref{sec:application} verifies the analytic and evolution hypotheses
for several local and nonlocal autonomous Hamiltonian perturbations, using
Lemma~\ref{lem:app-physical-contract}. The perturbations are fixed
independently of $\eps$. All initial data below are smooth, and the stated
initial Sobolev bounds are uniform over the prescribed family,
independently of $\eps$.

\paragraph{Every fixed slow-time horizon.}
For a fixed smooth profile $a\in C^\infty(\T;\R)$, the linear functional
$\mathcal F(w)=\int_\T a(x)w(x)\,dx$ generates the additive forcing
\begin{equation}\label{eq:intro-additive}
 w_t=-w_{xxx}+6ww_x+\eps a_x(x).
\end{equation}
By combining the local smooth existence and continuation criterion of Lemma~\ref{lem:app-local}  with an inductive estimate on the KdV conservation laws, we obtain action and spectral stability for exact finite-gap initial data on $[0,T_0/\eps]$ for
every prescribed $T_0>0$; see Proposition~\ref{prop:app-transport}.

The same proposition covers the general class
$\partial_x(Vw)+\mathcal R(w)$, where $V\in C^\infty(\T;\R)$ and
$\mathcal R$ is generated by a real-analytic Hamiltonian on $L^2_0$.
For each nonnegative integer $m$, $\mathcal R$ maps $L^2_0$ analytically
into $H^m_0$, with its value and first derivative bounded on $L^2$ balls;
it also satisfies the radial estimate \eqref{app:smoothing-radial}.
No sign condition on $V$ or spatial analyticity of $a$ or $V$ is required.
Examples include nonlinear finite-rank feedback with at most linear
growth, smooth symmetric-kernel interactions, and the nonlinear spatially
smoothed functionals described in Section~\ref{subsec:app-examples}.

For these classes, Corollary~\ref{cor:app-nearby}(1) gives the two-parameter
bounds \eqref{eq:intro-stability} on the same intervals for nearby data
with a uniform initial $H^{N+2}$ bound and sufficiently small $r_0,\eps$.
The initial data may have infinitely many open gaps. The constants and
smallness thresholds may depend on $T_0$.
For the same initial Sobolev bound, exact finite-gap data or a fixed
shrinking family $r_0\le C_0\eps^\kappa$ with $\kappa>\delta$ satisfy
the bounds of Corollary~\ref{cor:shrinking-data} uniformly on
$[0,\mathsf S(\eps)/\eps]$, by the diagonal argument in
Corollary~\ref{cor:app-growing-horizon}. Here $C_0>0$ is fixed and
$\mathsf S(\eps)\to\infty$ as $\eps\downarrow0$.

\paragraph{A fixed short slow-time horizon.}
For a local density $\mathcal F_P(w)=\int_\T P(x,w(x))\,dx$, assume that
$P$ is smooth in $x$ and real analytic in $w$, with the uniform local
holomorphic bounds specified in Section~\ref{subsec:app-local-densities}.
Its Hamiltonian vector field is $P_{xw}(x,w)+P_{ww}(x,w)w_x$.
Global $L^2$ bounds need not be available for this general class.
Instead, by controlling the KdV conservation laws through Sobolev order $N+2$
on a closed energy window, we obtain smooth existence and a uniform $H^{N+2}$
bound on $[0,T_*/\eps]$; see Corollaries~\ref{cor:app-local-density}
and~\ref{cor:app-nearby}(2).
Here $T_*>0$ depends on $N$, the fixed density, and the uniform initial
$H^{N+2}$ bound, but is independent of $\eps$ and the individual datum.
For sufficiently small $r_0,\eps$, Theorem~\ref{thm:main} then gives
the action and spectral estimates \eqref{eq:intro-stability} on this
interval.

\paragraph{The integrable Gardner case.}
The defocusing quartic perturbation
$\mathcal F_{\rm G}(w)=\frac12\int_\T w^4\,dx$
generates the modified KdV nonlinearity $6w^2w_x$.
For initial data satisfying a uniform $H^{N+3}$ bound and sufficiently
small $\eps>0$ and $r_0$, conservation of the positive-quartic Gardner
energy and the Gardner-to-KdV transformation yield a uniform
$H^{N+2}$ bound for all $t\ge0$.
Proposition~\ref{prop:app-gardner} and
Lemma~\ref{lem:app-physical-contract} therefore verify
Assumption~\ref{ass:contract} on $[0,T_0/\eps]$ for every fixed $T_0>0$.

Integrability also gives sharper all-time estimates. Comparing the
mean-zero Gardner transform with the fixed spatial translate
$w(\cdot+\sqrt\eps,t)$ gives an $O(\eps)$ difference in $H^N$.
Using translation invariance of the KdV actions,
Theorem~\ref{thm:app-gardner-alltime} bounds the full action variation
by $O(\eps)$ and the normal action variation by $O(\eps(r_0+\eps))$,
both with weight $j^{2N+1}$, uniformly for all $t\ge0$.
The $\ell^2$ spectral-gap variation is $O(r_0+\eps)$.
For exact finite-gap data ($r_0=0$), the weighted  normal action norm is
$O(\eps^2)$ and the spectral-gap error is $O(\eps)$.
\subsection{Proof strategy and organization}
\label{subsec:intro-strategy}

\paragraph{Curvature in the full action space.}
On the weighted action domain of the solutions, write $h(I)=\sum_{j\ge1}(2\pi j)^3I_j+h_{\rm nl}(I)$ and let
$\ell^2_+=\{I\in\ell^2_{\R}:I_j\ge0\text{ for all }j\}$.
The function $-h_{\rm nl}$ is convex on $\ell^2_+$
\cite{KorotyaevKuksin2011} and extends analytically to a complex
$\ell^2$ neighborhood, with $-D^2h_{\rm nl}(0)=6\operatorname{Id}$
\cite{KMMT2016}. Its Hessian is coercive on a relatively open dense
subset \cite[Theorem~1.3]{KappelerMolnar2018}, but this does not cover
every prescribed finite-gap action. Likewise, nondegeneracy on the
open modes \cite{BikbaevKuksin1993,Kuksin2000} does not control
directions that open closed gaps. Our main geometric theorem establishes coercivity in all directions of $\ell^2_{\mathbb{R}}$: 

\begin{geometricmaintheorem}[Theorem~\ref{thm:finitegap-coercivity}]
For every $\xi\in\ell^2_+$ with finite support, there exists $m_\xi>0$
such that
\[
 -D^2h_{\rm nl}(\xi)[v,v]\ge m_\xi\|v\|_{\ell^2}^2
 \qquad(v\in\ell^2_{\R}).
\]
The constant is uniform when $\xi$ ranges over a compact subset of
$\ell^2_+$ consisting entirely of finite-gap actions.
\end{geometricmaintheorem}

Appendix~\ref{sec:finitegap-hessian} derives a decomposition of this
Hessian into a diagonal operator and a nonnegative finite-rank operator
from period normalizations on the reduced spectral curve. The diagonal
entries are at least $6$ on closed-gap indices; nondegeneracy of the frequency map on the open stratum gives positivity in the
remaining directions. Analyticity and
compactness then give the uniform energy defect of
Corollary~\ref{cor:finitegap-tube}.

\paragraph{Averaging along the actual solution.}
The finite-gap projection $\xi$ fixes the coordinate chart, whereas the
energy expansion and frozen frequencies are centered at the actual
initial action $I^0=I(0)$, including its normal tail. With
$\mathcal D=I-I^0$ and $\nu_j=\partial_{I_j}h(I^0)$, energy conservation gives
$\nu\cdot\mathcal D+Q_0(\mathcal D)=O(\eps)$, where
$-Q_0(\mathcal D)$ controls $\|\mathcal D\|_{\ell^2}^2$ by 
concavity along the physical action segments $[I^0, I(t)]$.
To allow the stated derivative loss, we differentiate corrected action
observables along the supplied strong solution; no flow of a
normal-form generator is required.

A weighted quadratic Taylor expansion and angular truncation first
separate high normal terms by cubic dispersion. The diagonal terms
$a_j(\varphi)J_j$ are split into the frozen angular contribution
$a_j(\varphi)J_j(0)$ and the action increment. On the remaining finite
window, sparse signed monomial labels give a polynomial conditioning
bound. The scale hierarchy of Proposition~\ref{prop:hierarchy} selects a
near-resonant space $V$ at $I^0$ and a divisor bound outside $V$,
without arithmetic exclusions.

Finite-order averaging retains all normal-linear modes and the frozen
angular contribution. Treating the running action $I$ as an independent
coefficient parameter, an auxiliary Poisson algebra permits iteration
modulo normal degree at least two, avoiding divisor estimates for general higher-degree normal monomials (cf.\ \cite[Section~2]{KappelerKramer2024}). Restoring
the discarded terms along the full trajectory yields
$\mathcal D=A_V+E$, with $A_V\in V$. For a divided label $q$, the ratio
$[\nu\cdot q]/[Dh(I)\cdot q]$ is uniformly bounded on the common action
domain used in Section~\ref{sec:filtered}, reducing the divisor loss in the
frozen scalar pairing by one power. Together with separate scalar remainder
estimates, this closes the energy bootstrap for every $0<\delta<1/4$.

The estimate for $\|z(t)\|_\sigma^2$ follows by interpolating the positive
excess $(J_j(t)-J_j(0))_+$ between weights $2\sigma-1$ and
$2\sigma+2$, then applying Gronwall's inequality. The upper moment
comes from the independently supplied $h^{\sigma+1}$ bound.

\paragraph{Organization of the paper.}
Section~\ref{sec:setup} states the coordinate and evolution hypotheses,
the main theorem, and the chain-rule and spectral-transfer lemmas.
Sections~\ref{sec:window}--\ref{sec:filtered} establish, respectively,
the finite-window reduction, resonance geometry, and observable
averaging. Section~\ref{sec:closure} closes the bootstrap and derives the
conditional longer-time estimates. Section~\ref{sec:application} verifies
Hamiltonian examples and proves the independent Gardner estimate.
Appendix~\ref{sec:finitegap-hessian} proves Theorem~\ref{thm:finitegap-coercivity}.

\section{Coordinates, assumptions and the stability statement}
\label{sec:setup}

\subsection{Local canonical coordinates and action increments}
\label{local-coordinates}

We use the real mean-zero Sobolev spaces $H^s_0$, the weighted sequence spaces
$h^q$, and the canonical Birkhoff map $\mathcal B$ introduced in
Section~\ref{sec:introduction}. The sequence norm is given by
\eqref{setup:sequence-norm}; for complex scalar sequences $(c_j)_{j\ge1}$,
we replace $x_j^2+y_j^2$ by $|c_j|^2$. Writing $w$ for the physical
potential, the KdV Hamiltonian and Gardner bracket are
\begin{equation}
 \mathcal H_{\rm KdV}(w)
 =\int_0^1\left(\frac12 w_x^2+w^3\right)\,dx,
 \qquad
 \{F,G\}_{\rm G}=\int_0^1\nabla F\,\partial_x\nabla G\,dx,
 \label{setup:physical-hamiltonian}
\end{equation}
so that the unperturbed KdV equation takes the form $w_t=-w_{xxx}+6ww_x$.

Retain the integer $N\ge3$, the weight $\sigma=N+1/2$, the finite
tangential set $A\subset\N$, and the compact family $\mathcal K\Subset(0,\infty)^A$
from Section~\ref{sec:introduction}. Set $d=|A|$ and
$\mathcal N=\N\setminus A$. For each base action $\xi\in\mathcal K$, extended
by zero on $\mathcal N$, we use local coordinates $b=(\varphi,p,z)$.
Here $\varphi\in(\R/2\pi\mathbb Z)^A$ and $p\in\R^A$ are the tangential
angles and action deviations, while $z=(z_j)_{j\in\mathcal N}\in h^\sigma$
consists of real Cartesian pairs $z_j=(u_j,v_j)$ on the normal modes.
We choose the coordinate tube uniformly for $\xi\in\mathcal K$ so that
the tangential actions $\xi_a+p_a$ are bounded away from zero on its
real part. The tube and the constants below may depend on $\mathcal K$.

On the real phase space, we identify each Birkhoff pair $(x_j,y_j)$ with
the normalized complex coordinate $(x_j-iy_j)/\sqrt2$. This is
$1/\sqrt2$ times the corresponding positive-index coordinate of
Kappeler--Schaad--Topalov \cite[Theorem~1.3]{KST2013}, with the same
physical action $I_j=(x_j^2+y_j^2)/2$. The resulting complex Birkhoff
sequence $\mathcal Z$ is given by
\begin{equation}
 \mathcal Z_a=\sqrt{\xi_a+p_a}\,e^{i\varphi_a}\quad(a\in A),
 \qquad
 \mathcal Z_j=\zeta_j=\frac{u_j-iv_j}{\sqrt2}\quad(j\in\mathcal N).
 \label{setup:polar-map}
\end{equation}
We take the square roots to be positive on the real tube and use the corresponding holomorphic branches on its complexification.  On the real phase space,
$\mathcal B^{-1}\mathcal Z$ denotes the inverse Birkhoff map applied via
the above complex identification. The physical actions and the action
projection $\pi$ are
\begin{equation}
 \begin{aligned}
  I_a&=\xi_a+p_a,\qquad
       I_j=J_j=|\zeta_j|^2=\frac{u_j^2+v_j^2}{2},\\
  \mathcal I&=(p,J),\qquad I=\xi+\mathcal I,\qquad
       \pi(\varphi,p,z)=(p,J),
 \end{aligned}
 \label{setup:actions}
\end{equation}
where $J_j=|\zeta_j|^2$ is understood on the real phase space.
On complexified Cartesian neighborhoods, we write
$\bar\zeta_j=(u_j+iv_j)/\sqrt2$ and regard $\zeta_j$ and $\bar\zeta_j$
as independent holomorphic coordinates. Thus
$J_j=\zeta_j\bar\zeta_j=(u_j^2+v_j^2)/2$ is the holomorphic extension
of the normal action; $\bar\zeta_j$ is the complex conjugate of
$\zeta_j$ only on the real phase space. Our Poisson bracket and
Hamiltonian vector field conventions are
\begin{equation}
 \{\varphi_a,p_b\}=\delta_{ab},\qquad
 \{u_j,v_k\}=\delta_{jk},\qquad
 \{\zeta_j,\bar\zeta_k\}=i\delta_{jk},\qquad
 X_HF=\{F,H\}.
 \label{setup:poisson}
\end{equation}
If $D_zF$ is represented by $\nabla_zF$ in the unweighted sequence
pairing, then $X_F^z=\mathcal J\nabla_zF$, where
$\mathcal J(u,v)=(v,-u)$ on each normal pair.

For each admissible initial datum, we choose $\xi=I_A(0)$ and keep this
base fixed along the trajectory. Thus $p(0)=0$, while the initial normal
sequence $z^0=z(0)$ may be nonzero. We retain the unshifted normal
coordinates $z(t)$, with the reference torus given by $p=z=0$.
Writing $I^0=I(0)$, we distinguish between the action displacement
$\mathcal I(t)=I(t)-\xi$ from this
base and the signed action increment $\mathcal D(t)=I(t)-I^0$ from the
actual initial action:
\begin{equation}\label{setup:initial-increment}
 \begin{gathered}
 J_j^0=\tfrac12|z_j^0|^2,\qquad I^0=\xi+(0,J^0),\qquad
 \mathcal I^0=(0,J^0),\\
 \mathcal D(t)=I(t)-I^0=\mathcal I(t)-\mathcal I^0
                     =(p(t),J(t)-J^0),\\
 r_0=\|z^0\|_\sigma,\qquad W_0=r_0^2.
 \end{gathered}
\end{equation}
The normal actions $J_j(t)$ are nonnegative, whereas the increments
$J_j(t)-J_j^0$ may have either sign. Since $j^{2\sigma}\ge1$, the
initial action offset satisfies
$\|I^0-\xi\|_{\ell^2}\le\sum_{j\in\mathcal N}J_j^0\le r_0^2/2$.

For $q\in\R$, we measure signed real action vectors in the weighted spaces
\[
 \ell^1_q(\mathcal N)
   =\left\{a=(a_j)_{j\in\mathcal N}:
                   \sum_{j\in\mathcal N}j^q|a_j|<\infty\right\},
  \qquad \mathcal A_q=\R^A\times\ell^1_q(\mathcal N).
\]
We equip $\mathcal A_q$ with its weighted $\ell^1$ norm and use two
quantities to measure the trajectory:
\begin{equation}
 \begin{aligned}
 \|(p,a)\|_{\mathcal A_q}
   &=|p|_1+\sum_{j\in\mathcal N}j^q|a_j|,
       \qquad B=\|\mathcal D\|_{\ell^2},\\
 W&=|p|_1+\|z\|_\sigma^2
   =|p|_1+2\sum_{j\in\mathcal N}j^{2N+1}J_j.
 \end{aligned}
 \label{setup:action-norms}
\end{equation}
Here $\ell^2=\R^A\times\ell^2(\mathcal N)$ carries its usual real Hilbert
norm, including the tangential components. For action vectors, we also
write $\|\cdot\|_2$ for the unweighted $\ell^2$ norm. Thus $B(t)$ measures the
unweighted action drift from $I^0$, while $W(t)$ combines the tangential
action drift with the  weighted normal action norm. Initially, $B(0)=0$
and $W(0)=W_0=r_0^2$.

When $A=\varnothing$, we set $\mathcal K=\{0\}$ and omit the tangential
variables $(\varphi,p)$ and their positive-action condition.

\subsection{Integrable frequencies and finite-gap curvature}

In these coordinates, the integrable KdV Hamiltonian decomposes as
\begin{equation}
 h(I)=\sum_{j\ge1}(2\pi j)^3I_j+h_{\rm nl}(I).
 \label{setup:integrable-hamiltonian}
\end{equation}
The full Hamiltonian is evaluated on actions in the domain of $h_{\rm nl}$
for which $\sum_{j\ge1}j^3|I_j|<\infty$. This summability holds on the
coordinate tube, since the normal coordinates belong to $h^\sigma$ and
$2\sigma=2N+1\ge7$.

The nonlinear part $h_{\rm nl}$ extends holomorphically to a complex
$\ell^2$ neighborhood of the positive action cone. Its gradient
$Dh_{\rm nl}$ is an analytic $\ell^2$-valued map with
$Dh_{\rm nl}(0)=0$ and $D^2h_{\rm nl}(0)=-6\operatorname{Id}$
\cite[Theorem~1.1, arXiv:1502.05857v1 numbering]{KMMT2016}; see also
\cite[Theorem~1.2(iv)]{KappelerMolnar2018}.

By
Theorem~\ref{thm:finitegap-coercivity}, the negative Hessian
$-D^2h_{\rm nl}(\xi)$ is strictly coercive on $\ell^2_{\mathbb R}$ at every
finite-gap action $\xi$. Operator-norm continuity of $D^2h_{\rm nl}$ and
compactness of $\mathcal K$ give a fixed complex $\ell^2$ neighborhood of
$\mathcal K$ on which the Hessian has a uniform operator-norm bound.
At real points of this neighborhood, $-D^2h_{\rm nl}$ also has a uniform
positive lower bound; see Corollary~\ref{cor:finitegap-tube} and its proof.

We shrink the coordinate tube and choose a fixed admissible normal radius
$r_{\rm adm}>0$ so that, for $r_0\le r_{\rm adm}$, the initial action $I^0$
and every current action before exit from the tube lie in a common small
$\ell^2$ ball about the fixed base $\xi$, contained in the curvature
neighborhood. Their physical segments $[I^0,I(t)]$ therefore remain in
that neighborhood for as long as the trajectory stays in the tube.
The initial action $I^0$ may have infinitely many open gaps: the curvature
bound follows from its proximity to the compact finite-gap family.

The coordinate chart and the Taylor expansion of the perturbation remain
centered at $\xi$, while we expand the integrable energy at the actual
initial action $I^0$. Write $\nu_j=\partial_{I_j}h(I^0)$, using $Dh(I^0)$
to denote this frequency sequence. The nonlinear remainder $Q_0$ and its
bounds, with uniform constants $C_h<\infty$ and $\mu>0$, are
\begin{equation}
 \begin{split}
 \nu&=Dh(I^0),\qquad
 Q_0(\mathcal D)=h_{\rm nl}(I^0+\mathcal D)
       -h_{\rm nl}(I^0)-Dh_{\rm nl}(I^0)[\mathcal D],\\
 h(I^0+\mathcal D)-h(I^0)&=\nu\cdot\mathcal D+Q_0(\mathcal D),
       \qquad Q_0(0)=DQ_0(0)=0,\\
 \|D^2Q_0\|&\le C_h,\qquad
 D^2Q_0(\mathcal D)[v,v]\le-\mu\|v\|_{\ell^2}^2
       \quad(v\in\ell^2_{\mathbb R}).
 \end{split}
 \label{setup:concavity}
\end{equation}
The remainder is defined whenever $I^0+\mathcal D$ lies in the chosen
complex $\ell^2$ neighborhood. Its derivatives are Fr\'echet derivatives
on $\ell^2$, and $\|D^2Q_0\|$ denotes the bilinear operator norm.
The Hessian upper bound is uniform there; the concavity inequality is
used at real physical actions and in real directions. The full-energy
identity in \eqref{setup:concavity} is evaluated on the weighted action
domain specified above.

The integrable frequencies are
\begin{equation}
 \Omega_j(I)=(2\pi j)^3+\beta_j(I),\qquad
 \beta(I)=Dh_{\rm nl}(I)\in\ell^2,
 \label{setup:frequencies}
\end{equation}
with tangential components indexed by $A$. On a fixed smaller complex
neighborhood of $\mathcal K$, analyticity and compactness give uniform
$\ell^2$ bounds for $\beta$ and uniform multilinear operator bounds for
its action derivatives of each fixed finite order.

The frozen sequence $\nu=\Omega(I^0)$ has components
$\nu_j=(2\pi j)^3+\beta_j(I^0)$. Thus $\nu\notin\ell^2$, and $Dh(I^0)$
is not an $\ell^2$ Fr\'echet derivative. It defines a bounded linear
functional on $\mathcal A_q$ for every $q\ge3$, so the pairing
$\nu\cdot\mathcal D$ converges absolutely for the admissible action increments.

\subsection{Perturbation hypotheses and the main stability theorem}

In the local coordinates of Section~\ref{local-coordinates}, we write the autonomous
perturbed Hamiltonian as
\begin{equation}
 H_\xi(\varphi,p,z)=h(\xi+\pi(\varphi,p,z))
                      +\eps f_\xi(\varphi,p,z),
 \label{setup:perturbed-hamiltonian}
\end{equation}
suppressing the subscript $\xi$ when there is no ambiguity. Any fixed physical
Hamiltonian perturbation expressed in the local charts around $\mathcal K$
defines such a family.

\begin{assumption}[Two-level analyticity and supplied strong trajectories]
\label{ass:contract}
Consider a collection of initial data satisfying $I_A(0)=\xi\in\mathcal K$
and $r_0=\|z(0)\|_\sigma\le r_{\rm adm}$, with arbitrary initial phases and
possibly infinitely supported initial normal actions $J^0$. The collection
is also required to satisfy the trajectory bounds in~(2) below.
We assume that the following properties hold
uniformly over this collection and all sufficiently small $\eps>0$:
\begin{enumerate}
\item \emph{Two-level local analyticity.}
The scalar perturbation $f_\xi(\varphi,p,z)$ is real analytic and extends
holomorphically to fixed complex coordinate neighborhoods with a fixed positive
angular strip width. Its normal derivative is represented with respect to the
unweighted sequence pairing by an analytic map at two adjacent regularity
levels,
\begin{equation}
 D_zf(b)[v]=\sum_{j\in\mathcal N}(\nabla_zf(b))_j\cdot v_j,
 \qquad
 \nabla_zf:h^\sigma\longrightarrow h^{\sigma-1},
 \quad
 \nabla_zf:h^{\sigma+1}\longrightarrow h^\sigma,
 \label{setup:two-level-gradient}
\end{equation}
with compatible restrictions and joint analytic dependence on
$(\varphi,p,z)$. On the fixed smaller coordinate neighborhoods used in the
proof, $f_\xi$, the represented normal gradient maps, and their derivatives of each
fixed finite order are uniformly bounded in the corresponding scalar,
sequence, and multilinear operator norms. At the upper level $h^{\sigma+1}$,
these analytic domains contain a neighborhood of uniform radius around the
local trajectory and its radial Taylor segments on every pre-exit interval
$[0,T]\subseteq[0,T_0/\eps]$. This radius and the corresponding bounds are
independent of $T$.
\item \emph{Supplied strong trajectories and higher normal bound.}
There exist fixed constants $T_0>0$ and $H_+<\infty$, independent of $\eps$,
such that each datum in the collection generates a trajectory $\mathcal Z(t)$
in Cartesian Birkhoff coordinates on $[0,T_0/\eps]$, with $I_A(0)=\xi$ (so $p(0)=0$)
and $z(0)=z^0$ (so $W(0)=r_0^2$), satisfying
\begin{equation}
 \begin{gathered}
 \mathcal Z\in C([0,T_0/\eps];h^\sigma)
       \cap C^1([0,T_0/\eps];h^{\sigma-3}),\\
 \sup_{0\le t\le T_0/\eps}\|z(t)\|_{\sigma+1}\le H_+.
 \end{gathered}
 \label{setup:strong-solution}
\end{equation}
Before exit from the fixed positive-action coordinate tube, the local path
$b(t)=(\varphi(t),p(t),z(t))$ defined by \eqref{setup:polar-map} satisfies
Hamilton's equations, with $(\varphi,p)\in C^1$ and the normal equation
understood in $h^{\sigma-3}$. All local estimates are made before this exit.
For sufficiently small $r_0$ and $\eps$, Theorem~\ref{thm:main} proves that
the trajectory remains in the tube throughout the supplied interval
$[0,T_0/\eps]$.
\end{enumerate}
\end{assumption}

The coordinate condition \eqref{setup:two-level-gradient} holds for a
real-analytic physical Hamiltonian perturbation whose vector field has
compatible analytic mappings $H^N_0\to H^{N-1}_0$ and
$H^{N+1}_0\to H^N_0$ with corresponding uniform local holomorphic bounds.
Apply the compatible differentials of $\mathcal B$ at levels $N-1$ and $N$,
then the inverse differential of the finite polar chart. The latter is
uniformly analytic on the chosen tube because of its positive-action margin.
Canonicity gives
$D_zf(b)[v]=\langle-\mathcal J(X_f)^z(b),v\rangle_{\ell^2}$ for finitely
supported real variations $v$, and the two mapping bounds extend this
unweighted pairing by density to all $v\in h^\sigma$ and $v\in h^{\sigma+1}$,
respectively; see \eqref{app:representedgradient}.

These mapping bounds allow at most one derivative loss; the strong evolution
and higher-norm bound in \eqref{setup:strong-solution} are separate hypotheses.

To state the spectral stability estimate, let $(\lambda_j(w))_{j\ge0}$ be
the periodic eigenvalues of $-\partial_x^2+w$ on $[0,2]$ for a real potential
$w\in H^N_0$, counted with multiplicity. These comprise
the periodic and antiperiodic band edges on $[0,1]$ and satisfy
\cite[arXiv version, pp.~7--8]{KST2013}
\[
 \begin{gathered}
 \lambda_0(w)<\lambda_1(w)\le\lambda_2(w)
   <\lambda_3(w)\le\lambda_4(w)<\cdots,\\
 \gamma(w)=(\gamma_j(w))_{j\ge1},\qquad
 \gamma_j(w)=\lambda_{2j}(w)-\lambda_{2j-1}(w)\ge0,
 \end{gathered}
\]
where $\gamma_j(w)$ is the $j$-th spectral gap length, which vanishes
precisely when the $j$-th gap is closed.

\begin{theorem}[Two-parameter stability near finite-gap tori]
\label{thm:main}
Fix $N\ge3$, $A$, and $\mathcal K$ as above, and suppose a collection of
supplied trajectories satisfies Assumption~\ref{ass:contract}. For each fixed
$0<\delta<1/4$, there exist constants $r_\delta>0$, $\eps_\delta>0$, and
$C_\delta\ge1$ such that every trajectory in the collection with
$r_0=\|z(0)\|_\sigma\le r_\delta$ and $0<\eps\le\eps_\delta$ remains in the
fixed coordinate tube throughout $[0,T_0/\eps]$ and satisfies, for all
$0\le t\le T_0/\eps$,
\begin{equation}\label{setup:near-actions}
 W(t)=|p(t)|_1+2\sum_{j\in\mathcal N}j^{2N+1}J_j(t)
       \le C_\delta\bigl(r_0^2+\eps^{2\delta}\bigr),
\end{equation}
\begin{equation}\label{setup:near-increments}
 |p(t)|_1+\sum_{j\in\mathcal N}j^{2N+1}|J_j(t)-J_j^0|
       \le C_\delta\bigl(r_0^2+\eps^{2\delta}\bigr),
\end{equation}
and
\begin{equation}\label{setup:near-gaps}
 \|\gamma(w(t))-\gamma(w(0))\|_{\ell^2}^2
       \le C_\delta\bigl(r_0^2+\eps^{2\delta}\bigr),
\end{equation}
where $p_a(t)=I_a(t)-I_a(0)$, $J_j^0=I_j(0)$, and
$w(t)=\mathcal B^{-1}\mathcal Z(t)$. Since $C_\delta\ge1$, the spectral
estimate also gives
$\|\gamma(w(t))-\gamma(w(0))\|_{\ell^2}\le C_\delta(r_0+\eps^\delta)$.
The constants $r_\delta,\eps_\delta,C_\delta$ are
uniform over $\xi\in\mathcal K$, all initial phases, and all admissible initial
normal tails, without any nonresonance exclusions; they depend only on
$\delta$, $N,A,\mathcal K$ (and its positive-action margin), the fixed
neighborhood radii and $r_{\rm adm}$, the uniform analytic bounds on $f_\xi$,
and $T_0,H_+$.
\end{theorem}

Theorem~\ref{thm:main} is proved in Section~\ref{subsec:closure-fixed}.
The quantities $r_0$ and $\eps$ are independent small parameters: a fixed sufficiently
small $r_0>0$ is permitted as $\eps\to0$. Since $W(0)=r_0^2$, the initial-size
term in \eqref{setup:near-actions} is necessary on an $\eps$-independent
neighborhood of the torus. 

For shrinking initial neighborhoods, estimates with constant $1$ follow
by applying the theorem with an exponent $\delta_1\in(\delta,1/4)$ and
choosing a sufficiently small constant $c_\delta$ in $r_0\le c_\delta\eps^\delta$:

\begin{corollary}[Exact finite-gap data and shrinking initial neighborhoods]
\label{cor:shrinking-data}
Under Assumption~\ref{ass:contract}, fix $0<\delta<1/4$. There exist
$c_\delta>0$ and $\eps_0>0$ such that whenever $r_0\le c_\delta\eps^\delta$ and
$0<\eps<\eps_0$,
\begin{equation}
 |p(t)|_1\le\eps^{2\delta},\qquad
 \sum_{j\in\mathcal N}j^{2N+1}J_j(t)\le\eps^{2\delta},
 \qquad 0\le t\le T_0/\eps,
 \label{setup:main-actions}
\end{equation}
and
\begin{equation}
 \|\gamma(w(t))-\gamma(w(0))\|_{\ell^2}\le\eps^\delta,
 \qquad 0\le t\le T_0/\eps.
 \label{setup:main-gaps}
\end{equation}
In particular, \eqref{setup:main-actions}--\eqref{setup:main-gaps} hold for
exact finite-gap initial data ($r_0=0$) and for any shrinking neighborhood
satisfying a uniform bound $r_0=O(\eps^\kappa)$ with a fixed $\kappa>\delta$; thus an
$O(\eps^{1/4})$ initial normal neighborhood admits every exponent $0<\delta<1/4$.
\end{corollary}

\begin{proof}
Choose $\delta_1\in(\delta,1/4)$ and apply Theorem~\ref{thm:main} with
exponent $\delta_1$. Choose $c_\delta>0$ so that
$C_{\delta_1}c_\delta^2\le1/4$, and then choose $\eps_0\in(0,\eps_{\delta_1}]$
small enough that $c_\delta\eps_0^\delta\le r_{\delta_1}$ and
$C_{\delta_1}\eps^{2(\delta_1-\delta)}\le1/4$ for all $0<\eps<\eps_0$. Then the
right-hand sides of \eqref{setup:near-actions} and \eqref{setup:near-gaps} are
at most $\eps^{2\delta}/2$, which yields
\eqref{setup:main-actions}--\eqref{setup:main-gaps}. If
$r_0\le C_0\eps^\kappa$ with fixed $C_0>0$ and $\kappa>\delta$, the condition
$r_0\le c_\delta\eps^\delta$ holds after reducing $\eps_0$ so that
$C_0\eps_0^{\kappa-\delta}\le c_\delta$.
\end{proof}

\subsection{Strong chain rules and conserved physical energy}

\begin{lemma}[Exact action equation]
\label{lem:action-chain}
For a trajectory satisfying Assumption~\ref{ass:contract}, the action paths
$\mathcal I(t)$ and $\mathcal D(t)$ are $C^1$ with values in
$\mathcal A_{2\sigma-1}$ (and hence in $\ell^2$) on every interval before
exit from the coordinate tube, and satisfy
\begin{equation}
 \mathcal D'(t)=\mathcal I'(t)=\eps a_f(b(t)),\qquad
 a_F=D\pi X_F
   =\left(-F_\varphi,
          (u_jF_{v_j}-v_jF_{u_j})_{j\in\mathcal N}\right).
 \label{setup:action-chain}
\end{equation}
For a Hamiltonian $F$ with represented normal gradient in $h^{\sigma-1}$,
$D\pi X_F$ denotes the continuous extension given by the displayed component
formula and justified by \eqref{setup:weighted-product} below. The Cartesian
formula remains valid when any or all normal coordinates vanish.
\end{lemma}

\begin{proof}
The tangential equations give $\dot p=-\eps f_\varphi$. By
\eqref{setup:strong-solution}, each normal coordinate $z_j=(u_j,v_j)$ is
continuously differentiable in time. Differentiating $J_j=(u_j^2+v_j^2)/2$
coordinatewise yields $J_j'(t)=\eps(u_jf_{v_j}-v_jf_{u_j})$, since the
unperturbed rotation $\Omega_j(I)\mathcal Jz_j$ is orthogonal to $z_j$ and
contributes zero. For any sequence $\mathbf g\in h^{\sigma-1}$, Cauchy--Schwarz gives
the bilinear estimate
\begin{equation}
 \sum_{j\in\mathcal N}j^{2\sigma-1}
      |z_j\cdot\mathcal J\mathbf g_j|
 \le\|z\|_\sigma\|\mathbf g\|_{\sigma-1}.
 \label{setup:weighted-product}
\end{equation}
Since $z\in C(h^\sigma)$ and $\nabla_zf(b(t))\in C(h^{\sigma-1})$ by
Assumption~\ref{ass:contract}, \eqref{setup:weighted-product} implies that
$t\mapsto a_f(b(t))$ is continuous with values in $\mathcal A_{2\sigma-1}$.
Fix $t_0$ in the interval. Since $z(t_0),z^0\in h^\sigma$, the vectors
$\mathcal I(t_0)$ and $\mathcal D(t_0)$ lie in
$\mathcal A_{2\sigma}\subset\mathcal A_{2\sigma-1}$. The Bochner integral
$G(t):=\eps\int_{t_0}^t a_f(b(s))\,ds$ therefore defines a $C^1$ path in
$\mathcal A_{2\sigma-1}$, and coordinatewise integration of the equations of
motion identifies $G(t)$ with $\mathcal I(t)-\mathcal I(t_0)
=\mathcal D(t)-\mathcal D(t_0)$. The continuous embedding
$\mathcal A_{2\sigma-1}\subset\ell^2$ yields the $C^1(\ell^2)$ assertion.
\end{proof}

\begin{lemma}[Energy identity in the supplied strong class]
\label{lem:energy}
Under Assumption~\ref{ass:contract}, suppose the supplied trajectory remains
in the coordinate tube on an initial interval $[0,T]$. Then the perturbed
Hamiltonian $H_\xi(b(t))$ is conserved on $[0,T]$, and, for $0\le t\le T$,
\begin{equation}\label{setup:initial-energy-identity}
 \nu\cdot\mathcal D(t)+Q_0(\mathcal D(t))
             =\eps\bigl(f(b(0))-f(b(t))\bigr).
\end{equation}
In particular, if $|f|\le F_0$ on the coordinate tube, then
\begin{equation}
 |\nu\cdot\mathcal D(t)+Q_0(\mathcal D(t))|
       \le2F_0\eps,\qquad
 -Q_0(\mathcal D(t))\ge\frac\mu2 B(t)^2.
 \label{setup:energy-coercivity}
\end{equation}
\end{lemma}

\begin{proof}
Fix $t\in[0,T]$, and let $P_R$ denote the projection onto normal indices
$j\le R$. The continuous path $z:[0,t]\to h^\sigma$ has compact image, so
$P_Rz\to z$ uniformly in $h^\sigma$. The full path has compact image in the
open coordinate tube; hence its projections
$b_R(s)=(\varphi(s),p(s),P_Rz(s))$ lie in that tube for all sufficiently
large $R$. Apply the finite-dimensional chain rule to $H_\xi(b_R(s))$ and
integrate over $[0,t]$. 

With $I^0$ and $\nu$ fixed, the decomposition \eqref{setup:concavity} gives
the represented normal gradient
\[
 \nabla_zH_\xi(b)
   =(\nu_jz_j)_{j\in\mathcal N}
     +\bigl((DQ_0(\mathcal D))_jz_j\bigr)_{j\in\mathcal N}
     +\eps\nabla_zf(b),
\]
where $\mathcal D=\xi+\pi(b)-I^0$. The first term belongs to
$h^{\sigma-3}$ because $\nu_j=O(j^3)$.
The second belongs to $h^\sigma$ because $DQ_0$ is locally bounded in
$\ell^2\subset\ell^\infty$, and the third belongs to $h^{\sigma-1}$ by
Assumption~\ref{ass:contract}.

On $[0,t]$, the strong regularity also gives $P_R\dot z\to\dot z$ uniformly
in $h^{\sigma-3}$, while the projected actions converge uniformly in
$\ell^2$. Continuity of $DQ_0$ and $\nabla_zf$ near the compact path,
together with the bound for multiplication by $\nu$, therefore gives
$\nabla_zH_\xi(b_R(s))\to\nabla_zH_\xi(b(s))$ uniformly in
$h^{\sigma-3}$. The tangential derivatives converge uniformly as well.
Because $\sigma-3=N-5/2\ge1/2$, any
two sequences $a,c\in h^{\sigma-3}$ satisfy
\begin{equation}
 \sum_{j\in\mathcal N}|a_j\cdot c_j|
      \le\|a\|_{\sigma-3}\|c\|_{\sigma-3}.
 \label{setup:energy-pairing}
\end{equation}
This estimate justifies convergence of the integrands. The endpoint energies
also converge: the weighted normal sum $\sum_{j\in\mathcal N}j^3J_j$ is
continuous on $h^\sigma$, and the finite tangential and nonlinear terms are
continuous. Passing to the limit $R\to\infty$ in the integrated chain rule
therefore yields
\begin{equation}
 H_\xi(b(t))-H_\xi(b(0))
  =\int_0^t\left(
    H_\varphi\cdot\dot\varphi+H_p\cdot\dot p
      +\sum_{j\in\mathcal N}(\nabla_zH)_j\cdot\dot z_j\right)\,ds=0,
 \label{setup:energy-chain}
\end{equation}
where all Hamiltonian derivatives are evaluated at $b(s)$. The tangential
terms cancel by Hamilton's equations, and each normal summand in the
absolutely convergent series is
$(\nabla_zH)_j\cdot\mathcal J(\nabla_zH)_j=0$.

Subtracting the initial energy and using $\mathcal D(0)=0$ gives
\eqref{setup:initial-energy-identity}, which immediately implies the first
bound in \eqref{setup:energy-coercivity}. For the second bound, note that for
every $\theta\in[0,1]$, the action $I^0+\theta\mathcal D(t)
=(1-\theta)I^0+\theta I(t)$ has nonnegative normal components
$(1-\theta)J_j^0+\theta J_j(t)\ge0$ and tangential components within the
positive-action margin. Thus the entire segment $[I^0,I(t)]$ is physical.
By the choice of the coordinate tube preceding \eqref{setup:concavity}, it
also lies in the uniform concavity neighborhood. Taylor's integral formula
and \eqref{setup:concavity} then yield
\[
 Q_0(\mathcal D)=\int_0^1(1-\theta)
 D^2h_{\rm nl}(I^0+\theta\mathcal D)
       [\mathcal D,\mathcal D]\,d\theta
 \le-\frac\mu2\|\mathcal D\|_{\ell^2}^2,
\]
completing the proof.
\end{proof}

\subsection{A local action-to-gap estimate}

For the chart based at $\xi$, write $w(b)=\mathcal B^{-1}\mathcal Z$, with
$\mathcal Z$ as in \eqref{setup:polar-map}. The following estimate transfers
action bounds to the spectral gaps.

\begin{lemma}[Spectral transfer near a compact finite-gap family]
\label{lem:spectral}
After reducing the fixed coordinate tube, there is a constant $C_\gamma>0$
such that, for every $\xi\in\mathcal K$ and every real point
$b=(\varphi,p,z)$ in the tube,
\begin{equation}
 \|\gamma(w(b))-\gamma(w_\xi)\|_{\ell^2}^2
 \le C_\gamma\left(|p|_1^2+
                    \sum_{j\in\mathcal N}jJ_j\right)
 \le C_\gamma W\quad\hbox{if }W\le1.
 \label{setup:spectral-transfer}
\end{equation}
Here $w_\xi$ is any point of the finite-gap torus with action $\xi$;
its phase need not equal $\varphi$.
\end{lemma}

\begin{proof}
Set $\widehat b=(\varphi,0,0)$ and
$\widehat w=w(\widehat b)$. The Cartesian Birkhoff coordinates of these
base points form a compact subset of $h^{1/2}$, parameterized by
$\mathcal K\times(\R/2\pi\Z)^A$. Their physical image under
$\mathcal B^{-1}$ is compact in $L^2_0$.
The inverse Birkhoff map has bounded first derivative on a common
neighborhood of the compact set of base coordinates in $h^{1/2}$.
Indeed, choose at each base coordinate a ball on which its derivative
is bounded, take a finite cover by smaller balls, and reduce the tube
so that each relevant segment lies in one of the original balls.
The finite polar map also has uniformly bounded first derivative,
since the tangential actions have a common positive lower bound.
The mean-value formula therefore gives
\begin{equation}
 \|w(b)-\widehat w\|_{L^2}
      \le C\bigl(|p|_1+\|z\|_{1/2}\bigr).
 \label{setup:inverse-lipschitz}
\end{equation}

By \cite[Proposition 3.1]{KST2013}, the quotient $I_j/\gamma_j^2$
has a nonvanishing analytic extension across closed gaps and is positive
on real potentials. For $a\in A$, the action $I_a$ is uniformly positive
near the compact physical base family. Thus $\gamma_a^2$ is analytic
there, and its uniformly positive square root $\gamma_a$ has a uniformly
bounded first derivative on a smaller $L^2$ neighborhood. Hence
\begin{equation}
 \sum_{a\in A}|\gamma_a(w(b))-\gamma_a(\widehat w)|^2
       \le C\|w(b)-\widehat w\|_{L^2}^2.
 \label{setup:open-gaps}
\end{equation}
For the normal gaps, the same proposition gives the locally uniform
asymptotic
\[
 8\pi j\,\frac{I_j(w)}{\gamma_j(w)^2}
       =1+j^{-1}r_j(w),
 \qquad (r_j(w))_{j\ge1}\hbox{ locally bounded in }\ell^2.
\]
The quotient is understood through its analytic extension. A finite
cover of the compact physical base family bounds $\|(r_j(w))\|_2$
uniformly on a common $L^2$ neighborhood. Thus
$|j^{-1}r_j(w)|\le1/2$ for every sufficiently large $j$.
The finitely many remaining extended quotients have a common positive
lower bound on a smaller real neighborhood, by positivity and compactness.
Since $I_j=J_j$ for $j\in\mathcal N$, this proves
\begin{equation}
 \gamma_j(w(b))^2\le CjJ_j\qquad(j\in\mathcal N)
 \label{setup:closed-gaps}
\end{equation}
with a constant independent of $j$.

Fixed-modulus Birkhoff fibers are isospectral tori
\cite[Corollary 5.7, with $s=0$ and $p=2$, arXiv:1502.05857v1 numbering]{KMMT2016}.
Consequently $\gamma(\widehat w)=\gamma(w_\xi)$ and
$\gamma_j(\widehat w)=0$ for $j\in\mathcal N$. These identities hold
for every choice of the tangential phases, without any nonresonance
assumption. Together with \eqref{setup:inverse-lipschitz} and
$\|z\|_{1/2}^2=2\sum_{j\in\mathcal N}jJ_j$, the estimates
\eqref{setup:open-gaps} and \eqref{setup:closed-gaps} give the first
inequality in \eqref{setup:spectral-transfer}, after increasing $C_\gamma$ if necessary. The second follows from
$|p|_1^2\le|p|_1$ when $W\le1$ and $j\le j^{2N+1}$.
\end{proof}

\paragraph{Physical distance and the choice of reference torus.}
Let $\mathcal T_\xi$ be the unperturbed finite-gap torus with actions
$\xi$ and zero normal actions. For a real point $b=(\varphi,p,z)$,
compare $w(b)$ with $\widehat w=w(\varphi,0,0)\in\mathcal T_\xi$.
The same finite-cover argument for
$\mathcal B^{-1}:h^\sigma\to H^N_0$ gives, on a smaller common tube,
\begin{equation}\label{setup:torus-distance}
 \operatorname{dist}_{H^N}(w(b),\mathcal T_\xi)
       \le C\bigl(|p|_1+\|z\|_\sigma\bigr).
\end{equation}
The comparison point has the current tangential phases; no tracking
of a particular unperturbed trajectory is used. For trajectories covered
by Theorem~\ref{thm:main}, the distance is at most
$C_\delta(r_0+\eps^\delta)$ throughout $0\le t\le T_0/\eps$.

Suppose instead that $w_0$ is sufficiently close in $H^N$ to a specified
torus $\mathcal T_{\bar\xi}$ with $\bar\xi\in\mathcal K$.
Uniform local Lipschitz continuity of $\mathcal B:H^N_0\to h^\sigma$
gives $r_0\le C\operatorname{dist}_{H^N}(w_0,\mathcal T_{\bar\xi})$
and places $I_A(w_0)$ in a fixed compact enlargement of $\mathcal K$ contained
in $(0,\infty)^A$. If $w_0$ belongs to a collection satisfying
Assumption~\ref{ass:contract} uniformly for this enlarged family,
Theorem~\ref{thm:main} applies with $\xi=I_A(w_0)$ and the corresponding
constants. Relative to $\bar\xi$, the initial tangential offset must be retained:
\[
 \sum_{a\in A}|I_a(t)-\bar\xi_a|
 \le\sum_{a\in A}|I_a(0)-\bar\xi_a|
                        +C_\delta(r_0^2+\eps^{2\delta}).
\]
Closeness in $H^N$ alone does not imply the higher-norm bound in
Assumption~\ref{ass:contract}.  
% END sections/02_setup
% BEGIN sections/03_window
\section{Raw errors and reduction to a finite window}\label{sec:window}

Starting from Assumption~\ref{ass:contract}, we estimate action drift at
weight $2\sigma-1$ and its frozen-frequency pairing. Cubic dispersion
controls the high normal terms that change normal actions, while positivity
and the supplied higher norm transfer the estimates to the weight $2\sigma$ in $W$.
The resulting window identity retains the frozen contribution of the initial normal
actions in the angular jet and provides the error estimates for
Section~\ref{sec:filtered}. All identities hold along the supplied
strong trajectory.

Throughout this section we use the notation
\[
 \mathcal N=\N\setminus A,\qquad
 \mathcal I=(p,J),\qquad
 \mathcal D=I-I^0=(p,J-J^0),\qquad w_*=2\sigma-1.
\]
Unless otherwise indicated, indices in sums involving $z_j$, $J_j$,
or their normal coefficients range over $\mathcal N$.
For a signed action vector $v=(v_A,v_{\mathcal N})\in\mathcal A_{w_*}$,
we use the norm
\begin{equation}\label{rw:actionspace}
 \|v\|_{\mathcal A_{w_*}}
 =|v_A|_1+\sum_{j\in\mathcal N}j^{w_*}|v_j|.
\end{equation}
Since $w_*=2\sigma-1\ge6$, this norm dominates the action Hilbert norm.
The frozen frequency sequence $\nu=Dh(I^0)$ defines a bounded linear
functional on $\mathcal A_{w_*}$: its tangential entries are uniformly
bounded and its normal entries satisfy $|\nu_j|\le Cj^3$.
Implicit constants may depend on the fixed $N,A,\mathcal K$, the
coordinate neighborhoods, analyticity widths and uniform analytic bounds,
and $T_0$.  Dependence on $H_+$ is indicated when it is used.
These constants are independent of $s,\eps$ and of the cutoffs satisfying
the conditions stated below.

We work on a stopped initial interval $[0,T]$, $T\le T_0/\eps$, on which
the supplied trajectory remains in the coordinate tube and
\begin{equation}\label{rw:stopped}
 W=|p|_1+\|z\|_\sigma^2\le s\le s_0\le1.
\end{equation}
Since $p(0)=0$, the bound at $t=0$ gives $W_0=W(0)\le s$.
Thus $B\le W+W_0/2\le2s$.  The same estimates hold on every shorter
initial interval.

The analytic bounds for the perturbation are taken on the fixed smaller
neighborhoods in Assumption~\ref{ass:contract}. Let $P_m$ retain the normal
indices $j\le m$.  Its contraction property on $h^\sigma$ allows $s_0$
to be chosen so that the points $(\varphi,p,P_mz)$ and the Taylor segments
used below lie in the fixed smaller $h^\sigma$ coordinate neighborhood,
uniformly in $m$.

\subsection{The raw Taylor remainder and its scalar pairing}

On \eqref{rw:stopped}, $|p|_1\le s$ and $\|z\|_\sigma\le\sqrt s$.
Accordingly, assign weight two to $p$ and weight one to $z$.  The Taylor
polynomial at $(p,z)=(0,0)$ containing all terms of weighted degree at
most two is
\begin{equation}\label{rw:jet}
 \begin{split}
 g(\varphi,p,z)={}&f(\varphi,0,0)
  +D_pf(\varphi,0,0)[p]+D_zf(\varphi,0,0)[z]\\
 &+\tfrac12D_z^2f(\varphi,0,0)[z,z],\qquad r=f-g.
 \end{split}
\end{equation}
All angular Fourier modes are retained in $g$.  Recall from
\eqref{setup:concavity} that
\[
 h(I)-h(I^0)=\nu\cdot\mathcal D+Q_0(\mathcal D),
 \qquad Q_0(0)=DQ_0(0)=0,\qquad \|D^2Q_0\|\le C_h.
\]
The Hessian bound holds on the fixed action neighborhood.  Since
$\|\mathcal D\|_2=B\le2s$, it gives
$\|DQ_0(\mathcal D)\|_2\le2C_hs$ on \eqref{rw:stopped}.

For a perturbing Hamiltonian $F$ whose normal derivative is represented
in the unweighted pairing by $\nabla_zF\in h^{\sigma-1}$, set
\begin{equation}\label{rw:actiondrift}
 a_F=D\pi X_F
 =\left(-F_\varphi,
   \bigl(u_jF_{v_j}-v_jF_{u_j}\bigr)_{j\in\mathcal N}\right).
\end{equation}
Here $D\pi X_F$ denotes the continuous extension described in
Lemma~\ref{lem:action-chain}, with values in $\mathcal A_{w_*}$.
The normal vector field may lie only in $h^{\sigma-1}$.

\begin{lemma}[Taylor estimates]\label{rw:lem-taylor}
On \eqref{rw:stopped} one has
\begin{equation}\label{rw:taylor}
 \begin{gathered}
 |r|+|r_\varphi|_1\le Cs^{3/2},\qquad
 |r_p|_\infty\le C\sqrt s,\qquad
 \|\nabla_zr\|_{\sigma-1}\le Cs,\\
 \|a_r\|_{\mathcal A_{w_*}}\le Cs^{3/2},\qquad
 |\{r,g\}|\le C\sqrt s,\qquad
 |\{r,Q_0\}|\le Cs^{5/2}.
 \end{gathered}
\end{equation}
The estimates hold with the same constants for finite normal restrictions.
\end{lemma}

\begin{proof}
Suppress $\varphi$ in the notation.  Expand first in $p$ at fixed $z$,
then expand $f(0,z)$ to second order in $z$.  Integral Taylor formulas
bound the three remainders, of orders $p^2$, $pz$, and $z^3$, by
\[
 \tfrac12\sup\|f_{pp}\|\,|p|_1^2,
 \qquad \sup\|f_{pz}\|\,|p|_1\|z\|_\sigma,
 \qquad \tfrac16\sup\|f_{zzz}\|\,\|z\|_\sigma^3.
\]
Their sum is $O(s^{3/2})$ because $s\le1$.  Applying the same
expansions to $f_\varphi$ gives the angular estimate.  Since
$r_p=f_p(p,z)-f_p(0,0)$, the mean-value estimate gives
$|r_p|_\infty\le C(|p|_1+\|z\|_\sigma)\le C\sqrt s$.
Let $A_f=\nabla_zf$, viewed as the analytic map into
$h^{\sigma-1}$ specified in the hypotheses.  Compatibility of represented
gradients and differentiation gives
\[
 \nabla_zg=A_f(0,0)+D_zA_f(0,0)z,
\]
so vector-valued Taylor expansion gives
\[
 \begin{aligned}
 \|\nabla_zr\|_{\sigma-1}
 &\le \sup\|D_pA_f\|\,|p|_1
       +\tfrac12\sup\|D_z^2A_f\|\,\|z\|_\sigma^2\\
 &\le C\bigl(|p|_1+\|z\|_\sigma^2\bigr)=CW\le Cs.
 \end{aligned}
\]
These Taylor estimates use multilinear operator bounds on the fixed
coordinate neighborhoods.

Write $\mathcal J(u,v)=(v,-u)$.  Weighted Cauchy--Schwarz yields
\begin{equation}\label{rw:cs}
 \sum_j j^{2\sigma-1}|z_j\cdot\mathcal J(\nabla_zr)_j|
 \le \|z\|_\sigma\|\nabla_zr\|_{\sigma-1}.
\end{equation}
Together with the angular estimate this proves the action bound.  The
normal series in $\{r,g\}$ is absolutely convergent and bounded by
$\|\nabla_zr\|_0\|\nabla_zg\|_0\le Cs$; its two tangential products
are bounded by $C(s^{3/2}+\sqrt s)$.  Finally, since $Q_0$ depends only
on the actions, one has the absolutely convergent identity
\[
 \{r,Q_0(\mathcal D)\}=-DQ_0(\mathcal D)\cdot a_r.
\]
The bound $\|DQ_0\|_\infty\le2C_hs$ and the action estimate give
$|\{r,Q_0\}|\le Cs^{5/2}$.

For the restriction $f_m(\varphi,p,z)=f(\varphi,p,P_mz)$, the
weighted Taylor jet is $g_m(\varphi,p,z)=g(\varphi,p,P_mz)$.
Thus $r_m=f_m-g_m=r(\varphi,p,P_mz)$ and
\[
 \nabla_zr_m(\varphi,p,z)
 =P_m\nabla_zr(\varphi,p,P_mz).
\]
The projections are contractions at every Sobolev level, and the
projected Taylor segments lie in the same fixed neighborhoods.
The constants are therefore independent of $m$.
\end{proof}

\begin{lemma}[Scalar chain rule at the available regularity]
\label{rw:lem-chain}
Suppose $F$ is continuously differentiable on the $h^\sigma$ coordinate
tube and its normal derivative is represented in the unweighted pairing
by a continuous $h^{\sigma-1}$-valued gradient.  Along the supplied strong
path on $[0,T]$, the function $F(b(t))$ is continuously differentiable and satisfies
\begin{equation}\label{rw:chain}
 \frac{d}{dt}F(b(t))
 =F_\varphi\dot\varphi+F_p\dot p
     +\sum_j(\nabla_zF)_j\cdot\dot z_j.
\end{equation}
The series is absolutely convergent and its sum is continuous in time.
\end{lemma}

\begin{proof}
For $x\in h^{\sigma-1}$ and $y\in h^{\sigma-3}$,
\[
 \sum_j|x_j||y_j|
 \le\|x\|_{\sigma-1}\|y\|_{\sigma-3},
\]
because $2\sigma-4\ge0$.  Fix a compact time interval on which the
path lies in the open coordinate tube.  The projections $P_mz$ converge
uniformly there in $h^\sigma$; hence the projected paths lie in the
tube for all sufficiently large $m$.  For such $m$, apply the ordinary
finite-dimensional chain rule to $F(\varphi(t),p(t),P_mz(t))$ and integrate.
Continuity on the compact image of the path, followed by a finite covering
argument, implies uniform convergence of all the evaluated derivatives.  In particular
\[
 P_m\nabla_zF(\varphi,p,P_mz)
       \longrightarrow\nabla_zF(\varphi,p,z)
       \quad\hbox{uniformly in }h^{\sigma-1}.
\]
The pairing estimate allows passage to the limit in the integrated
normal term against the continuous velocity of the actual path,
$\dot z\in h^{\sigma-3}$.  The endpoint values and tangential terms
also converge uniformly.  The limiting integrand is continuous, which
proves \eqref{rw:chain}.  \end{proof}

\begin{proposition}[The raw error]\label{rw:prop-raw}
Set $R=\eps r$.  The Bochner integral
\[
 E_R(t)=\eps\int_0^t a_r(b(v))\,dv
\]
exists in $\mathcal A_{w_*}$ for $0\le t\le T$ and satisfies
\begin{equation}\label{rw:rawbudget}
 \sup_{t\le T}\|E_R(t)\|_{\mathcal A_{w_*}}\le Cs^{3/2},
 \qquad
 \sup_{t\le T}|\nu\cdot E_R(t)|
 \le C\bigl(\eps s^{3/2}+s^{5/2}+\eps\sqrt s\bigr).
\end{equation}
\end{proposition}

\begin{proof}
Continuity and \eqref{rw:taylor} justify the integral and give
$\|E_R(t)\|_{\mathcal A_{w_*}}\le C\eps t s^{3/2}$, hence the norm
bound.  Apply Lemma~\ref{rw:lem-chain} to $R$ in
$H-h(I^0)=\nu\cdot\mathcal D+Q_0+\eps g+R$.  Since
$\{R,\nu\cdot\mathcal D\}=-\nu\cdot a_R$ and $\{R,R\}=0$,
\begin{equation}\label{rw:rawidentity}
 \nu\cdot E_R(t)
 =-[R(b(t))-R(b(0))]
    +\int_0^t\{R,Q_0+\eps g\}(b(v))\,dv.
\end{equation}
All normal brackets converge absolutely, and the bounded functional
$\nu$ commutes with the Bochner integral.  By \eqref{rw:taylor}, the
endpoint term and the integrals involving $Q_0$ and $\eps g$ are bounded,
respectively, by $C\eps s^{3/2}$, $C\eps T s^{5/2}$, and
$C\eps^2T\sqrt s$.  Using $T\le T_0/\eps$ gives \eqref{rw:rawbudget}.
\end{proof}

\subsection{High spatial divisors and the mixed operator inverse}

Let $g_K$ be the angular Fourier truncation of the jet $g$ to
$|k|_1\le K$, $K\ge1$, and let $P_M$ retain the normal modes
$S_M=\mathcal N\cap[1,M]$.  The part of $g$ that depends on $z$ is
\begin{equation}\label{rw:normaljet}
 g^{\rm n}(\varphi,z)
   =\langle a(\varphi),z\rangle
        +\tfrac12\langle C(\varphi)z,z\rangle.
\end{equation}
Here $a=\nabla_zf(\varphi,0,0)$ and
$C=D_z\nabla_zf(\varphi,0,0)$; for fixed $\xi$, both depend only on
$\varphi$.  The Hessian $C$ is symmetric in the unweighted pairing.
By Assumption~\ref{ass:contract}, the normal Hamiltonian matrix
$B_f=\mathcal JC$ satisfies
\[
 B_f:h^\sigma\longrightarrow h^{\sigma-1},\qquad
 B_f:h^{\sigma+1}\longrightarrow h^\sigma.
\]
For a normal linear or quadratic monomial
$e^{ik\varphi}\zeta^\alpha\bar\zeta^\beta$, with finitely supported
nonnegative multi-indices $\alpha,\beta$, set
\[
 q=(k,\alpha-\beta),\qquad
 d_q(I)=k\cdot\omega(I)+(\alpha-\beta)\cdot\Omega(I).
\]
Here $\omega(I)=\partial_{I_A}h(I)$ and
$\Omega_j(I)=c_0j^3+\beta_j(I)$, $c_0=8\pi^3$.
Choose uniform bounds
\begin{equation}\label{rw:freqbounds}
 \|\omega\|_\infty\le C_\omega,
 \quad\|\beta\|_\infty\le C_\beta,
 \quad\|D_I\omega\|+\|D_I\beta\|\le L_\nu,
\end{equation}
where the derivative norms are operator norms from the full action
space $\ell^2$ to $\ell^\infty(A)$ and $\ell^\infty(\mathcal N)$,
respectively.  These bounds hold on the fixed complex action
neighborhood by \eqref{setup:frequencies} and the analytic bounds
following it.

Let $g^{\rm h}=g^{\rm h}_{K,M}$ consist of the following terms of $g_K$:
normal linear monomials with index greater than $M$, and normal
quadratic monomials with $\alpha-\beta\ne0$ and at least one
active normal index greater than $M$.  Conjugate terms are selected
together.  Thus every $e^{ik\varphi}J_j$, including those with $k\ne0$,
is excluded from $g^{\rm h}$.

\begin{lemma}[Uniform separation]\label{rw:lem-separation}
For a fixed $D\ge1$, sufficiently large in terms of
\eqref{rw:freqbounds}, let $K,M\ge1$ satisfy $M^2\ge D(K+1)$.
Then, for every $k\in\mathbb Z^A$ with $|k|_1\le K$ and normal
indices $j,\ell\in\mathcal N$, uniformly for $I$ in the fixed action
neighborhood, including its complex points, one has
\begin{equation}\label{rw:divisors}
 \begin{aligned}
 |k\cdot\omega\pm\Omega_j|&\ge(c_0/2)j^3
       &&(j>M),\\
 |k\cdot\omega\pm(\Omega_j-\Omega_\ell)|
       &\ge(c_0/2)|j^3-\ell^3|
       &&(j\ne\ell,\ \max(j,\ell)>M),\\
 |k\cdot\omega\pm(\Omega_j+\Omega_\ell)|
       &\ge(c_0/2)(j^3+\ell^3)
       &&(\max(j,\ell)>M).
 \end{aligned}
\end{equation}
The last case permits $j=\ell$.  For each such label,
\begin{equation}\label{rw:divderivative}
 |D_Id_q(I)[v]|\le L_\nu(|k|_1+2)\|v\|_2.
\end{equation}
\end{lemma}

\begin{proof}
The bounded correction to the indicated cubic leading term has size
at most $KC_\omega+2C_\beta$.  If $m=\max(j,\ell)$ and $j\ne\ell$,
then $|j^3-\ell^3|\ge m^2$; also $j^3+\ell^3\ge m^2$ and
$j^3\ge j^2$.  Taking
$D\ge\max\{1,2(C_\omega+2C_\beta)/c_0\}$ makes the correction
at most half the leading magnitude.  Differentiating the divisor
gives \eqref{rw:divderivative}.  
\end{proof}

Let $c_q=c_q(\xi)$ be the coefficient of a selected monomial in
$g^{\rm h}$.  With $\xi$ fixed and $I$ treated as an independent
parameter, divide by the full divisor to define
\begin{equation}\label{rw:primitive}
 \chi(\varphi,I,z)
   =\sum_{q\ \text{selected}}\frac{c_q}{i d_q(I)}
                     e^{ik\varphi}\zeta^\alpha\bar\zeta^\beta.
\end{equation}
The infinite normal sum is defined by the operator construction below.
Write its normal Hamiltonian vector field at fixed $I$ as
\begin{equation}\label{rw:frozenfield}
 V(\varphi,I,z)=b_*(\varphi,I)+S_*(\varphi,I)z.
\end{equation}
After substituting $I=\xi+\pi(b)$, we also write $\chi$ for the
resulting function of $b$.  Its full normal vector field includes
derivatives through $I$.  The composed primitive satisfies the exact
homological identity
\begin{equation}\label{rw:homological}
 \{\chi,h\}=g^{\rm h},
\end{equation}
because $\{I_j,h\}=0$ for every action coordinate.

\begin{lemma}[Smoothing of the primitive]\label{rw:lem-mixed}
Let $K,M\ge1$ satisfy $M^2\ge D(K+1)$, with $D$ as in
Lemma~\ref{rw:lem-separation}.  On a fixed smaller angular strip and
the fixed complex action neighborhood, uniformly in $K,M$,
\begin{equation}\label{rw:smoothing}
 \|b_*\|_{\sigma+2}
  +\|S_*\|_{\sigma\to\sigma+1}
  +\|S_*\|_{\sigma+1\to\sigma+2}
  +\|S_*\|_{\sigma-1\to\sigma}\le C.
\end{equation}
The same bounds hold after up to two angular derivatives, alone or
combined with one action derivative $D_I[v]$; in the latter case the
right-hand side is $C\|v\|_2$.  These derivatives are continuous in
the displayed sequence and operator norms.  The primitive is real
on the real phase space.
\end{lemma}

\begin{proof}
The uniform analytic bounds give Banach-valued Fourier estimates,
for some fixed $\rho>0$,
\begin{equation}\label{rw:fouriercoeff}
 \|a_k\|_{\sigma-1}
  +\|(B_f)_k\|_{\sigma\to\sigma-1}
  +\|(B_f)_k\|_{\sigma+1\to\sigma}
       \le Ce^{-\rho|k|_1}.
\end{equation}
Division of a linear coefficient by a divisor bounded below by $cj^3$
gains three derivatives.  For the quadratic part use doubled complex
coordinates $(\zeta_j,\bar\zeta_j)$ and let $\Lambda$ multiply both
signed coordinates at index $j$ by $j$.

We record an operator estimate that does not require summability of the normal
coefficients.  If $A_0$ is bounded on $\ell^2$ and
$\sup_i\sum_j|m_{ij}|^2\le C_m^2$, then on finite sequences
\begin{equation}\label{rw:schur}
 \begin{split}
 \sum_i\left|\sum_jm_{ij}(A_0)_{ij}x_j\right|^2
 &\le C_m^2\sum_j|x_j|^2\|A_0e_j\|^2\\
 &\le C_m^2\|A_0\|^2\|x\|^2.
 \end{split}
\end{equation}
The first inequality is rowwise Cauchy--Schwarz.  Truncating the output indices and then applying monotone convergence
and density extends the estimate to all $\ell^2$ inputs.

For $r=\sigma,\sigma+1$, set
$A_0=\Lambda^{r-1}(B_f)_k\Lambda^{-r}$.
The normalized divided matrix
$\Lambda^{r+1}(S_*)_k\Lambda^{-r}$ is its entrywise product with a
multiplier whose magnitude, with output index $j$ and input index
$\ell$, is bounded in the difference and sum sectors by
\[
 C\frac{j^2}{|j^3-\ell^3|}\quad(j\ne\ell),
 \qquad C\frac{j^2}{j^3+\ell^3},
\]
respectively.  The factors $k\cdot\omega$ have already been accounted
for in \eqref{rw:divisors}.  The first expression is at most
$C/|j-\ell|$; the second is at most $C/\ell$, since
$j^3+\ell^3\ge j^2\ell$.  Both squared row sums are uniformly bounded
by multiples of $\sum_{n\ge1}n^{-2}$.  Entries excluded by the selection
have multiplier zero and cannot increase these sums.

Before division, the selected matrix is also bounded at both original
regularity levels, with a loss of one derivative.  Let $(B_f)_{K,\mathrm d}$ be the diagonal part of
$(B_f)_K$ in doubled coordinates, with $P_M$ acting on both signs.
Then the selected matrix is
\[
 (B_f)_{K,M}^{\rm h}
   =(B_f)_K-P_M(B_f)_KP_M-(1-P_M)(B_f)_{K,\mathrm d}.
\]
Coordinate projection and diagonal extraction are contractions in the normalized operator norms for these mappings.  Its norm is therefore at most
$3\|(B_f)_K\|$ at either regularity level.  Opposite signs at the same
physical index belong to the divided sum sector.  Identical signed
indices form the action diagonal, corresponding to the
$e^{ik\varphi}J_j$ terms excluded from $g^{\rm h}$.

An angular derivative contributes a factor of $|k|_1$.  An action
derivative $D_I[v]$ of the reciprocal multiplies the above bounds by
at most $C(|k|_1+2)\|v\|_2$, by \eqref{rw:divderivative} and the
uniform lower bound on each selected divisor.  All polynomial factors
from the stated mixed derivatives are summable against
\eqref{rw:fouriercoeff} on a fixed narrower strip, independently of $K$.
To verify Fr\'echet differentiation in operator norm, take a small
action increment $v$ with $[I,I+v]$ inside the fixed complex action
neighborhood, and set $d=d_q(I)$ and
$\delta d=d_q(I+v)-d_q(I)$.  The uniform second action derivative
bounds for the frequencies give
\[
 |\delta d-Dd_q(I)[v]|
 \le C(|k|_1+2)\|v\|_2^2.
\]
Consequently the actual derivative remainder is
\[
 \frac1{d+\delta d}-\frac1d+\frac{Dd_q(I)[v]}{d^2}
 =-\frac{\delta d-Dd_q(I)[v]}{d^2}
   +\frac{(\delta d)^2}{d^2(d+\delta d)}.
\]
By \eqref{rw:divisors} and $M^2\ge D(K+1)$, the increment can be
chosen small independently of $K,M$ so that $|\delta d|\le |d|/2$.
The preceding row estimates apply to both terms, with an additional factor
$\|v\|_2^2$ in the resulting operator bound; the polynomial Fourier factors
are absorbed as above.
Applying the same estimates to differences of the differentiated
reciprocal proves continuity of the derivative.  The linear
coefficients satisfy the analogous componentwise estimates.
This proves the mixed derivative bounds at the two upper levels.

At real parameters, conjugation replaces $q,d_q$ by $-q,-d_q$, so the
selected primitive is real.  Its frozen quadratic Hessian is symmetric
for the complex bilinear unweighted pairing in the Cartesian variables
$z=(u,v)$, including at complex parameters.  Denoting the transpose for
this pairing by $^{\mathsf T}$,
we therefore have
\[
 S_*^{\mathsf T}=\mathcal J S_*\mathcal J.
\]
 By duality with respect to this bilinear pairing, the bounds at
$r=\sigma,\sigma+1$ also hold at $r=-\sigma-1,-\sigma-2$, with a gain of one
derivative.  Hilbert-scale
interpolation between the endpoint bounds supplies all intermediate
levels.  In particular, interpolate from
\[
 S_*:h^{-\sigma-1}\longrightarrow h^{-\sigma}
 \qquad\hbox{to}\qquad
 S_*:h^\sigma\longrightarrow h^{\sigma+1}
\]
with $\theta=2\sigma/(2\sigma+1)$.  The input index is
$(1-\theta)(-\sigma-1)+\theta\sigma=\sigma-1$, and the output
index is $\sigma$.  Thus $S_*:h^{\sigma-1}\to h^\sigma$, as required.
Angular and action derivatives satisfy the same transpose identity
by differentiation, so the same duality and interpolation argument
applies to them.  This proves
\eqref{rw:smoothing}.

For fixed $(\varphi,I)$, the construction defines the continuous scalar
polynomial
\[
 \chi=\langle-\mathcal Jb_*,z\rangle
       +\tfrac12\langle-\mathcal JS_*z,z\rangle,
\]
whose frozen gradient is $-\mathcal J(b_*+S_*z)$.
Normal compression gives $P_mb_*$ and $P_mS_*P_m$.
The vectors converge in $h^{\sigma+2}$, and the operators converge
strongly at the stated levels.  The same holds after the indicated
derivatives, with the action direction fixed.  Uniform boundedness
makes the operator convergence uniform on compact input sets.
\end{proof}

\subsection{An action primitive on the actual trajectory}

We combine the tangential and all normal action components of $\chi$
in the vector-valued primitive
\begin{equation}\label{rw:wholeprimitive}
 \mathcal U(\varphi,I,z)
  =\left(-\partial_\varphi\chi,
      \bigl(z_j\cdot V_j\bigr)_{j\in\mathcal N}\right).
\end{equation}
After substituting $I=\xi+\pi(b)$, it equals $a_\chi$.
Indeed, $\{p_a,I_\ell\}=\{J_j,I_\ell\}=0$ for every action coordinate,
so the derivatives of the coefficients through $I$ contribute zero to
these action brackets.

Partial derivatives of $\mathcal U$ are taken with
$(\varphi,I,z)$ independent.
After composition, the dependence on $p$ is carried by $I$.
The partial derivative $D_z\mathcal U$, initially
defined on $h^\sigma$, is used on $h^{\sigma-1}$ through the continuous
extension proved below.  The notation $\mathcal U(b)$ means evaluation
at $(\varphi,\xi+\pi(b),z)$.

\begin{lemma}[Action primitive]\label{rw:lem-primitive}
Let $K,M\ge1$ satisfy $M^2\ge D(K+1)$, with $D$ as in
Lemma~\ref{rw:lem-separation}.  At real parameters in the fixed
neighborhoods, put $Z=\|z\|_\sigma$.  The primitive and its derivatives
satisfy
\begin{equation}\label{rw:ubounds}
 \begin{aligned}
 \|\mathcal U\|_{\mathcal A_{w_*}}
 +\|D_\varphi\mathcal U\|_{\mathcal L(\R^A,\mathcal A_{w_*})}
       &\le C(Z+Z^2),\\
 \|D_I\mathcal U\|_{\ell^2\to\mathcal A_{w_*}}
       &\le C(Z+Z^2),\\
 \|D_z\mathcal U[v]\|_{\mathcal A_{w_*}}
       &\le C(1+Z)\|v\|_{\sigma-1}.
 \end{aligned}
\end{equation}
Set
\begin{equation}\label{rw:uresidual}
 \mathcal R
   =D_\varphi\mathcal U[f_p]
       +D_z\mathcal U[X_f^z]+D_I\mathcal U[a_f].
\end{equation}
Here the partial derivatives of $\mathcal U$ are evaluated at
$(\varphi,\xi+\pi(b),z)$, and $f_p,X_f^z,a_f$ at $b$.
On a stopped interval $[0,T]$ of the supplied strong trajectory,
where \eqref{rw:stopped} holds, one has
$\|\mathcal R\|_{\mathcal A_{w_*}}\le C$
and, for $0\le t\le T$,
\begin{equation}\label{rw:highidentity}
 E_{\rm h}(t):=\eps\int_0^t a_{g^{\rm h}}(b(v))\,dv
   =\eps[\mathcal U(b(t))-\mathcal U(b(0))]
         -\eps^2\int_0^t\mathcal R(b(v))\,dv.
\end{equation}
Both integrals are Bochner integrals in $\mathcal A_{w_*}$.
In particular, if $T\le T_0/\eps$, then
\begin{equation}\label{rw:highbudget}
 \sup_{t\le T}\bigl(\|E_{\rm h}(t)\|_{\mathcal A_{w_*}}
                 +|\nu\cdot E_{\rm h}(t)|\bigr)\le C\eps.
\end{equation}
Let $g_{1,\rm h}$ be the normal linear part of $g_K$ with index
$j>M$.  On the same stopped interval with $T\le T_0/\eps$,
\begin{equation}\label{rw:highlinear}
 \sup_{t\le T}\left(
  \left\|\eps\int_0^t a_{g_{1,\rm h}}\,dv\right\|_{\mathcal A_{w_*}}
    +\left|\nu\cdot\eps\int_0^t a_{g_{1,\rm h}}\,dv\right|
                         \right)\le C\eps.
\end{equation}
The integrands are evaluated at $b(v)$, and the normal sum includes
every index greater than $M$.
\end{lemma}

\begin{proof}
For scalar or real-pair sequences,
\[
 \sum_jj^{w_*}|x_j||y_j|
       \le\|x\|_r\|y\|_{w_*-r}.
\]
Apply this to $z$ and $V=b_*+S_*z$, using
Lemma~\ref{rw:lem-mixed}.  It also proves the bounds for the normal
components of the parameter derivatives.  At fixed $I$,
\[
 D_z(z_j\cdot V_j)[v]
   =v_j\cdot(b_*+S_*z)_j+z_j\cdot(S_*v)_j.
\]
For the first product use $v\in h^{\sigma-1}$ and
$V\in h^{\sigma+1}\subset h^\sigma$.  For the second use
$z\in h^\sigma$ and $S_*v\in h^\sigma\subset h^{\sigma-1}$.
For the tangential components use the scalar linear and quadratic
forms for $\chi$ and its angular derivatives.  Their frozen normal
gradients are $-\mathcal JV$ and its corresponding derivatives, in
$h^{\sigma+1}$, so their pairing with $v\in h^{\sigma-1}$ is
continuous.  These scalar forms are bounded in absolute value by $C(Z+Z^2)$.  The bounds in
Lemma~\ref{rw:lem-mixed}, including two angular derivatives and the
mixed angular/action derivatives, cover all these terms.
This proves \eqref{rw:ubounds}.  Density of $h^\sigma$ in
$h^{\sigma-1}$ gives the unique continuous extension of
$D_z\mathcal U$ asserted above.

The perturbation bounds in Assumption~\ref{ass:contract} give
\[
 \|a_f\|_{\mathcal A_{w_*}}
     \le |f_\varphi|_1+\|z\|_\sigma\|\nabla_zf\|_{\sigma-1}
     \le C.
\]
Together with $|f_p|\le C$, $\|X_f^z\|_{\sigma-1}\le C$, and the
embedding $\mathcal A_{w_*}\subset\ell^2$, this proves the residual
bound, since $Z\le\sqrt s\le1$.

To handle the three-derivative loss in the integrable velocity, we
first differentiate finite normal compressions along the supplied path.
Lemma~\ref{lem:action-chain} gives $I\in C^1(\ell^2)$ and
$\dot I=\eps a_f$.  For $m\ge1$, set
\[
 \chi_m(\varphi,I,z)=\chi(\varphi,I,P_mz),\qquad
 g_m^{\rm h}(\varphi,z)=g^{\rm h}(\varphi,P_mz).
\]
Only the explicit normal variables are compressed; every denominator
retains the full parameter $I(t)$.  The frozen normal field of $\chi_m$ is
\[
 V_m=P_mb_*+P_mS_*P_mz.
\]
Define $\mathcal U_m$ by \eqref{rw:wholeprimitive} with $\chi_m,V_m$
in place of $\chi,V$.  Its normal components vanish for $j>m$.
Define $\mathcal R_m$ by \eqref{rw:uresidual} with $\mathcal U_m$
in place of $\mathcal U$ and normal input $P_mX_f^z$; all perturbation
quantities are evaluated at the full point $b(t)$.
The chain rule in the finite explicit coordinates and the
$C^1(\ell^2)$ parameter $I(t)$ gives
\[
 \frac{d}{dt}\mathcal U_m
       =a_{g^{\rm h}_m}+\eps\mathcal R_m.
\]
Indeed, the unperturbed angular and normal parts of the derivative
multiply each monomial by $id_q(I)$, and \eqref{rw:homological}
cancels that divisor exactly.  The action-parameter derivative is
multiplied by the full $\dot I=\eps a_f$.

Integrate this finite identity before passing to $m\to\infty$.
The integrand $a_{g^{\rm h}_m}$ converges to $a_{g^{\rm h}}$ in
$\mathcal A_{w_*}$ by the jet bounds with a loss of one derivative and weighted
Cauchy--Schwarz.  Endpoints and residuals
converge strongly in the same space using \eqref{rw:smoothing},
\eqref{rw:ubounds}, and the continuity of the operator evaluations.
To see the uniformity along the actual path, set
$S(t)=S_*(\varphi(t),I(t))$ and $v(t)=X_f^z(b(t))$.  Then
\[
 P_mS(t)P_mv(t)-S(t)v(t)
 =P_mS(t)\bigl(P_mv(t)-v(t)\bigr)
+(P_m-1)S(t)v(t).
\]
The first term tends uniformly to zero in $h^\sigma$ by the uniform
$h^{\sigma-1}\to h^\sigma$ bound and the compact image of the
continuous path $v$ in $h^{\sigma-1}$.  The second does so because
$t\mapsto S(t)v(t)$ is continuous into $h^\sigma$ and has compact
image there.  For the action derivative terms, the operator path
$S_I(t)=D_IS_*(\varphi(t),I(t))[a_f(b(t))]$ is continuous in
$\mathcal L(h^\sigma,h^{\sigma+1})$.  Apply the same decomposition
to $S_I(t)z(t)$; the linear coefficients and angular derivatives
are handled similarly.  The limiting integrands are continuous in
$\mathcal A_{w_*}$, and the uniform bounds justify passage through
the Bochner integrals.  This gives \eqref{rw:highidentity}.

The endpoint contribution is at most $C\eps(\sqrt s+s)$, and the
residual contribution is at most $C\eps^2T$.  For $T\le T_0/\eps$,
the bound $s\le1$ gives a total contribution of at most $C\eps$.
Pairing with the bounded functional $\nu$ on $\mathcal A_{w_*}$ gives
\eqref{rw:highbudget}.

For \eqref{rw:highlinear}, use the linear part of \eqref{rw:primitive}.
Its frozen field is
$b_*$ and its coefficient gains three derivatives by
\eqref{rw:divisors}, so $\|b_*\|_{\sigma+2}\le C$, with the same
angular and action derivative bounds.  In
\eqref{rw:wholeprimitive} the normal component is $z_j\cdot b_{*,j}$
and the tangential component is the derivative of a continuous linear
form.  Consequently the first two bounds in \eqref{rw:ubounds} improve
to $CZ$, and the last to $C\|v\|_{\sigma-1}$.  Compress its explicit
normal indices while keeping the full $I(t)$ in the denominators.
The proof of \eqref{rw:highidentity}, now with $S_*=0$ and the full
$f$ in the residual, gives an endpoint contribution of size
$C\eps\sqrt s$ and a residual contribution of size $C\eps^2T$.
Passing to all normal indices is justified by the
$h^{\sigma+2}$ coefficient convergence and the same weighted
pairing estimates.  Use $s\le1$, $T\le T_0/\eps$, and the bounded
pairing with $\nu$ to obtain \eqref{rw:highlinear}.
\end{proof}

\subsection{Jet leakage and transfer to the highest action weight}

We first estimate the integrated jet drift outside the normal cutoff,
then use the supplied higher norm to transfer this control from weight
$2\sigma-1$ to weight $2\sigma$.
For a finite set $S\subset\mathcal N$ and $0\le t\le T$, define
\begin{equation}\label{rw:leakdefinition}
 \begin{split}
 G_j(t)&=\eps\int_0^t\{J_j,g\}(b(v))\,dv,\\
 \Gamma_S(t)&=\sup_{0\le v\le t}\sum_{j\notin S}
                                  j^{2\sigma-1}|G_j(v)|,\\
 K_S&=\left(d+4\sum_{j\in S}j^{4\sigma}\right)^{1/2},\qquad
 \mathcal B_t=\sup_{v\le t}B(v),\quad
 \mathcal W_t=\sup_{v\le t}W(v).
 \end{split}
\end{equation}
Here $G_j$ is the contribution of the full Taylor jet $g$ to the normal
action increment.  The sums over $j\notin S$ run over
$\mathcal N\setminus S$, and all running suprema are over $[0,t]$.
In $\Gamma_S$, absolute values are taken after time integration,
and the supremum is taken after the weighted sum.

\begin{proposition}[Exterior leakage]\label{rw:prop-leakage}
There are constants $a,C>0$ such that, for every $K,M\ge1$
satisfying $M^2\ge D(K+1)$ and every stopped interval
$[0,T]$ with $T\le T_0/\eps$,
\begin{equation}\label{rw:leakage}
 \Gamma_{S_M}(T)\le C(\eps+e^{-aK}).
\end{equation}
\end{proposition}

\begin{proof}
Split $g=g_K+(g-g_K)$.  For every normal index $j>M$,
$\{J_j,g_K\}=\{J_j,g^{\rm h}\}$: the terms outside $g^{\rm h}$
are independent of $z_j$, apart from the normal action diagonal, whose normal
action drift vanishes even when its coefficients depend on $\varphi$.
Restricting \eqref{rw:highbudget} to these components
bounds the integrated contribution of $g_K$ by $C\eps$.
The Fourier tail satisfies
\[
 \|\nabla_z(g-g_K)\|_{\sigma-1}
          \le Ce^{-aK}(1+\|z\|_\sigma)
\]
on a fixed narrower strip, by \eqref{rw:fouriercoeff}.  The polynomial
number of angular lattice points and fixed angular derivative factors
are absorbed by reducing the positive exponent $a$.  Weighted
Cauchy--Schwarz and $\|z\|_\sigma\le\sqrt s$ bound its integrated
exterior action drift by $C\eps T e^{-aK}(\sqrt s+s)$.
Since $\eps T\le T_0$ and $s\le1$, this contribution is at most
$Ce^{-aK}$.  Combining the two contributions proves
\eqref{rw:leakage}, with constants independent of $K,M$.
\end{proof}

\begin{proposition}[High-weight transfer with an initial normal tail]
\label{prop:transfer}
On a supplied stopped interval $[0,T]$, assume the uniform bound
\[
 \sup_{t\le T}\|z(t)\|_{\sigma+1}\le H_+,
 \qquad W_0=\|z(0)\|_\sigma^2.
\]
Here $W_0=W(0)$ since $p(0)=0$.  Choose $C_z$ so that
$\|\nabla_zr\|_{\sigma-1}\le C_zW$ on the fixed small tube,
as in the proof of Lemma~\ref{rw:lem-taylor}.
For every finite $S\subset\mathcal N$ and $0\le t\le T$,
\begin{equation}\label{rw:transferexact}
 \mathcal W_t\le e^{(4\sqrt2/3)H_+C_z\eps t}
 \left[2^{1/3}\bigl(W_0+K_S\mathcal B_t\bigr)
                  +2H_+^{2/3}\Gamma_S(t)^{2/3}\right].
\end{equation}
In particular, if $T\le T_0/\eps$ and $M^2\ge D(K+1)$, then
\begin{equation}\label{rw:transfer}
 \sup_{t\le T}W(t)\le C_{\rm tr}
 \left[W_0+K_{S_M}\sup_{t\le T}B(t)
                  +(\eps+e^{-aK})^{2/3}\right],
\end{equation}
where $C_{\rm tr}$ may depend on $H_+,C_z,T_0$ and the fixed
analytic data, and is independent of the cutoffs and the admissible
initial tail.  \end{proposition}

\begin{proof}
Define the positive excess $J_j^+(v)=(J_j(v)-J_j^0)_+$.
We estimate it through the integrated action equation.  Lemma~\ref{lem:action-chain}
and $f=g+r$ give
\[
 J_j(v)-J_j^0
       =G_j(v)+\eps\int_0^v\{J_j,r\}(b(u))\,du.
\]
Take positive parts and use the triangle inequality and
\eqref{rw:cs}, with $\|z(u)\|_\sigma\le W(u)^{1/2}$.
Passing to the limit in the nonnegative partial sums gives
\begin{equation}\label{rw:loweraction}
 L_S(v):=\sum_{j\notin S}j^{2\sigma-1}J_j^+(v)
 \le\Gamma_S(v)+C_z\eps\int_0^vW(u)^{3/2}\,du.
\end{equation}
Since $0\le J_j^+(v)\le J_j(v)$, its upper moment obeys
$\sum_jj^{2\sigma+2}J_j^+(v)\le H_+^2/2$.
H\"older's inequality and
$2\sigma=\frac23(2\sigma-1)+\frac13(2\sigma+2)$ therefore give
\[
 2\sum_{j\notin S}j^{2\sigma}J_j^+(v)
       \le2^{2/3}H_+^{2/3}L_S(v)^{2/3}.
\]
For the finite part, Cauchy--Schwarz gives
\[
 |p(v)|_1+2\sum_{j\in S}j^{2\sigma}J_j^+(v)\le K_SB(v),
\]
because replacing a signed component of $J-J^0$ by its positive
part cannot increase the $\ell^2$ norm.  Since
$J_j(v)\le J_j^0+J_j^+(v)$, it follows that
\[
 W(v)\le W_0+K_SB(v)+2\sum_{j\notin S}j^{2\sigma}J_j^+(v).
\]
Fix $t\le T$, put $A_t=W_0+K_S\mathcal B_t$, and set
$Y(v)=W(v)^{3/2}$.  Using
$(a+b)^{3/2}\le\sqrt2(a^{3/2}+b^{3/2})$ yields, for $v\le t$,
\[
 Y(v)\le\sqrt2A_t^{3/2}+2\sqrt2H_+\Gamma_S(t)
              +2\sqrt2H_+C_z\eps\int_0^vY(u)\,du.
\]
Apply integral Gronwall on $[0,t]$, raise the resulting bound to the
power $2/3$, and use $(a+b)^{2/3}\le a^{2/3}+b^{2/3}$ for $a,b\ge0$.
Taking the supremum over $0\le v\le t$ gives
\eqref{rw:transferexact}.  If $T\le T_0/\eps$,
Proposition~\ref{rw:prop-leakage} then gives \eqref{rw:transfer}.
\end{proof}

\subsection{The high normal diagonal and the complete window error}

We now treat the high normal action diagonal excluded from $g^{\rm h}$,
separate the contribution of the initial normal actions $J^0$, and assemble the
window error.  We
then rewrite the retained jet in the form used for averaging in
Section~\ref{sec:filtered}.  All integrands below are evaluated along the supplied trajectory $b(v)$.

Define the jet with both angular and normal cutoffs by
\[
 g_{K,M}^{\ell}(\varphi,p,z)=g_K(\varphi,p,P_Mz).
\]
The remaining terms of $g_K$ are $g^{\rm h}$ and the high normal
action diagonal:
\begin{equation}\label{rw:highdecomposition}
 g_K=g_{K,M}^{\ell}+g^{\rm h}+d_{\rm h},\qquad
 d_{\rm h}(\varphi,z)=\sum_{j>M}a_j^{\rm d}(\varphi)J_j.
\end{equation}
The coefficients $a_j^{\rm d}$ are real for real angles.  Diagonal extraction, the one-derivative bound on $(B_f)_K$, and angular
analyticity give
\begin{equation}\label{rw:diagonalbounds}
 |a_j^{\rm d}|+|\partial_\varphi a_j^{\rm d}|
       \le Cj,
 \qquad \|\nabla_zd_{\rm h}\|_{\sigma-1}\le C\|z\|_\sigma.
\end{equation}
Indeed, in signed complex coordinates, each diagonal entry of
$\Lambda^{\sigma-1}(B_f)_K\Lambda^{-\sigma}$ is bounded by the
operator norm; undoing the weights gives the bound $Cj$.
Every fixed number of further angular derivatives satisfies the same
estimate on a smaller fixed strip.  Since
$\sum_{j>M}jJ_j\le\|z\|_\sigma^2/2$, the scalar series and its
angular derivatives converge absolutely.  It defines a $C^1$ scalar function on the $h^\sigma$ tube, with continuous
represented gradient in $h^{\sigma-1}$.  The diagonal preserves every normal action, but its
tangential drift $a_{d_{\rm h}}=(-\partial_\varphi d_{\rm h},0)$
may be nonzero.

Separate the contributions of the initial normal actions and their increments:
\begin{equation}\label{rw:frozen-diagonal}
 \begin{split}
 c_{\rm h}^0(\varphi)&=\sum_{j>M}a_j^{\rm d}(\varphi)J_j^0,\\
 d_{\rm h}^{\rm inc}(\varphi,z)
       &=\sum_{j>M}a_j^{\rm d}(\varphi)(J_j-J_j^0),\qquad
 d_{\rm h}=c_{\rm h}^0+d_{\rm h}^{\rm inc}.
 \end{split}
\end{equation}
The values $J^0$ are fixed parameters when differentiating in the
running action $I$.  The Fourier support of $c_{\rm h}^0$ lies in
$|k|_1\le K$.  On a fixed smaller strip of width $a_1>0$, its
coefficients satisfy
\begin{equation}\label{rw:frozen-diagonal-bound}
 \sum_{|k|_1\le K}e^{a_1|k|_1}|(c_{\rm h}^0)_k|
       \le C\sum_{j>M}jJ_j^0\le CW_0,
\end{equation}
with the same bound for each fixed number of angular derivatives
after reserving strip width.  Indeed the diagonal coefficients have
$|(a_j^{\rm d})_k|\le Cj e^{-a_0|k|_1}$ for $a_0>a_1$.
Thus $c_{\rm h}^0$ is a well-defined trigonometric polynomial even
for an infinite initial tail.  We retain it in the averaging jet,
so its $O(W_0)$ size does not enter the window error.

The averaging argument in Section~\ref{sec:filtered} uses the
pure-angle terms and all normal-linear modes as a single jet.
Let $g_{01}^{\rm raw}$ be this part of $g_K$, with every normal
index retained, and include the frozen initial diagonal by setting
\begin{equation}\label{rw:augmented-jet}
 g_{01}=g_{01}^{\rm raw}+c_{\rm h}^0.
\end{equation}
Let $g_2^{\rm low}$ consist of the terms linear in $p$, together
with those quadratic in $z$, in $g_{K,M}^{\ell}$.

\begin{proposition}[Finite-window reduction]\label{prop:window}
Suppose $p(0)=0$, $W_0\le s$, and \eqref{rw:stopped} holds on
$[0,T]$, with $T\le T_0/\eps$.  Let $K,M\ge1$ satisfy
$M^2\ge D(K+1)$, with $D$ as in Lemma~\ref{rw:lem-separation}.
For $0\le t\le T$, there is an exact decomposition
\begin{equation}\label{rw:alllinearidentity}
 \mathcal D(t)=\eps\int_0^t a_{g_{01}}(b(v))\,dv
                +\eps\int_0^t a_{g_2^{\rm low}}(b(v))\,dv+E_0(t),
\end{equation}
where
\begin{equation}\label{rw:e0budget}
 \begin{aligned}
 \sup_{t\le T}\|E_0(t)\|_{\mathcal A_{w_*}}
       &\le C(\eps+e^{-aK}+s^{3/2}),\\
 \sup_{t\le T}|\nu\cdot E_0(t)|
       &\le C(\eps+e^{-aK}+s^{5/2}).
 \end{aligned}
\end{equation}
Here $a,C>0$ are fixed independently of $K,M,\eps,s$ and the admissible
initial tail. \end{proposition}

\begin{proof}
Put $e_K=e^{-aK}$.  Integrating the normal action equation with
$f=g+r$ gives
$J_j(t)-J_j^0=G_j(t)+(E_R(t))_j$.  Hence the exterior weighted sum
is bounded by $\Gamma_{S_M}(t)+\|E_R(t)\|_{\mathcal A_{w_*}}$.
The leakage and raw error bounds \eqref{rw:leakage} and
\eqref{rw:rawbudget} yield
\begin{equation}\label{rw:diagonalfeedback}
 L_H(t):=\sum_{j>M}j^{w_*}|J_j(t)-J_j^0|
     \le C(\eps+e_K+s^{3/2})=:C\Lambda_s.
\end{equation}
Since $j\le j^{w_*}$, the
coefficient bound \eqref{rw:diagonalbounds} gives, along the stopped
trajectory,
\begin{equation}\label{rw:diagonalscalar}
 |d_{\rm h}^{\rm inc}|+|\partial_\varphi d_{\rm h}^{\rm inc}|_1\le C\Lambda_s.
\end{equation}
The action drift of $d_{\rm h}^{\rm inc}$ has only tangential
components.  Thus
$E_d(t)=\eps\int_0^t a_{d_{\rm h}^{\rm inc}}\,dv$ satisfies
$\|E_d(t)\|_{\mathcal A_{w_*}}\le C\eps t\Lambda_s\le C\Lambda_s$.

For the sharper scalar estimate, apply Lemma~\ref{rw:lem-chain}
to $d_{\rm h}^{\rm inc}$ in the full equation.  Its normal gradient
equals that of $d_{\rm h}$, since $c_{\rm h}^0$ is independent of $z$,
so the regularity established above applies.  Using
$h-h(I^0)=\nu\cdot\mathcal D+Q_0$ gives
\begin{equation}\label{rw:diagonalidentity}
 \begin{split}
 \nu\cdot E_d(t)={}&-\eps[d_{\rm h}^{\rm inc}(b(t))-d_{\rm h}^{\rm inc}(b(0))]\\
 &+\eps\int_0^t\{d_{\rm h}^{\rm inc},Q_0\}\,dv
       +\eps^2\int_0^t\{d_{\rm h}^{\rm inc},f\}\,dv.
 \end{split}
\end{equation}
Since $(d_{\rm h}^{\rm inc})_p=0$ and $d_{\rm h}^{\rm inc}$ commutes
with every normal action, the feedback bound
\eqref{rw:diagonalscalar} and $\|DQ_0\|_2\le Cs$ give
\[
 \{d_{\rm h}^{\rm inc},Q_0\}
       =(d_{\rm h}^{\rm inc})_\varphi\cdot(DQ_0)_A,
 \qquad |\{d_{\rm h}^{\rm inc},Q_0\}|\le Cs\Lambda_s.
\]
For the remaining bracket, use the uniform tube bound
$|(d_{\rm h}^{\rm inc})_\varphi|\le C(W+W_0)\le Cs$.
The normal gradients of $d_{\rm h}^{\rm inc}$ and $f$ lie in
$h^{\sigma-1}\subset\ell^2$, with sizes $C\sqrt s$ and $C$,
respectively.  Their normal bracket is therefore absolutely
convergent, and
\[
 |\{d_{\rm h}^{\rm inc},f\}|
 \le |(d_{\rm h}^{\rm inc})_\varphi|\,|f_p|
       +\|\nabla_zd_{\rm h}^{\rm inc}\|_0\|\nabla_zf\|_0
 \le C\sqrt s.
\]
 For the endpoints,
$|d_{\rm h}^{\rm inc}|\le C(W+W_0)\le Cs$ and
$d_{\rm h}^{\rm inc}(b(0))=0$.  Inserting these estimates in
\eqref{rw:diagonalidentity} and using $\eps T\le T_0$ gives
\begin{equation}\label{rw:diagonalbudget}
 |\nu\cdot E_d(t)|
       \le C(\eps s+s\Lambda_s+\eps\sqrt s)
       \le C(\eps+e_K+s^{5/2}).
\end{equation}
The angular Fourier tail contributes
\[
 \sup_{t\le T}
 \left\|\eps\int_0^t a_{g-g_K}\,dv\right\|_{\mathcal A_{w_*}}
       \le Ce_K.
\]
Analytic Fourier estimates for $(g-g_K)_\varphi$ and weighted
Cauchy--Schwarz as in \eqref{rw:cs} bound its tangential and normal
action components by $Ce_K$.  Time integration gives
$C\eps T e_K\le Ce_K$.  Its scalar pairing obeys the same bound,
since $\nu$ is a bounded functional on $\mathcal A_{w_*}$.

Finally, the exact increment identity
$\mathcal D(t)=\eps\int_0^t a_f\,dv$,
\eqref{rw:highdecomposition}, and \eqref{rw:frozen-diagonal} give
\begin{equation}\label{rw:windowidentity}
 \mathcal D(t)
      =\eps\int_0^t a_{g_{K,M}^{\ell}+c_{\rm h}^0}(b(v))\,dv
                +E_{\rm old}(t),
\end{equation}
with
\[
 E_{\rm old}=E_R+E_{\rm h}+E_d
                  +\eps\int_0^t a_{g-g_K}\,dv.
\]
Combine \eqref{rw:rawbudget}, \eqref{rw:highbudget},
\eqref{rw:diagonalbudget}, and the tail estimate.  Since $s\le1$,
the raw scalar terms $\eps s^{3/2}+\eps\sqrt s$ are bounded by
$2\eps$, giving
\begin{equation}\label{rw:windowbudget}
 \begin{aligned}
 \sup_{t\le T}\|E_{\rm old}(t)\|_{\mathcal A_{w_*}}
       &\le C(\eps+e^{-aK}+s^{3/2}),\\
 \sup_{t\le T}|\nu\cdot E_{\rm old}(t)|
       &\le C(\eps+e^{-aK}+s^{5/2}).
 \end{aligned}
\end{equation}

Since
\[
 g_{K,M}^{\ell}+c_{\rm h}^0
       =g_{01}-g_{1,\rm h}+g_2^{\rm low},
\]
retaining every normal-linear mode in $g_{01}$ gives
\eqref{rw:alllinearidentity} with
\begin{equation}\label{rw:e0definition}
 E_0(t)=E_{\rm old}(t)-\eps\int_0^t a_{g_{1,\rm h}}\,dv.
\end{equation}
The bound \eqref{rw:highlinear} controls the subtracted integral by
$C\eps$ in both $\mathcal A_{w_*}$ and the frozen-frequency pairing.
Together with \eqref{rw:windowbudget}, it proves \eqref{rw:e0budget};
the weighted norm also dominates $\|\cdot\|_2$.
\end{proof}

% END sections/03_window
% BEGIN sections/04_geometry
\section{Signed resonances and compatible confinement scales}
\label{sec:geometry}

The finite-window monomials from Section~\ref{sec:window} have signed
labels with at most two nonzero normal entries.
Subsections~\ref{subsec:geo-drift}--\ref{subsec:geo-hierarchy} use this
structure to select a resonance space with compatible action radii and
divisor bounds, without arithmetic assumptions on the base frequency.
Subsection~\ref{subsec:geo-energy} then combines its frequency bound
with energy conservation and concavity along physical action segments.
These results supply the averaging domain and the energy estimate used
to close the full action bootstrap.

\subsection{Signed monomials and action drift}
\label{subsec:geo-drift}

Relabel the $d$ tangential and $m$ retained normal coordinates by
$a=1,\ldots,d$ and $j=1,\ldots,m$, respectively.  Use
$\zeta_j=(u_j-iv_j)/\sqrt2$ and $\bar\zeta_j=(u_j+iv_j)/\sqrt2$.
These variables are complex conjugates on the real phase space and
independent holomorphic coordinates on its complexification.
With the convention $X_HF=\{F,H\}$ from \eqref{setup:poisson}, the
Poisson bracket is
\begin{equation}
 \{F,G\}=\sum_{a=1}^d
 (F_{\varphi_a}G_{p_a}-F_{p_a}G_{\varphi_a})
 +i\sum_{j=1}^m(F_{\zeta_j}G_{\bar\zeta_j}
                         -F_{\bar\zeta_j}G_{\zeta_j}).
 \label{geo:bracket}
\end{equation}
For $r\in\mathbb Z_{\ge0}^d$, $\alpha,\beta\in\mathbb Z_{\ge0}^m$,
and $k\in\mathbb Z^d$, define the monomial $M$, its signed label
$q(M)$, and its weighted degree $w(M)$ by
\begin{equation}
 M=e^{ik\cdot\varphi}p^r\zeta^\alpha\bar\zeta^\beta,
 \qquad q(M)=(k,\alpha-\beta),
 \qquad w(M)=2|r|_1+|\alpha|_1+|\beta|_1.
 \label{geo:monomial}
\end{equation}
Let $\pi=(p,J)$, where $J_j=\zeta_j\bar\zeta_j$, and restrict the frozen
frequency sequence $\nu$ to these retained action coordinates.
Direct differentiation gives the polynomial identities
\begin{equation}
 D\pi X_M=-iq(M)M,
 \qquad \{M,\nu\cdot\pi\}=i(\nu\cdot q(M))M.
 \label{geo:drift}
\end{equation}
\begin{lemma}[Signed selection]\label{geo:lem-signed}
Let $G$ be a real finite Fourier--Taylor polynomial with like monomials
combined.  Partition its monomials into classes invariant under
\[
 (k,r,\alpha,\beta)\longmapsto(-k,r,\beta,\alpha).
\]
The sum of the terms in each class is real.  If its nonzero labels
span a real subspace $V\subset\R^{d+m}$, its action drift belongs to $V$ at every
real phase point.  Its integral along any continuous real path on a
compact interval also belongs to $V$.  Among monomials with $w(M)\le2$,
those with zero label are exactly $1,p_a,J_j$, and their action drift
vanishes.
\end{lemma}

\begin{proof}
Since $G$ is real, the coefficient of the conjugate monomial $M^*$ is
the complex conjugate of the coefficient of $M$.  The conjugate label
is $-q(M)$, so selection preserves this relation.  A pair
$cM+\overline cM^*$ with $q(M)\ne0$ has action drift
\[
 -iq(M)cM+iq(M)\overline{cM}
      =2q(M)\operatorname{Im}(cM).
\]
This is a real multiple of its label, and zero-label terms contribute
no drift by \eqref{geo:drift}.  Summation and then integration prove
the span statements.  Finally, $q=0$ means $k=0$ and
$\alpha=\beta$.  The weighted-degree restriction leaves exactly the
listed monomials.  Formula~\eqref{geo:drift} shows that their drift is
zero.
\end{proof}

The same argument applies to real-analytic coefficient functions $c_M(I)$
of the full action parameter, provided
$c_{M^*}(I)=\overline{c_M(I)}$ for real $I$ treated as an independent
parameter.  After substituting $I=\xi+\pi$, with $\pi$ now the full
action map, the coefficient derivatives contribute zero because the
action components Poisson commute:
$\{\pi,c_M(I)M\}=c_M(I)\{\pi,M\}$.

For an infinite expansion, take finite truncations that preserve
conjugate pairs.  If all selected labels lie in a fixed finite-dimensional
real subspace $V$ and the truncated drifts converge in the action norm,
their limit belongs to $V$, since $V$ is closed.  Its integral also lies
in $V$ whenever the drift is Bochner integrable along the path.  The
coefficient estimates below give the required convergence and continuity
along the supplied trajectory.

The signs matter even at normal degree one.  Let $e_j$ be the $j$th
standard basis vector of $\R^m$ and put $\theta=k\cdot\varphi$.  Then
\[
 \cos\theta\,u_j
 =\frac{u_j\cos\theta+v_j\sin\theta}{2}
  +\frac{u_j\cos\theta-v_j\sin\theta}{2}.
\]
The two summands have label pairs $\pm(k,e_j)$ and $\pm(k,-e_j)$,
respectively, with divisors $k\cdot\omega(I)+\Omega_j(I)$ and
$k\cdot\omega(I)-\Omega_j(I)$.  One pair may be resonant while the
other is not.  Likewise, $e^{ik\cdot\varphi}J_j$ has label $(k,0)$.
It preserves every normal action for all $k$, and depends only on
actions precisely when $k=0$.

\subsection{Polynomial conditioning of signed quadratic labels}
\label{subsec:geo-conditioning}

Let $e_j$ denote the standard basis in $\R^m$.  Write
\begin{equation}
 \mathscr E_m=
 \{0,\ \pm e_j,\ \pm2e_j,\ \pm e_j\pm e_h:\ j\ne h\}.
 \label{geo:labelclass}
\end{equation}
For $m=0$, this set consists only of the zero vector.
The next lemma gives a coefficient bound for the resonance spaces used
below.

\begin{lemma}[Polynomial conditioning]\label{lem:conditioning}
Let $d,m$ be nonnegative integers and $K_*\ge0$.  Let $\mathscr Q$
be a finite subset of
$\mathbb Z^d\times\mathbb Z^m$ with
each $q=(k_q,\ell_q)$ satisfying $|k_q|_1\le K_*$ and
$\ell_q\in\mathscr E_m$.  Put
$V=\operatorname{span}_{\R}\mathscr Q$.
There is a subset $\mathscr Q_*\subset\mathscr Q$, with at most
$d+m$ elements, such that every $x\in V$ has a representation
\begin{equation}
 x=\sum_{q\in\mathscr Q_*}c_q q,
 \qquad \sum_{q\in\mathscr Q_*}|c_q|
 \le\kappa_{d,m,K_*}\|x\|_2,
 \qquad
 \kappa_{d,m,K_*}\le C_d(1+m+K_*)^{2d+2}.
 \label{geo:conditioning}
\end{equation}
The constant $C_d$ depends only on $d$.  The same bound holds for
every subset of $\mathscr Q$, with the spanning subset chosen anew.
Give $V$ the induced Euclidean norm and $V^*$ its dual norm.
Consequently, for any
$\nu\in\R^{d+m}$,
\begin{equation}
 \|\nu|_V\|_{V^*}
 \le\kappa_{d,m,K_*}\max_{q\in\mathscr Q}|\nu\cdot q|.
 \label{geo:restriction}
\end{equation}
If $\mathscr Q$ is empty, both sides in~\eqref{geo:restriction} are
interpreted as zero.
\end{lemma}

\begin{proof}
We first select a basis for the normal projection of $V$.  Form a
labelled multigraph on the normal indices occurring in $\mathscr Q$.
Each full label with two nonzero normal entries gives an edge;
labels with one nonzero normal entry contribute vertices but no edges.
Choose a spanning tree in each connected component.  A diagonal
change of signs in the normal coordinates makes the normal parts of
the chosen tree labels differences $e_u-e_v$, up to overall sign.
Assign these coordinate signs successively from a root.  The change
preserves both the Euclidean and $\ell^1$ norms.  Reversing a full
label only reverses its coefficient, so the construction yields
representations in the original elements of $\mathscr Q$ even when
$\mathscr Q$ is not symmetric.

Consider one component $C$ with $n_C$ vertices and root $r_C$.  The
normal parts of its tree labels form a basis of
\[
 E_C=\{w\in\R^C:\ \sum_{j\in C}w_j=0\}.
\]
For $T_j=e_j-e_{r_C}$, the tree path expresses $T_j$ with coefficients
$\pm1$ of total absolute value at most $n_C-1$.
After the sign change, every remaining normal label, up to overall
sign, is $e_u-e_v$, $e_u+e_v$, $e_u$, or $2e_u$.

If all normal labels in the component have coordinate sum zero, their
span is $E_C$.  For $w\in E_C$, the identity
$w=\sum_{j\in C}w_jT_j$ gives tree coefficients of total absolute
value at most $(n_C-1)|w|_1$.
Otherwise choose a normal label $p_C$ whose coordinate sum is nonzero
and choose its overall sign so that
\[
 h_C=\sum_{j\in C}(p_C)_j\in\{1,2\},
 \qquad |p_C|_1=h_C.
\]
The normal parts of the tree labels, together with $p_C$, form a
basis of $\R^C$.  With
$s_C=\sum_{j\in C}w_j$,
\begin{equation}
 w=\frac{s_C}{h_C}p_C+
 \sum_{j\in C}\left(w_j-\frac{s_C}{h_C}(p_C)_j\right)T_j.
 \label{geo:pinning}
\end{equation}
The sum of the absolute values of its coefficients in the chosen basis is at
most
\[
 \frac{|s_C|}{h_C}+(n_C-1)
 \left(|w|_1+\frac{|s_C|}{h_C}|p_C|_1\right)
 \le(2n_C-1)|w|_1.
\]

Return to the original coordinate and label signs, and let
$P\subset\mathscr Q$ consist of the full labels selected by all trees
and the possible pinning labels.  Their normal parts are linearly
independent and span the normal projection $U$ of $V$, so
$|P|=\dim U\le m$.  The preceding
formulas define a linear coefficient map $A:U\to\R^P$ satisfying
\begin{equation}
 \sum_{p\in P}(Aw)_p\ell_p=w,
 \qquad \|Aw\|_1\le2m|w|_1\le2m^2\|w\|_2.
 \label{geo:normalinverse}
\end{equation}
For $m=0$, this is the zero map.  For $w\in U\cap\mathbb Z^m$,
the tree and pinning formulas give $Aw\in(\tfrac12\mathbb Z)^P$:
only division by $1$ or $2$ occurs in~\eqref{geo:pinning}, and
tree paths have integer coefficients.  Since $|\ell_q|_1\le2$,
for every $q\in\mathscr Q$ we obtain
\begin{equation}
 A\ell_q\in(\tfrac12\mathbb Z)^P,
 \qquad \|A\ell_q\|_1\le4m.
 \label{geo:half-integral}
\end{equation}
This normal-coordinate half-integrality is a standard property of
bidirected incidence matrices; see Bolker and
Zaslavsky~\cite{BolkerZaslavsky2006}.

Eliminate the normal component of each full label by setting
\begin{equation}
 R_q=q-\sum_{p\in P}(A\ell_q)_p p=(h_q,0).
 \label{geo:residual}
\end{equation}
Then $2h_q\in\mathbb Z^d$ and
$|h_q|_1\le B_*:=(1+4m)K_*$.  This elimination is linear and fixes
vectors with zero normal component, so the $R_q$ span
$V\cap(\R^d\times\{0\})$.  Choose independent residuals
$h_{q_1},\ldots,h_{q_r}$ spanning all residuals; here $r\le d$.
For $r\ge1$, some $r$ coordinate rows of the integer columns
$2h_{q_i}$ form an invertible square matrix.  Its determinant has
absolute value at least one.  Each cofactor has absolute value at
most $(r-1)!(2B_*)^{r-1}$, by the determinant expansion.  For $y$
in the residual span, the corresponding system has right-hand side
equal to twice the selected coordinates of $y$.  Cramer's rule gives
\begin{equation}
 y=\sum_{i=1}^r\alpha_i h_{q_i},
 \qquad \sum_{i=1}^r|\alpha_i|
 \le C_d'(B_*)\|y\|_2,
 \qquad C_d'(B_*)=2d\,d!(1+2B_*)^{d-1}.
 \label{geo:cramer}
\end{equation}
For $r=0$ the representation is empty.  When $d=0$ this entire
residual step is absent.

Now let $x=(x_T,x_N)\in V$.  Subtract the full pivot labels with
coefficients $a=Ax_N$.  The remainder is $(y,0)$ in the residual
span, and
\[
 \|a\|_1\le2m^2\|x\|_2,
 \qquad \|y\|_2\le(1+2m^2K_*)\|x\|_2.
\]
Represent $y$ by~\eqref{geo:cramer} and substitute
\eqref{geo:residual}.  This uses only the labels in
$\mathscr Q_*=P\cup\{q_1,\ldots,q_r\}$, at most $m+d$ of them,
with  coefficients whose absolute values sum to at most $\|x\|_2$ times
\[
 2m^2+(1+4m)C_d'(B_*)(1+2m^2K_*).
\]
For $d\ge1$, take one plus this expression as
$\kappa_{d,m,K_*}$; for $d=0$, take $1+2m^2$.
Writing $n_*=1+m+K_*$, use
\[
 1+4m\le5n_*,\qquad
 1+2B_*\le11n_*^2,\qquad
 1+2m^2K_*\le3n_*^3
\]
to obtain~\eqref{geo:conditioning}.  One may take
$C_d=3+30d\,d!\,11^{d-1}$ for $d\ge1$ and $C_0=3$ for $d=0$.  These constants do not depend on which labels are present.
Finally, pair the representation of $x$ with $\nu$ and take the
supremum over $x\in V$ with $\|x\|_2\le1$ to
get~\eqref{geo:restriction}, including when $V=\{0\}$.
\end{proof}

The normal elimination uses only integer coefficients and division by
$2$.  The remaining determinant has size at most the fixed tangential
dimension $d$, which gives the polynomial dependence on $m$ and $K_*$
needed below.

\subsection{A rank hierarchy on compatible action domains}
\label{subsec:geo-hierarchy}

We use Lemma~\ref{lem:conditioning} to choose a finite resonance space
and an action radius on which the remaining divisors stay separated
from zero.  The polynomial conditioning bound controls the dependence of these scales on
the size of the normal window.

The base frequency $\nu=Dh(I^0)$ is the componentwise sequence
introduced in~\eqref{setup:frequencies}.  Although $\nu\notin\ell^2$,
its restriction to any finite coordinate subspace is well defined.
On a fixed complex $\ell^2$ neighborhood of $\mathcal K$, write
\begin{equation}
 Dh(I)=\nu^{\mathrm{lin}}+Dh_{\mathrm{nl}}(I),
 \qquad \|D^2h_{\mathrm{nl}}(I)\|\le C_h.
 \label{geo:hessian}
\end{equation}
Here $\nu^{\mathrm{lin}}_j=(2\pi j)^3$ is the action-independent
dispersive frequency, indexed by the original physical modes.
The nonlinear gradient is $\ell^2$-valued, so
$Dh(I)-\nu=Dh_{\mathrm{nl}}(I)-Dh_{\mathrm{nl}}(I^0)\in\ell^2$.
The uniform complex Hessian bound follows from analyticity and
compactness, as in~\eqref{setup:concavity}.

Let $I^0$ range over a fixed smaller real neighborhood of $\mathcal K$
with a uniform positive distance from the complement of the complex
neighborhood.  The center may have infinite support.
For every finitely supported label $q$ define
$d_q(I)=Dh(I)\cdot q$.  Whenever the segment $[I^0,I]$ lies in
this complex neighborhood, integration of the nonlinear Hessian gives
\begin{equation}
 |d_q(I)-\nu\cdot q|
 \le C_h\|q\|_2\|I-I^0\|_2.
 \label{geo:variation}
\end{equation}
This estimate holds for real or complex $I$.  For a normal-linear label
$q=(k,\pm e_j)$, the bound $\|q\|_2\le |k|_1+1$ is independent
of the physical normal index $j$.

\begin{proposition}[Compatible hierarchy]\label{prop:hierarchy}
Fix the tangential dimension $d$, a positive integer $L$, and a
positive constant $\mu$.  For $K\ge1$ and an integer $M\ge0$, let
$S_M=\mathcal N\cap[1,M]$, $m=|S_M|$, and $n=d+m$.
Let $\mathscr Q$ be any finite symmetric set of
labels supported on the tangential coordinates and $S_M$, with
$|k|_1\le LK$ and normal component in $\mathscr E_m$.
Choose $C_d\ge1$ as in Lemma~\ref{lem:conditioning} and set
\begin{equation}
 B_Q=LK+2,\qquad
 \kappa=C_d(1+m+LK)^{2d+2},\qquad
 R=\max\{2,32C_hB_Q\kappa/\mu\}.
 \label{geo:ratio}
\end{equation}
For $\alpha>0$ define $\alpha_j=\alpha R^j$ and
\begin{equation}
 V_j=\operatorname{span}_{\R}
 \{q\in\mathscr Q:|\nu\cdot q|\le\alpha_j\},
 \qquad 0\le j\le n+1.
 \label{geo:hierarchy}
\end{equation}
There is $j\in\{0,\ldots,n\}$ for which $V_j=V_{j+1}$.
For any such $j$, put
\begin{equation}
 V=V_j,\qquad \Delta=\alpha_{j+1},\qquad
 \rho=8\kappa\alpha_j/\mu,
 \qquad d_V=\|\nu|_V\|_{V^*}.
 \label{geo:chosen}
\end{equation}
Use $d_{\{0\}}=0$.  Then
\begin{equation}
 d_V\le\mu\rho/8,
 \qquad q\in\mathscr Q\setminus V
       \ \Longrightarrow\ |\nu\cdot q|>\Delta.
 \label{geo:rankgap}
\end{equation}
If the complex ball $\|I-I^0\|_2\le2\rho$ lies in the fixed
neighborhood from~\eqref{geo:hessian}, every
$q\in\mathscr Q\setminus V$ satisfies
\begin{align}
 |d_q(I)-\nu\cdot q|&\le\Delta/4,
 & |d_q(I)|&>3\Delta/4
 &&(\|I-I^0\|_2\le\rho),\label{geo:realdivisor}\\
 |d_q(I)-\nu\cdot q|&\le\Delta/2,
 & |d_q(I)|&>\Delta/2
 &&(\|I-I^0\|_2\le2\rho).
 \label{geo:complexdivisor}
\end{align}
The first line applies on the real ball of radius $\rho$; the second
also holds on the complex ball of radius $2\rho$.

For fixed $b,c>0$, let $0<\lambda<1$,
$\alpha=\lambda^b$, $K=\lceil c\log(1/\lambda)\rceil$, and
$M\asymp\sqrt K$, with fixed comparison constants.  As
$\lambda\downarrow0$, uniformly over $\mathscr Q$, the choice of
$j$, and the admissible base action,
\begin{equation}
 c_1\lambda^b\le\rho,\Delta
 \le\lambda^{b-o(1)},\qquad
 \log R^{n+1}=O(\sqrt K\log(K+1))
             =o(\log(1/\lambda)).
 \label{geo:subpower}
\end{equation}
Here $o(1)$ can be chosen nonnegative and independent of the label
family, the selected index, and the base action.  Thus the upper bound
is $\lambda^b$ times a uniform subpower factor.  In particular, the
required balls eventually lie in the fixed neighborhood and
$0<\Delta\le1$.
\end{proposition}

\begin{proof}
The spaces $V_j$ are nested subspaces of an $n$-dimensional space.
Among the $n+1$ successive inclusions, at least one is therefore an
equality.  Lemma~\ref{lem:conditioning}, applied to the labels
defining $V_j$, gives $d_V\le\kappa\alpha_j=\mu\rho/8$.
If $q\in\mathscr Q\setminus V$ had $|\nu\cdot q|\le\alpha_{j+1}$, it would
belong to $V_{j+1}=V$, proving~\eqref{geo:rankgap}.

For $q=(k,\ell)\in\mathscr Q$, one has
$\|q\|_2\le |k|_1+|\ell|_1\le B_Q$.  From
\eqref{geo:ratio} and~\eqref{geo:chosen},
\[
 C_hB_Q\rho
 =8C_hB_Q\kappa\alpha_j/\mu
 \le R\alpha_j/4=\Delta/4.
\]
The segment from $I^0$ to $I$ stays in the larger ball, so
\eqref{geo:variation} applies.  The variation is at most $\Delta/4$
on the ball of radius $\rho$ and at most $\Delta/2$ on the ball of
radius $2\rho$.  Combining these bounds with the strict gap
in~\eqref{geo:rankgap} proves the divisor estimates.  The larger
complex ball gives an open neighborhood of the real ball on which
the inverse divisors are analytic.

Since $M\asymp\sqrt K$, we have $n=O(\sqrt K)$.  For fixed
$d,L,C_h,\mu$, the quantities $\kappa$ and $R$ are bounded by
fixed powers of $K+1$.  Hence
\[
 \log(\kappa R^{n+1})=O(\sqrt K\log(K+1)).
\]
Dividing this bound by $\log(1/\lambda)\asymp K$ gives a number tending
to zero.  The formulas for $\rho,\Delta$ therefore imply the uniform
upper bounds in~\eqref{geo:subpower}.  Their lower bounds follow from
$R\ge2$, $\kappa\ge1$ and $j\ge0$; for example one can take
$c_1=\min\{1,8/\mu\}$.  Because $b>0$, the upper bounds tend to
zero, proving the remaining assertions.
\end{proof}

The space $V$ is selected at $I^0$ and then kept fixed;
the estimates are uniform in this choice.

\begin{corollary}[High normal-linear labels]\label{geo:cor-highunary}
Suppose the normal frequencies have the form
\[
 \Omega_j(I)=c_0j^3+\beta_j(I),\qquad c_0>0,
\]
with the normal corrections $\beta_j$ uniformly bounded in
$j\in\mathcal N$, and the tangential frequencies uniformly bounded,
on the fixed complex action neighborhood.  With the asymptotic
choices in Proposition~\ref{prop:hierarchy}, take $\lambda$
sufficiently small, uniformly over the admissible centers and selected
indices of the hierarchy.  Every normal-linear label $q=(k,\pm e_j)$ with
$j\in\mathcal N$, $j>M$, and $|k|_1\le LK$ then satisfies
\begin{equation}
 |d_q(I)|\ge c_2j^3,
 \qquad |\nu\cdot q/d_q(I)|\le2
 \qquad (\|I-I^0\|_2\le2\rho),
 \label{geo:highunary}
\end{equation}
where $c_2>0$ is fixed.  These labels are outside the finite space
$V$.  They also satisfy
$|d_q(I)-\nu\cdot q|\le\Delta/2$ on the same complex ball.
\end{corollary}

\begin{proof}
Choose $C_\omega,C_\beta$ bounding the tangential frequencies and
normal corrections, respectively, on the fixed complex neighborhood.
The triangle inequality gives
\[
 |d_q(I)|\ge c_0j^3-LK C_\omega-C_\beta.
\]
Since $M\asymp\sqrt K$, the ratio $M^3/(LK+1)$ tends to infinity
for fixed $L$.  Thus the right side is at least $c_0j^3/2$ for
$j>M$, uniformly on the ball, once $\lambda$ is small enough.
The index-independent label bound $\|q\|_2\le LK+1\le B_Q$ and
\eqref{geo:variation} give variation at most
$2C_hB_Q\rho\le\Delta/2$.
Because $\Delta\le1$ and the cubic lower bound grows with $M$,
this variation is at most $|d_q(I)|$ after a further uniform reduction
of $\lambda$.  Thus
$|\nu\cdot q|\le |d_q(I)|+\Delta/2\le2|d_q(I)|$, proving
the ratio bound.
Finally, a vector with a nonzero coordinate at $j>M$ cannot belong
to $V$, which is supported on $S_M$ and the tangential coordinates.
\end{proof}

\subsection{Energy closure on physical action segments}
\label{subsec:geo-energy}

The preceding hierarchy makes the frozen frequency small on $V$.
The next lemma combines this bound with energy conservation to
control the full action increment from its decomposition into a
$V$-component and an error.

\begin{lemma}[Physical energy closure]\label{lem:energyclosure}
Use the frozen frequency $\nu$ and the nonlinear Taylor remainder
$Q_0$ defined in~\eqref{setup:concavity}.  Let
$\mathcal D=(p,J-J^0)\in\mathcal A_3$, $B=\|\mathcal D\|_2$, and
assume that the segment $I^0+t\mathcal D$, $0\le t\le1$, lies in
the physical action neighborhood specified there.  Then
\begin{equation}
 h(I^0+\mathcal D)-h(I^0)
   =\nu\cdot\mathcal D+Q_0(\mathcal D),
 \qquad Q_0(0)=DQ_0(0)=0,
 \label{geo:energyexpansion}
\end{equation}
and $D^2h_{\mathrm{nl}}\le-\mu\operatorname{Id}$ along this segment,
with $\mu>0$.  Suppose the energy bound
\begin{equation}
 |\nu\cdot\mathcal D+Q_0(\mathcal D)|\le C_E\eps
 \label{geo:energybound}
\end{equation}
holds for $C_E\ge0$.  Let $\mathcal D=A_V+E$, with $A_V\in V$
for a real subspace $V$ supported on finitely many coordinates, and suppose
\[
 \|E\|_2\le\eta,\qquad |\nu\cdot E|\le\tau,
 \qquad d_V=\|\nu|_V\|_{V^*}.
\]
Since $A_V$ has finite support, $E=\mathcal D-A_V\in\mathcal A_3$,
so all scalar pairings converge absolutely. Then
\begin{equation}
 \frac\mu2 B^2
 \le C_E\eps+d_V(B+\eta)+\tau,
 \qquad
 B\le\frac{2d_V}{\mu}
    +\sqrt{\frac{2(d_V\eta+\tau+C_E\eps)}{\mu}}.
 \label{geo:energyclosure}
\end{equation}
In particular, with $\mu>0$ fixed and $C_E$ uniformly bounded, if
$d_V\le\mu\rho/8$, $\eta=o(\rho)$ and
$\tau+\eps=o(\rho^2)$ uniformly as $\rho\to0$, then
$B\le\rho/4+o(\rho)$.
\end{lemma}

\begin{proof}
Apply the integral Taylor formula to $h_{\mathrm{nl}}$ along the
specified physical segment:
\[
 Q_0(\mathcal D)
 =\int_0^1(1-t)D^2h_{\mathrm{nl}}(I^0+t\mathcal D)
               [\mathcal D,\mathcal D]\,dt
 \le-\frac\mu2B^2.
\]
Using~\eqref{geo:energybound}, then the exact decomposition, yields
\[
 \frac\mu2 B^2
 \le -Q_0(\mathcal D)
 \le C_E\eps+\nu\cdot\mathcal D
 \le C_E\eps+d_V\|A_V\|_2+\tau
 \le C_E\eps+d_V(B+\eta)+\tau.
\]
Solving this quadratic inequality for $B\ge0$ gives the upper bound
\[
 \frac{d_V}{\mu}+
 \sqrt{\frac{d_V^2}{\mu^2}
             +\frac{2(d_V\eta+\tau+C_E\eps)}{\mu}}.
\]
The inequality $\sqrt{a+b}\le\sqrt a+\sqrt b$ proves the stated
bound.  For the final assertion, the hypotheses give
$d_V\eta+\tau+C_E\eps=o(\rho^2)$ and $2d_V/\mu\le\rho/4$.
\end{proof}

Let $F_0$ be a uniform bound for $|f|$ on the coordinate tube.
By Lemma~\ref{lem:energy} and its proof, the segment $[I^0,I(t)]$
lies in the physical concavity neighborhood and \eqref{geo:energybound}
holds with $C_E=2F_0$, uniformly in $J^0$.
The bootstrap therefore needs separate
bounds for $\|E\|_2$ and $|\nu\cdot E|$; an unweighted
action-norm bound alone does not control the latter.
% END sections/04_geometry
% BEGIN sections/05_filtered
\section{Finite-order averaging of angular and normal-linear observables}\label{sec:filtered}
\label{the-all-mode-linear-jet-and-its-error-budget}

The window identity~\eqref{rw:alllinearidentity} expresses the action
increment as integrals of the augmented angular and normal-linear jet
$g_{01}$ and the low-mode quadratic jet $g_2^{\mathrm{low}}$, plus
$E_0$. Under Assumption~\ref{ass:contract} and the compatible divisor
bounds below, Proposition~\ref{prop:filtered} reduces these integrals
to a $V$-valued drift and errors with separate action and scalar bounds.
These are the inputs to the energy closure in Section~\ref{sec:closure}.

We use $I=\xi+\mathcal I=I^0+\mathcal D$, with
$\mathcal I=(p,J)$ and $\mathcal D=(p,J-J^0)$. Along the actual strong
trajectory, $I'=\mathcal D'=\varepsilon a_f$, where $a_f=D\pi X_f$
by Lemma~\ref{lem:action-chain}. All integrals are evaluated on this
trajectory; $B=\|\mathcal D\|_2$ and $W=|p|_1+\|z\|_\sigma^2$
retain the normalization of~\eqref{setup:action-norms}.

Fix an integer averaging order $L\ge1$ before choosing the small
parameters.  Let $K,M\ge1$ be integer cutoffs with
$M^2\ge D(K+1)$, where $D$ is the constant in
Lemma~\ref{rw:lem-separation}.  The augmented jet is

\begin{equation}
g_{01}(\varphi,z)=
 \sum_{|k|_1\le K}a_ke^{ik\varphi}
 +\sum_{|k|_1\le K}\sum_{j\in\mathcal N}\sum_{\iota=\pm1}
 b_{k,j,\iota}e^{ik\varphi}\zeta_j^\iota.
 \label{filt:eq1}
\end{equation}

Here $\zeta_j^+=\zeta_j$, $\zeta_j^-=\bar\zeta_j$, and $\iota=\pm1$
denotes the sign.  As in~\eqref{rw:augmented-jet}, $a_k$ is the
pure-angle coefficient of $f(\varphi,0,0)$ plus $(c_{\rm h}^0)_k$,
and $b_{k,j,\iota}$ is the original normal-linear coefficient.
Thus the jet has angular cutoff $K$ and retains every normal-linear
index.  The parameters $\xi,I^0,J^0$ are fixed during differentiation
in the running action $I$, so all such derivatives of these initial
coefficients vanish.

Set $\varsigma=\sigma-1$.  By the represented normal-gradient bound in Assumption~\ref{ass:contract}, the
analytic Fourier norms of the $b$-sequences in $h^\varsigma$ are uniformly
summable on a fixed smaller angular strip.  The added angular coefficients have norm at most $CW_0$ by
\eqref{rw:frozen-diagonal-bound}.  These bounds are uniform in the
admissible initial tails and the cutoffs.

Write $g_2^{\rm low}$ for the finite quadratic jet in
\eqref{rw:alllinearidentity}, including its terms linear in $p$.
Let $V$ be a real subspace of the action space supported on $A$ and
the normal indices $j\le M$, and let $e_j$ denote the normal
coordinate unit vector.  Let $\mathscr Q$ be the finite family of
all angular labels $(k,0)$ with $|k|_1\le LK$, all normal-linear
labels $(k,\pm e_j)$ with the same angular bound and
$j\in\mathcal N$, $j\le M$, and all original labels of $g_2^{\rm low}$.
Fix $\mathscr Q$ and $V$ before computing the averaging coefficients,
including labels whose coefficients vanish.

Let $\mathscr D$ be a common complex action neighborhood contained
in the fixed analytic neighborhood of Section~\ref{sec:setup}.
For every $q\in\mathscr Q\setminus V$, assume on $\mathscr D$ that
\begin{equation}
|d_q(I)|\ge\Delta/2,\;
 |d_q(I)-\nu\cdot q|\le\Delta/2,\;
 \nu=Dh(I^0),\quad d_q(I)=Dh(I)\cdot q,
 \quad0<\Delta\le1.
 \label{filt:eq2}
\end{equation}
For each high normal-linear label $q=(k,\pm e_j)$ with $|k|_1\le LK$ and
$j\in\mathcal N$, $j>M$, assume either~\eqref{filt:eq2} or the
cubic lower bound $|d_q(I)|\ge c j^3$ together with a uniform bound
on $|\nu\cdot q/d_q(I)|$, on the same domain $\mathscr D$.
Here $c>0$ and the ratio bound are fixed independently of the label.
The space $V$ and the domain $\mathscr D$ remain fixed throughout
the construction.  Proposition~\ref{prop:hierarchy} and
Corollary~\ref{geo:cor-highunary} supply these hypotheses for the
asymptotic cutoff choices used in Section~\ref{sec:closure}.

Consider a stopped initial interval $[0,T]$ on which
\eqref{rw:stopped} holds, $I(t)\in\mathscr D$, and
$T\le T_0/\varepsilon$.  In particular, $W_0\le s$.
Put $e_K=e^{-aK}$, where $a>0$ is a reduced angular analyticity
width chosen for the fixed order $L$, and $T_0$ is the supplied
time constant.

\begin{proposition}[Finite-order observable reduction]\label{prop:filtered}
Under the preceding hypotheses, for every $0\le t\le T$ there is
an exact decomposition

\begin{equation}
\mathcal D(t)=A_V(t)+E_0(t)+E_{2}(t)+E_{\mathrm{av}}^{[L]}(t),
 \qquad A_V(t)\in V,
 \label{filt:eq3}
\end{equation}

where $E_0$ is the window error in Proposition~\ref{prop:window}.
Its bound~\eqref{rw:e0budget} gives

\begin{equation}
\|E_0\|_2\le C(\varepsilon+e_K+s^{3/2}),\qquad
 |\nu \cdot E_0|\le C(\varepsilon+e_K+s^{5/2}),
 \label{filt:eq4}
\end{equation}

the contribution $E_2$ of the labels of $g_2^{\rm low}$ outside $V$
satisfies

\begin{equation}
 \begin{aligned}
 \|E_2\|_2&\le C\mathfrak p(M)\varepsilon(\Delta^{-1}+s\Delta^{-2}),\\
 |\nu\cdot E_2|&\le C\mathfrak p(M)\varepsilon(1+s\Delta^{-1}),
 \qquad \mathfrak p(M)=(1+M)^3,
 \end{aligned}
 \label{filt:eq5}
\end{equation}

and the finite-order averaging error satisfies

\begin{equation}
\|E_{\mathrm{av}}^{[L]}\|_2\le C_L\sum_{n=1}^L
 \left(\frac{\varepsilon^n}{\Delta^{2n-1}}
 +\frac{\varepsilon^n(s+e_K)}{\Delta^{2n}}\right)
 +C_L\frac{\varepsilon^L}{\Delta^{2L}},
 \label{filt:eq6}
\end{equation}
\begin{equation}
|\nu \cdot E_{\mathrm{av}}^{[L]}|\le C_L\sum_{n=1}^L
 \left(\frac{\varepsilon^n}{\Delta^{2n-2}}
 +\frac{\varepsilon^n(s+e_K)}{\Delta^{2n-1}}\right)
 +C_L\frac{\varepsilon^L}{\Delta^{2L-1}}.
 \label{filt:eq7}
\end{equation}

All estimates are uniform on $[0,T]$, over the admissible initial data
and the choices of $V$ satisfying the hypotheses.  The scalar
quantities are the absolutely convergent frozen-frequency pairings.
Constants $C,C_L$ may depend on $L,T_0$, the fixed finite-gap family,
the standing analytic bounds and the fixed high normal-linear separation
and ratio bounds.  They are independent of
$K,M,\varepsilon,\Delta,s$; the polynomial dependence on $M$ in
\eqref{filt:eq5} is displayed explicitly as $\mathfrak p(M)$.
\end{proposition}

\paragraph{Proof of Proposition~\ref{prop:filtered}.}
The construction has four steps.
\begin{enumerate}
\item Subsection~\ref{the-coefficient-action-poisson-algebra} treats the
running action and the explicit variables as independent. Its auxiliary
bracket separates coefficient derivatives from explicit derivatives and
makes terms of normal degree at least two a Poisson ideal. This permits
finite-order averaging in the degree-at-most-one quotient.

\item Subsection~\ref{analytic-module-bounds} proves the uniform
coefficient estimates. Each generated normal-linear coefficient contains
exactly one original $b_{k,j,\iota}$ factor; pure-angle coefficients contain
none. This preserves the available $h^{\sigma-1}$ regularity.

\item Subsection~\ref{finite-formal-normal-form-and-the-scalar-gain}
constructs the corrected observable through order $L$. The compatible
bound on $(\nu\cdot q)/d_q(I)$ saves one inverse-divisor factor in its
frozen-frequency pairing, at every retained order.

\item Subsection~\ref{exact-discarded-terms-and-the-full-trajectory}
restores the discarded brackets and explicit derivatives along the
supplied strong trajectory, estimating the averaging error while keeping
the leading residual integral exact.
Subsection~\ref{the-retained-quadratic-integral} identifies that integral
with the low quadratic drift plus $E_0$ and averages its labels outside
$V$ in one step to obtain $E_2$.
\end{enumerate}

\hypertarget{the-coefficient-action-poisson-algebra}{%
\subsection{The coefficient-action Poisson algebra}\label{the-coefficient-action-poisson-algebra}}

Treat the action parameter $I$ and the explicit variables
$(\varphi,p,\zeta,\bar\zeta)$ as independent; on the complexified
domain, $\zeta$ and $\bar\zeta$ are also independent.  For a monomial
$F=c(I)e^{ik\varphi}p^r\zeta^\alpha\bar\zeta^\beta$, take
$k\in\mathbb Z^A$, $r\in\mathbb Z_{\ge0}^A$, and finitely supported
nonnegative multi-indices $\alpha,\beta$ on $\mathcal N$, and put
$q=(k,\alpha-\beta)$.  The derivations
$D_a=\partial_{\varphi_a}$ for $a\in A$ and
$D_j=i(\zeta_j\partial_{\zeta_j}-\bar\zeta_j\partial_{\bar\zeta_j})$
for $j\in\mathcal N$ satisfy $D_jF=iq_jF$ and extend linearly to
finite sums.  They commute pairwise and with every
$\partial_{I_\ell}$.  Define

\begin{equation}
\{F,G\}_\times
 =\sum_j\bigl(D_jF\,\partial_{I_j}G
                 -\partial_{I_j}F\,D_jG\bigr).
 \label{filt:eq8}
\end{equation}

Here $j$ ranges over all tangential and normal action coordinates.
At a finite explicit normal cutoff, the commuting derivations make
this a Poisson bracket, so the Jacobi identity holds exactly.
Set $a_F=(-D_jF)_j$; for a monomial, $a_F=-iqF$.  Then

\begin{equation}
\{F,G\}_\times=D_IF[a_G]-D_IG[a_F].
 \label{filt:eq9}
\end{equation}

After substituting $I=\xi+\pi$ in both factors, their physical
bracket is the sum of \eqref{filt:eq8} and the explicit bracket at
fixed $I$, evaluated at the substituted action.  Indeed,
$\{F,\pi_j\}_{\rm exp}=D_jF$ by the convention
\eqref{setup:poisson}, so the mixed chain-rule terms give
\eqref{filt:eq8}.  Terms involving two coefficient-action derivatives
vanish because $\{\pi_j,\pi_\ell\}=0$.

The explicit normal degree of a monomial is
$|\alpha|_1+|\beta|_1$; its coefficient $c(I)$ has degree zero.
Both $D_j$ and the coefficient derivatives in $I$ preserve this
degree.  Thus each nonzero bracket term of homogeneous factors has
the sum of their degrees, and terms of degree at least two form a
Poisson ideal for \eqref{filt:eq8}.  The subalgebra with no explicit
$p$-dependence is also closed under this bracket.  We carry out the
finite formal construction in its quotient by that ideal (but not a Poisson ideal for the full physical bracket), retaining
only pure-angle and normal-linear terms.  The discarded terms are
restored in Section~\ref{exact-discarded-terms-and-the-full-trajectory}.

For a monomial $F$ with label $q$, the integrable homological
derivation is

\begin{equation}
\{F,h\}_\times=i d_q(I)F.
 \label{filt:eq10}
\end{equation}

For the auxiliary construction, fix an integer normal cutoff
$R\ge M$.  Retain only the explicit normal-linear indices
$j\le R$ in $g_{01}$, while keeping its entire convergent frozen
angular coefficient $c_{\rm h}^0$.  Define $h_R$ by restricting the
linear sum in \eqref{setup:integrable-hamiltonian} to $A$ and the
normal indices $j\in\mathcal N$, $j\le R$, while retaining the complete
nonlinear part $h_{\rm nl}(I)$.  Keep the full action parameter
$I$ in $h_{\rm nl}$ and in every coefficient.  Then $h_R$ is
holomorphic on the common complex $\ell^2$ neighborhood, and
$Dh_R(I)q=d_q(I)$ for every retained label.  Its action derivatives
of order at least two are those of $h_{\rm nl}$.  Before norm
estimation, replace each bracket $\{h_R,\chi_n\}_\times$ by its
solved homological term from \eqref{filt:eq13}.

At each fixed formal order, the generated pure-angle coefficients
are independent of $R$, and each retained normal-linear coefficient
at index $j$ stabilizes once $R\ge\max\{M,j\}$.  For each admissible
initial datum, the summable bounds in
Section~\ref{analytic-module-bounds} give convergence of the retained
coefficient series and their required action derivatives; the
trajectory limits are justified in
Section~\ref{exact-discarded-terms-and-the-full-trajectory}.
Equation~\eqref{filt:eq10} is thus interpreted coefficientwise as the
full diagonal homological derivation.
\hypertarget{analytic-module-bounds}{%
\subsection{Analytic coefficient bounds}\label{analytic-module-bounds}}

Use the fixed complex action domain $\mathscr D$ and the divisor
hypotheses of Section~\ref{the-all-mode-linear-jet-and-its-error-budget}.
Fix an available angular width $a_0>a>0$.  For the chosen finite order
$L$, reserve finitely many intermediate widths between $a_0$ and $a$
for bracket estimates and explicit angular derivatives.  These widths
are chosen once, independently of $K,M,\varepsilon,\Delta,s$; $a$ is
the common final width and may be reduced once, depending on $L$.

Recall that $\varsigma=\sigma-1$.  Write $\ell^2_{\C}$ for the
complexified action Hilbert space, and take $\mathcal X=\C$ for
scalar coefficients or $\mathcal X=\ell^2_{\C}$ for action-vector
coefficients.  An $\mathcal X$-valued expression of explicit normal
degree at most one, with no explicit $p$-dependence, has the expansion
\[
 F=\sum_k e^{ik\varphi}\left(c_{k,0}(I)
       +\sum_{\iota=\pm1}\sum_{j\in\mathcal N}
                       c_{k,j,\iota}(I)\zeta_j^\iota\right).
\]
Here $k\in\mathbb Z^A$.  For an integer $m\ge0$, define
\[
 \|c\|_{m;\mathcal X}
 :=\sup_{I\in\mathscr D}
       \|D_I^m c(I)\|_{\mathcal L^m(\ell^2_{\C};\mathcal X)},
 \qquad \mathcal L^0(\ell^2_{\C};\mathcal X):=\mathcal X.
\]
Here $\mathcal L^m$ is the space of bounded $m$-linear maps.  For
$a\le r<a_0$, use the seminorms
\begin{equation}
 \begin{split}
 |F|_{r,m;\mathcal X}
 :=\sum_k e^{r|k|_1}\Bigg(&\|c_{k,0}\|_{m;\mathcal X}\\
 &+\sum_{\iota=\pm1}\left(\sum_{j\in\mathcal N}j^{2\varsigma}
             \|c_{k,j,\iota}\|_{m;\mathcal X}^{\,2}\right)^{1/2}\Bigg).
 \end{split}
 \label{filt:eq11}
\end{equation}
Brackets of action-vector
expressions with scalar expressions are taken componentwise.
For the generated corrections $Y_n$ in \eqref{filt:eq14}, a linear
coefficient at index $j$ is supported on the tangential space and
$\operatorname{span}(e_j)$, while a pure-angle coefficient is supported
on the tangential space.  These properties follow from the label-vector
form of $\{I,\chi_n\}_\times$ and the subsequent componentwise brackets.
For $\nu\cdot Y_n$, we use the scalar seminorm after the finite-cutoff
pairing in Section~\ref{finite-formal-normal-form-and-the-scalar-gain}.
Its estimates use the divisor cancellation there.

At formal order $n$, $1\le n\le L$, it suffices to retain action
derivatives
\[
 0\le m\le m_n:=L-n+2.
\]
Thus at most $L+1$ action derivatives of a coefficient and $L+2$
action derivatives of $h_{\rm nl}$ are required.  The full-trajectory
residuals use at most one action derivative of each $Y_n$.  These are
finite analytic derivative bounds already supplied by the hypotheses.
\paragraph{Linearity in the normal coefficients.}
Initially the coefficients $b_{k,j,\iota}$ are independent of $I$.
At every finite order of the quotient construction, each summand in
a retained normal-linear coefficient at $(j,\iota)$ contains exactly
one original factor $b_{k_0,j,\iota}$ with the same normal index and
sign; $k_0$ is a summed input Fourier index.  Generated pure-angle
coefficients contain no such factor.  Coefficient-action
differentiation preserves explicit normal degree, index and sign.
Every bracket retaining degree one combines a degree-zero factor
with a degree-one factor, while label selection and homological
division leave the explicit monomial unchanged.  These observations
prove the invariant by induction.

Consequently, each generated linear coefficient is a finite Fourier
convolution of one original $b_{k_0,j,\iota}$ with pure-angle
coefficients, derivatives of $h_{\rm nl}$ and inverse divisors.
The remaining factors have action-derivative bounds uniform in $j$,
with the divisor and angular losses recorded below.  Since the
original $b$ factor is independent of $I$, these bounds give a
majorant for each coefficient derivative after taking the supremum
over $\mathscr D$.  The original $h^\varsigma$ Fourier norms then
control the weighted normal sums in \eqref{filt:eq11} for these
generated coefficients.

Write $\Pi_{\le1}$ for projection onto explicit normal degrees zero
and one.  All brackets in the next lemma are coefficient-action
brackets, with this projection applied.  For a retained expression $N$ supported
on labels outside $V$, define the homological division operator by
$(\mathcal S N)_q=N_q/(i d_q)$.

\begin{lemma}[Bracket and homological-division estimates]
\label{lem:filtered-module-estimates}
Let $a\le r'<r<a_0$ be two of the reserved angular widths, and let
$m\ge0$ be a fixed integer.  Let $F$ be $\mathcal X$-valued and
$G$ scalar-valued, with explicit normal degree at most one, no
explicit $p$-dependence, and finite seminorms
\eqref{filt:eq11} through order $m+1$.  Then
\begin{equation}\label{filt:bracket-estimate}
 \begin{aligned}
 &|\Pi_{\le1}\{F,G\}_\times|_{r',m;\mathcal X}\\
 &\quad\le C_{m,r-r'}\sum_{\ell=0}^m\binom m\ell
 \Bigl(|F|_{r,\ell+1;\mathcal X}|G|_{r,m-\ell;\C}\\
 &\hspace{43mm}
       +|F|_{r,\ell;\mathcal X}|G|_{r,m-\ell+1;\C}\Bigr).
 \end{aligned}
\end{equation}
For an $\mathcal X$-valued retained expression $N$ whose labels
satisfy the outside-divisor hypotheses, one also has
\begin{equation}\label{filt:division-estimate}
 |\mathcal S N|_{r',m;\mathcal X}
 \le C_{m,r-r'}\sum_{\ell=0}^m
       \Delta^{-(m-\ell+1)}|N|_{r,\ell;\mathcal X}.
\end{equation}
Label selection is a contraction in these seminorms.  The constants
depend only on the indicated derivative order and strip loss, the
fixed tangential dimension, and the standing analytic and separation
bounds; they are independent of $K,M,R,\varepsilon,\Delta,s$.
For expressions generated from $g_{01}$, these operations preserve linearity in the original normal coefficients $b_{k,j,\iota}$ and, for action-vector coefficients, the output
support described above.
\end{lemma}

\begin{proof}
A retained bracket has degree pattern $(0,0)$, $(0,1)$, or $(1,0)$.
In a bracket between a pure-angle term and a normal-linear term at
$(j,\iota)$, the pure-angle coefficient is differentiated in tangential
directions or in $e_j$, while the linear coefficient is differentiated
only in tangential directions.  The full $\ell^2$ coefficient-derivative
norms control these operations uniformly in $j$.  The explicit
derivations contribute at most one factor $1+|k|_1$ from either
input Fourier index.  Fourier convolution, the triangle inequality
in the weighted normal $\ell^2$ sum, and the loss $r-r'>0$ therefore
bound each product.  Leibniz's rule distributes the $m$ further
action derivatives with the orders displayed in
\eqref{filt:bracket-estimate}.  

For an angular label $q=(k,0)$ or normal-linear label $q=(k,\pm e_j)$,
positive-order action derivatives satisfy
$\|D_I^\ell d_q\|\le C_\ell(1+|k|_1)$ for $\ell\ge1$,
uniformly in $j$ on $\mathscr D$.  Indeed, the cubic frequency part
is independent of $I$ and $\|q\|_2\le1+|k|_1$.
The lower divisor bound in \eqref{filt:eq2}, or its high normal-linear
alternative, consequently gives
\begin{equation}
 \|D_I^m(d_q^{-1})\|\le
 C_m(1+|k|_1)^m\Delta^{-(m+1)}.
 \label{filt:eq12}
\end{equation}
For $m=0$ this is the lower divisor bound.  For $m\ge1$, each
reciprocal-derivative term contains at most $m$ positive-order
derivatives of $d_q$ divided by at most $m+1$ powers of $d_q$;
use $0<\Delta\le1$.  Applying Leibniz's rule to $N_q/d_q$ and
absorbing the resulting Fourier powers in $r-r'$ proves
\eqref{filt:division-estimate}.  Selecting labels only removes
nonnegative summands from the coefficient seminorms.

Finally, a retained degree-one bracket contains one degree-one
factor, and coefficient differentiation and division preserve its
normal index and sign.  Componentwise differentiation cannot enlarge
the output support of a vector coefficient.  This proves the last
assertion, including preservation of the original single $b$ factor.
\end{proof}

Only finitely many applications of this lemma are needed at fixed
order $L$.  Choose enough intermediate widths once for these
applications and for the subsequent explicit angular derivatives.
There is no loss of action-domain width: the action derivatives are
estimated directly using the supplied derivative bounds and
\eqref{filt:eq12}.

All estimates may first be proved at the auxiliary normal cutoff
$R$, with constants independent of $R$ and uniform over the admissible
initial data.  For each fixed datum and formal order, the coefficient
at index $j$ stabilizes once $R\ge\max\{M,j\}$.  The same summable
majorants give convergence of the retained coefficient series and
their required action derivatives in these seminorms.

\hypertarget{finite-formal-normal-form-and-the-scalar-gain}{%
\subsection{Finite formal normal form and the scalar gain}\label{finite-formal-normal-form-and-the-scalar-gain}}

Work first at the auxiliary normal cutoff $R\ge M$ of
Section~\ref{the-coefficient-action-poisson-algebra}, suppressing
the cutoff in the notation: until it is removed below, $h$ denotes
$h_R$ and $g_{01}$ retains only the explicit normal-linear indices
$j\le R$, with its full frozen angular coefficient unchanged.
All calculations take place in the normal-degree quotient and are truncated after
$\varepsilon^L$.
At step $n$, the current Hamiltonian is
\[
 h+\sum_{j<n}\varepsilon^j Z_j+\varepsilon^n F_n
       +O(\varepsilon^{n+1}),
\]
where $F_1=g_{01}$ and all labels of each $Z_j$ belong to $V$.
Subscripts $q$ denote label components with their explicit Fourier
and normal factors included.  Split $F_n=Z_n+N_n$ by label membership
in $V$, set $(\chi_n)_q=0$ for $q\in V$, and solve

\begin{equation}
\{\chi_n,h\}_\times=N_n,
 \qquad (\chi_n)_q=(N_n)_q/(i d_q),\quad q\notin V.
 \label{filt:eq13}
\end{equation}

Apply $\exp(\varepsilon^n\mathcal L_{\chi_n})$, where
$\mathcal L_\chi F=\{F,\chi\}_\times$.  The order-$n$ term with labels outside $V$ cancels because $\{h,\chi_n\}_\times=-N_n$.  Conjugation symmetry
of the label sets and divisors preserves reality.  Generated retained
labels have angular size at most $nK$ and normal component zero or
$\pm e_j$, so all divisions lie in the label range fixed in
Section~\ref{the-all-mode-linear-jet-and-its-error-budget}.

Lemma~\ref{lem:filtered-module-estimates} supplies the bracket and
division estimates for the construction.  The induction below gives the derivative and divisor counts for the corrected-observable bounds.

Let $\mathscr T$ be the composition of these finite formal Poisson
automorphisms, with step $1$ acting first and step $L$ last.  Then

\[
 \mathscr T(h+\varepsilon g_{01})
   =h+\sum_{n=1}^L\varepsilon^n Z_n
       +O(\varepsilon^{L+1})
\]

in the degree-at-most-one quotient. Set

\begin{equation}
J^{[L]}=\operatorname{trunc}_{\le L}(\mathscr T^{-1}I)
       =I+\sum_{n=1}^L\varepsilon^nY_n.
 \label{filt:eq14}
\end{equation}

For the scalar estimate, pair first with $\nu_R=Dh_R(I^0)$.
This is an $\ell^2$ vector and agrees with the full frozen frequency
$\nu$ on the output
coordinates of every retained correction, which are supported on
$A$ and normal indices $j\le R$.  Thus the innermost pairing is

\begin{equation}
\nu\cdot\{I,\chi_j\}_\times
 =-i\sum_{q\notin V}(\nu\cdot q)(\chi_j)_q.
 \label{filt:eq15}
\end{equation}

Here $q$ is the label of the complete generated component
$(\chi_j)_q$, whose own homological divisor appears in
\eqref{filt:eq13}.  The bound $|\nu\cdot q/d_q|\le C$ cancels that
divisor: $C=2$ under \eqref{filt:eq2}, and the high normal-linear alternative
supplies a fixed bound.  For $\vartheta_q=(\nu\cdot q)/d_q$,
\[
 D_I\vartheta_q[v]=-\vartheta_q\,\frac{D_Id_q[v]}{d_q};
\]
$\nu\cdot q$ is frozen at $I^0$, so each derivative costs at most one
further inverse divisor while preserving the initial cancellation.
Moreover, for any action-vector expression $U$, the frozen
contraction commutes with every componentwise outer bracket:
$\nu_R\cdot\{U,\chi\}_\times=\{\nu_R\cdot U,\chi\}_\times$.
Hence the nested scalar brackets preserve the one-divisor gain.  The complete finite-order
bounds on the common final strip are
\begin{equation}
 \begin{aligned}
 |F_n|_{a,m;\C}
   &\le C_{n,m}\Delta^{-2(n-1)-m},\\
 |\chi_n|_{a,m;\C}+|Y_n|_{a,m;\ell^2_{\C}}
   &\le C_{n,m}\Delta^{-(2n-1)-m},\\
 |\nu\cdot Y_n|_{a,m;\C}
   &\le C_{n,m}\Delta^{-(2n-2)-m}.
 \end{aligned}
 \label{filt:eq16}
\end{equation}
These hold for $1\le n\le L$ and $0\le m\le m_n=L-n+2$.
Table~\ref{tab:filtered-induction} also records the terminal retained
remainder
$R_L=\Pi_{\le1}\{Y_L,g_{01}\}_\times$, which appears with
prefactor $\varepsilon^{L+1}$ in \eqref{filt:eq18}.  Its scalar
version is formed by the frozen pairing before removal of the cutoff.

\begin{table}[ht]
\centering
\small
\begin{tabular}{@{}llll@{}}
\toprule
Quantity & Formal prefactor & Coefficient-seminorm bound & Action derivatives\\
\midrule
$F_n$ & $\varepsilon^n$ & $C_{n,m}\Delta^{-2(n-1)-m}$ & $0\le m\le m_n$\\
$\chi_n$ & $\varepsilon^n$ & $C_{n,m}\Delta^{-(2n-1)-m}$ & $0\le m\le m_n$\\
$Y_n$ & $\varepsilon^n$ & $C_{n,m}\Delta^{-(2n-1)-m}$ & $0\le m\le m_n$\\
$\nu\cdot Y_n$ & $\varepsilon^n$ & $C_{n,m}\Delta^{-(2n-2)-m}$ & $0\le m\le m_n$\\
$R_L$ & $\varepsilon^{L+1}$ & $C_L\Delta^{-2L}$ & $m=0$\\
$\nu\cdot R_L$ & $\varepsilon^{L+1}$ & $C_L\Delta^{-(2L-1)}$ & $m=0$\\
\bottomrule
\end{tabular}
\caption{Finite-order bookkeeping on the common final angular strip.
The output norm is $\ell^2_{\C}$ for $Y_n,R_L$ and scalar for
the other rows.  The prefactor for $\chi_n$ refers to its Lie step.}
\label{tab:filtered-induction}
\end{table}

We establish the bounds in Table \ref{tab:filtered-induction}
 by induction on $n$.  Initially $F_1=g_{01}$
has a uniform coefficient bound and no action dependence.  For a
bracket between an order-$i$ Hamiltonian term and an order-$j$
generator, with $i+j=n$, \eqref{filt:bracket-estimate} gives the
denominator exponent
\[
 2(i-1)+(2j-1)+1+m=2(n-1)+m.
\]
The extra $1$ is the action derivative in the bracket.  Homological
division adds one inverse power at fixed derivative order, giving
the $\chi_n$ row by \eqref{filt:division-estimate}.  Every bracket
with $h$ is first replaced by its solved homological term from
\eqref{filt:eq13}.

In the inverse composition defining $Y_n$, the innermost bracket
$\{I,\chi_j\}_\times$ contributes a label but no coefficient-action
derivative.  Its label factor is absorbed by an angular-strip loss.
Each outer bracket contributes one action derivative and adds the
order of its generator.  Thus the exponent is $2n-1+m$.
The cancellation of the divisor associated with the generated label in
\eqref{filt:eq15} lowers that exponent by one for $\nu\cdot Y_n$, and the frozen contraction commutes
with every outer bracket.  For $i,j\ge1$, $i+j=n\le L$, and
$m\le m_n$, the inequality $m+1\le m_i,m_j$ verifies that every
input derivative required by the bracket estimate remains available.
This proves \eqref{filt:eq16} with the stated derivative order $m_n$.

Finally, $D_Ig_{01}=0$, so $R_L$ uses one action derivative of
$Y_L$, and its scalar version uses one action derivative of
$\nu\cdot Y_L$.  The last two rows follow from
\eqref{filt:bracket-estimate}.  On $W\le s\le1$, coefficient
seminorms bound the evaluated retained expressions.  Integration
over $T\le T_0/\varepsilon$ therefore leaves terminal errors
$C_L\varepsilon^L\Delta^{-2L}$ and
$C_L\varepsilon^L\Delta^{-(2L-1)}$, respectively, as in
\eqref{filt:eq6}--\eqref{filt:eq7}.
All constants may depend on $L,m$, the reserved widths and the fixed
analytic, separation and ratio bounds in
Proposition~\ref{prop:filtered}; they are independent of
$K,M,\varepsilon,\Delta,s$ and $R$.

We may now remove the auxiliary cutoff.  Pure-angle coefficients
are independent of $R$, and each linear coefficient at index $j$
stabilizes once $R\ge\max\{M,j\}$.  For each fixed admissible datum,
the summable majorants of Section~\ref{analytic-module-bounds}
give convergence of the scalar coefficient series and their required
action derivatives in the seminorms \eqref{filt:eq11}.  The same
bounds justify the limit of \eqref{filt:eq15}.
Section~\ref{exact-discarded-terms-and-the-full-trajectory} identifies
these scalar observables with the actual frequency pairings along
the trajectory.

For example, when  $L\ge2$, the first Lie step in the quotient gives
\[
 F_2=\{Z_1,\chi_1\}_\times
            +\tfrac12\{N_1,\chi_1\}_\times,
\]
and the first two corrections in \eqref{filt:eq14} are
\[
 \begin{aligned}
 Y_1&=-\{I,\chi_1\}_\times,\\
 Y_2&=-\{I,\chi_2\}_\times
       +\tfrac12\{\{I,\chi_1\}_\times,\chi_1\}_\times.
 \end{aligned}
\]
Poisson covariance at finite cutoff, followed by the coefficient
limit, gives the identity below, modulo $\varepsilon^{L+1}$ and
normal degree at least two:

\begin{equation}
\{\mathscr T^{-1}I,h+\varepsilon g_{01}\}_\times
       =\mathscr T^{-1}\{I,\mathscr T(h+\varepsilon g_{01})\}_\times.
 \label{filt:eq17}
\end{equation}

Each vector $\{I,Z_n\}_\times$ is a linear combination of fixed
labels in $V$ with scalar coefficients, so it takes values in the
complexification $V_{\C}$.  Every subsequent scalar Poisson
derivation acts componentwise and preserves this output subspace.
Thus all retained coefficients through order $L$ on the left side
of \eqref{filt:eq17} take values in $V_{\C}$.  By reality, these
terms give the $V$-valued observable $\mathcal V^{[L]}$ in
\eqref{filt:eq18} on the real phase space.
No divisor associated with a label in $V$ is inverted, including
at exact higher resonances.

\hypertarget{exact-discarded-terms-and-the-full-trajectory}{%
\subsection{Restoring discarded terms along the supplied trajectory}\label{exact-discarded-terms-and-the-full-trajectory}}

Let $Y=J^{[L]}-I$, using the degree-at-most-one representatives
$Y_n$ from \eqref{filt:eq14}.  Along the trajectory on $[0,T]$, write
$Y(b(t))=Y(I(t),\varphi(t),z(t))$, with
$I(t)=\xi+\pi(b(t))$.  Coefficient and explicit derivatives are
taken before this substitution.  Restoring the discarded explicit
normal degrees in the coefficient-action bracket gives the exact identity
below; the stated membership in $V$ is on the real phase space.

\begin{equation}
\{J^{[L]},h+\varepsilon g_{01}\}_\times
 =\mathcal V^{[L]}+\varepsilon^{L+1}R_L
       +\sum_{n=1}^L\varepsilon^{n+1}T_n,
 \qquad \mathcal V^{[L]}\in V.
 \label{filt:eq18}
\end{equation}

Here $R_L$ is the degree-at-most-one part of
$\{Y_L,g_{01}\}_\times$, and $T_n$ is the degree-two part of
$\{Y_n,g_{01}\}_\times$.  The retained terms through order $L$
form $\mathcal V^{[L]}$.  These are all remaining terms: the bracket
has perturbation order at most $L+1$ and explicit normal degree at
most two, while bracketing with $h$ preserves this degree.
By \eqref{filt:eq16}, the evaluated expressions on $W\le s\le1$
satisfy the following bounds, with $\|\cdot\|$ denoting the output
$\ell^2$ norm and $1\le n\le L$ in \eqref{filt:eq19}.
The scalar pairings are justified below.

\[
 \|R_L\|\le C_L\Delta^{-2L},\quad
 |\nu\cdot R_L|\le C_L\Delta^{-(2L-1)},
\]
\begin{equation}
\|T_n\|\le C_Ls\Delta^{-2n},\quad
 |\nu\cdot T_n|\le C_Ls\Delta^{-(2n-1)}.
 \label{filt:eq19}
\end{equation}

Since $D_Ig_{01}=0$, the quadratic term $T_n$ is obtained by
applying the action derivative of the normal-linear part of $Y_n$
to the normal-linear action drift of $g_{01}$.  The two factors
are bounded by $C_L\Delta^{-2n}\|z\|_\sigma$ and
$C\|z\|_\sigma$, respectively, giving the factor $s$ in
\eqref{filt:eq19}.  The scalar coefficient estimate gains one inverse
power of $\Delta$.

Write \(a_{01}=D\pi X_{g_{01}}\) and \(a_*=a_f-a_{01}\). The all-mode Taylor expansion gives

\begin{equation}
\|a_*\|_2\le C(s+e_K).
 \label{filt:eq20}
\end{equation}

For tangential components, removing the original pure-angle and
normal-linear terms leaves $O(|p|_1+\|z\|_\sigma^2)$.
The frozen angular addition contributes a tangential term bounded
by $CW_0\le Cs$, by \eqref{rw:frozen-diagonal-bound}.
For normal components, the remaining velocity is
$O(|p|_1+\|z\|_\sigma)$ in $h^{\sigma-1}$.  Multiplication by
$z$ and \eqref{setup:weighted-product} give an $O(s)$ action
drift.  The Fourier tail adds $O(e_K)$.  At the selected angular
modes, $g_{01}$ retains every normal-linear index, so no remainder
arises from a cutoff in the normal index.

The explicit derivative at fixed \(I\) is

\begin{equation}
D_{\rm exp}Y[X_f]
  =D_\varphi Y[f_p]+D_zY[X_f^z].
 \label{filt:eq21}
\end{equation}

There is no explicit $p$ dependence in $Y$.  For $Y_n$, this
derivative is bounded by $C_L\Delta^{-(2n-1)}$, and its scalar
pairing by $C_L\Delta^{-(2n-2)}$.  The angular derivative uses
a reserved strip reduction; the normal derivative acts on a linear
expression.  Neither differentiates a coefficient with respect to $I$.
The full $f_p$ and $X_f^z\in h^{\sigma-1}$ are bounded by
Assumption~\ref{ass:contract}.

To justify the chain rules, first use finite explicit normal sums
along the same full trajectory, keeping the full $I(t)$ in every
coefficient.  Lemma~\ref{lem:action-chain} gives
$I\in C^1(\ell^2)$ and $I'=\varepsilon a_f$.
A coefficient in $h^{\sigma-1}$ pairs continuously with
$z'\in h^{\sigma-3}\subset h^{1-\sigma}$, since $\sigma\ge2$;
the normal output components are bounded coordinatewise products
in $\ell^2$.  The coefficient convergence from
Section~\ref{analytic-module-bounds}, including first action
derivatives, therefore gives uniform convergence of the observables
and their derivatives on each fixed  interval $[0,T]$.  Their limits
are $C^1$.  The integrable contribution is evaluated by
\eqref{filt:eq10} along the supplied trajectory.

The weighted normal space has norm
$\|v\|_{\ell^1_3}=\sum_{j\in\mathcal N}j^3|v_j|$.
For a linear coefficient sequence $c\in h^{\sigma-1}$,
\[
 \sum_jj^3|c_j\zeta_j|
 \le \|c\|_{\sigma-1}\|\zeta\|_{4-\sigma}
 \le C\|c\|_{\sigma-1}\|z\|_\sigma,
\]
since $\sigma\ge2$.  For the generated corrections, the support
property in Section~\ref{analytic-module-bounds} places the normal
output of a linear coefficient at its explicit index $j$; action
derivatives preserve this support.  Applying the displayed estimate
to the coefficient norms therefore gives normal outputs in
$\ell^1_3$.  For $T_n$, apply it to the first action derivative of
the linear coefficient and multiply by the $\ell^2$ norm of the
linear part of $a_{01}$.  This uses operator norms and gives the
same weighted summability for the quadratic terms.

The estimate with $z'$ holds because
$h^{\sigma-3}\subset h^{4-\sigma}$ for $\sigma\ge7/2$,
including equality at the endpoint.  The finite sums for $Y_n$ and
their time derivatives therefore also converge in the weighted action
space $\mathcal A_3$.  Since $|\nu_j|\le Cj^3$ on the normal
coordinates, the generated observables and residuals have absolutely
convergent frequency pairings, equal to the scalar coefficient
limits.  Their sharper bounds use the same divisor cancellation.
Moreover, \eqref{setup:weighted-product} gives
$a_f,a_{01},a_*\in\mathcal A_{2\sigma-1}\subset\mathcal A_3$,
so $\nu\cdot a_*$ is also well defined.

The integrable vector field fixes every action, and its explicit
contribution at fixed $I$ is
\[
 D_{\rm exp}Y[X_h]
 =D_\varphi Y[\omega(I)]+D_zY[X_h^z]
 =\{Y,h\}_\times.
\]
Together with $I'=\varepsilon a_f$ and
$\{Y,g_{01}\}_\times=D_IY[a_{01}]$ (because $D_Ig_{01}=0$),
this gives the full physical chain rule
\[
 \frac{d}{dt}J^{[L]}
   =\{J^{[L]},h+\varepsilon g_{01}\}_\times
     +\varepsilon a_*+\varepsilon D_IY[a_*]
     +\varepsilon D_{\rm exp}Y[X_f].
\]
Inserting \eqref{filt:eq18} gives the exact full-trajectory identity

\begin{equation}
\frac{d}{dt}J^{[L]}
 =\mathcal V^{[L]}+\varepsilon a_*
 +\varepsilon^{L+1}R_L
 +\sum_{n=1}^L\varepsilon^{n+1}T_n
 +\varepsilon D_IY[a_*]
 +\varepsilon D_{\rm exp}Y[X_f].
 \label{filt:eq22}
\end{equation}

Integrating \eqref{filt:eq22}, with $I(0)=I^0$, gives
$\mathcal D(t)$ and the full endpoint difference
$Y(b(t))-Y(b(0))$.  Define $E_{\rm av}^{[L]}$ as the negative
of this endpoint difference plus the integrals of the
$R_L,T_n,D_IY[a_*]$ and $D_{\rm exp}Y[X_f]$ terms, with their
displayed powers of $\varepsilon$.  The integral of
$\mathcal V^{[L]}$ is $V$-valued.  The remaining integral 
$\varepsilon\int_0^t a_*$ is evaluated in
Section~\ref{the-retained-quadratic-integral}.

For $0\le t\le T\le T_0/\varepsilon$, the endpoints and explicit
derivative terms contribute
$C_L\sum_{n=1}^L\varepsilon^n\Delta^{-(2n-1)}$.
The quadratic and $D_IY[a_*]$ terms contribute
$C_L\sum_{n=1}^L\varepsilon^n(s+e_K)\Delta^{-2n}$,
and the terminal residual contributes
$C_L\varepsilon^L\Delta^{-2L}$.  The same calculation with the
frozen pairing $\nu=Dh(I^0)$ gains one inverse power in each term.
These are \eqref{filt:eq6}--\eqref{filt:eq7}.

\hypertarget{the-retained-quadratic-integral}{%
\subsection{The retained quadratic integral}\label{the-retained-quadratic-integral}}

We now evaluate the integral of $a_*$ retained in the preceding
subsection.  The augmented window identity \eqref{rw:alllinearidentity} is
\[
 \mathcal D(t)=\varepsilon\int_0^t a_{g_{01}}\,dv
       +\varepsilon\int_0^t a_{g_2^{\rm low}}\,dv+E_0(t).
\]
The error $E_0$ satisfies \eqref{rw:e0budget}.  Since
$\mathcal D(t)=\varepsilon\int_0^t a_f$, subtraction yields
\begin{equation}
 \varepsilon\int_0^t a_*\,dv
       =\varepsilon\int_0^t D\pi X_{g_2^{\rm low}}\,dv+E_0(t).
 \label{filt:eq23}
\end{equation}
Here $\mathcal D(t)=I(t)-I^0$, with the initial normal actions
included in $I^0$.

The jet $g_2^{\rm low}$ consists of terms linear in $p$ and
quadratic in $z$, with Taylor coefficients fixed at $\xi$ and
independent of the running action $I$.  Treat $I$ and the explicit
variables $(\varphi,p,z)$ as independent.  Retain all labels in $V$;
every outside label is covered by \eqref{filt:eq2}.
For each outside monomial $m_q$, with its coefficient included,
define $a_q=-iqm_q$ and set
\[
 U_2=\sum_{q\notin V}\frac{a_q}{i d_q(I)}.
\]
The sum runs over individual monomials, so labels may repeat.
At fixed $I$, the integrable rotation fixes $p$ and differentiates
$m_q$ to $i d_q(I)m_q$.  Thus its explicit derivative of $U_2$
is $a_{2,\mathrm{out}}:=\sum_{q\notin V}a_q$.

To track the cutoff dependence, let $C_k$ be the $k$th Fourier
coefficient of the original normal Hessian $D_z\nabla_zf(\varphi,0,0)$
from \eqref{rw:normaljet}.
Assumption~\ref{ass:contract} bounds this operator from $h^\sigma$
to $h^{\sigma-1}$.  Hessian symmetry gives $C_k^T=C_k$, even
when $C_k$ is complex.  Testing on coordinate vectors and using
this symmetry gives, entrywise in the Cartesian blocks at modes $i,j$,
\[
 |(C_k)_{ij}|\le C e^{-a_0|k|_1}
 \min\{i^{1-\sigma}j^\sigma,j^{1-\sigma}i^\sigma\}
 \le C e^{-a_0|k|_1}\min(i,j),
\]
where $a_0>a$ is a fixed available strip width.  For $i\le j$,
$j^{1-\sigma}i^\sigma=i(i/j)^{\sigma-1}\le i$, proving the
last inequality for $\sigma\ge1$.
There are $O(M^2)$ normal quadratic types per Fourier mode, and
summing $\min(i,j)$ over $i,j\le M$ costs at most $C(1+M)^3$.
The finitely many $p$-linear coefficients fit the same bound.
Passing to complex coordinates changes only fixed constants.
All these absolute sums are taken on the finite window.
Since $\|q\|_2\le |k|_1+2$, label factors and angular derivatives
cost only fixed powers of $|k|_1+2$, which are summable against
the exponential decay provided by the unused strip width $a_0-a$, uniformly in $K$.
Selecting any fixed set of labels can only decrease the absolute sum.

On $W\le s\le1$, with the coefficients factored out, each
monomial has size $O(s)$, its $p$ derivative has size $O(1)$,
and its $z$ derivative on $h^{\sigma-1}$ has size $O(\sqrt s)$.
Only the inverse divisors depend on $I$.
The Hessian bound \eqref{geo:hessian} gives
$\|D_Id_q\|\le C(|k|_1+2)$, so differentiating an inverse
divisor costs one further factor $\Delta^{-1}$ and an angular-label
factor already covered by the strip reduction.  With
$\mathfrak p(M)=(1+M)^3$, this gives

\begin{equation}
 \begin{aligned}
 \|U_2\|&\le C\mathfrak p(M)s/\Delta,
 &\quad \|D_IU_2\|&\le C\mathfrak p(M)s/\Delta^2,\\
 \|D_pU_2\|&\le C\mathfrak p(M)/\Delta,
 &\quad \|D_zU_2\|&\le C\mathfrak p(M)\sqrt s/\Delta.
 \end{aligned}
 \label{filt:eq24}
\end{equation}

In \eqref{filt:eq24}, the output norm is $\ell^2$, and the derivative
norms are operator norms for $D_IU_2:\ell^2\to\ell^2$,
$D_zU_2:h^{\sigma-1}\to\ell^2$, and $D_pU_2:\R^A\to\ell^2$,
with the same finite-dimensional domain convention for $D_\varphi U_2$.
The $\varphi$ derivative has the bound $C\mathfrak p(M)s/\Delta$
as well.  The frequency pairings are finite sums, since the normal
output is supported at indices at most $M$.  For the sharper scalar
estimate, use
\[
 \nu\cdot U_2=-\sum_{q\notin V}\frac{\nu\cdot q}{d_q(I)}m_q.
\]
By \eqref{filt:eq2}, $|\nu\cdot q/d_q|\le2$.  Since $\nu=Dh(I^0)$
is frozen, the first action derivative of this ratio is bounded by
$C(|k|_1+2)/\Delta$.  Hence each scalar analogue of
\eqref{filt:eq24}, including the $\varphi$ derivative, gains one
power of $\Delta$.

Write $U_2(t)=U_2(I(t),\varphi(t),p(t),z(t))$ and set
$E_2(t)=\varepsilon\int_0^t a_{2,\mathrm{out}}\,du$.
All partial derivatives are taken before substituting
$I=\xi+\pi(b)$.  The exact full-trajectory chain rule is
\begin{equation}\label{filt:quadraticchain}
 \varepsilon\int_0^t a_{2,\mathrm{out}}\,du
 =\varepsilon\bigl(U_2(t)-U_2(0)\bigr)
 -\varepsilon^2\int_0^t\mathcal R_2\,du,
\end{equation}
where
\begin{equation}\label{filt:quadraticresidual}
 \mathcal R_2=D_\varphi U_2[f_p]+D_pU_2[-f_\varphi]
 +D_zU_2[X_f^z]+D_IU_2[a_f].
\end{equation}
The finite explicit coordinates are $C^1$, and $I\in C^1(\ell^2)$
by Lemma~\ref{lem:action-chain}.  The $p$ derivative differentiates
the explicit monomials, while the full action derivative differentiates
their divisors; both contribute along the physical trajectory.
Assumption~\ref{ass:contract} and \eqref{setup:weighted-product}
bound the full quantities $f_p,f_\varphi,X_f^z,a_f$ in their
stated norms.  Using $s\le1$ gives
\[
 \|\mathcal R_2\|_2\le C\mathfrak p(M)(\Delta^{-1}+s\Delta^{-2}),
 \qquad |\nu\cdot\mathcal R_2|\le C\mathfrak p(M)(1+s\Delta^{-1}).
\]
Both endpoints are bounded by \eqref{filt:eq24} and its scalar
analogue.  Together with $\varepsilon^2t\le T_0\varepsilon$,
these estimates prove \eqref{filt:eq5}.  In particular, the
$D_pU_2$ term contributes $\varepsilon/\Delta$.

Since $V$ is a fixed real subspace, the selection of labels in $V$ is
invariant under conjugation. By Lemma~\ref{geo:lem-signed}, the resulting real
action drift takes values in $V$.
Adding its $\varepsilon$-weighted integral to the integral of
$\mathcal V^{[L]}$ defines $A_V$.  Combining \eqref{filt:eq23}
with the integrated identity \eqref{filt:eq22} gives
\eqref{filt:eq3} and completes the proof of
Proposition~\ref{prop:filtered}.
% END sections/05_filtered
% BEGIN sections/06_closure
\section{Two-parameter confinement and longer intervals}\label{sec:closure}

The window, geometry, and observable estimates of
Sections~\ref{sec:window}--\ref{sec:filtered} now close the two-parameter
bootstrap. The finite-gap projection $\xi=(I_A(0),0)$ fixes the chart;
the integrable energy, frozen frequencies $\nu=Dh(I^0)$, and resonance
hierarchy are centered at the actual initial action $I^0$. Constants
remain uniform over the admissible initial tails, and the averaging order
is fixed in terms of the target exponent.

Subsection~\ref{subsec:closure-fixed} proves Theorem~\ref{thm:main} using
$\lambda=\max\{\eps,r_0^{1/\delta}\}$ to choose the cutoffs.
Subsection~\ref{subsec:long-time-budgets} then makes the time dependence
explicit under separately supplied uniform long-interval bounds.
These budgets give the logarithmic extension in
Subsection~\ref{subsec:long-logarithmic} with the original
$h^{\sigma+1}$ bound. An additional $h^{\sigma+2}$ bound improves the
normal interpolation and gives the polynomial extension in
Subsection~\ref{subsec:long-polynomial}.

\subsection{The fixed slow-time theorem}\label{subsec:closure-fixed}

\begin{proof}[Proof of Theorem~\ref{thm:main}]
Fix $0<\delta<1/4$.  Introduce the auxiliary scale
\begin{equation}\label{cl:effective-scale}
 \lambda=\max\{\eps,r_0^{1/\delta}\},\qquad
 W_0=r_0^2\le\lambda^{2\delta},\qquad
 \lambda^{2\delta}=\max\{\eps^{2\delta},r_0^2\}.
\end{equation}
We use $\lambda$ only to set the cutoff scales $K$, $M$ and the bootstrap radii below.
Choose
\begin{equation}\label{cl:order}
 2\delta<b<\min\left\{\frac12,\frac52\delta,
                                \frac{1+2\delta}{3}\right\},
 \qquad L\ge1,\qquad L(1-2b)>b.
\end{equation}
The three upper endpoints exceed $2\delta$ by
$1/2-2\delta$, $\delta/2$, and $(1-4\delta)/3$, respectively.
Thus such a $b$ exists, and then a finite integer $L$ can be chosen.
Both are fixed before $r_0$ and $\eps$ are made small.

Set $K=\lceil c\log(1/\lambda)\rceil$, with $c$ large enough that
$e_K=e^{-aK}\le\lambda$, using the final angular width for this
fixed $L$.  Choose $M\asymp\sqrt K$ so that $M^2\ge D(K+1)$, where $D$ is as in
Lemma~\ref{rw:lem-separation}. This choice suffices for the window and
transfer estimates.

Include in the finite label family all original low quadratic
labels and all angular and normal-linear labels with $|k|_1\le LK$ and
normal index at most $M$.  The extra frozen diagonal
$c_{\rm h}^0$ has only angular labels already in this family.
Apply Proposition~\ref{prop:hierarchy} at $I^0$ with
$\alpha=\lambda^b$.  The normal dimension is $O(\sqrt K)$;
the conditioning constant and adjacent-scale ratio are polynomial
in $K$ for fixed $d,L$.  In particular,
\begin{equation}\label{cl:subpower}
 \exp\bigl(O(\sqrt K\log(2+K))\bigr)=\lambda^{-o(1)}.
\end{equation}
The logarithm of any such factor divided by $\log(1/\lambda)$
tends to zero, uniformly over the selected branch and initial data.
Write $V$, $\rho$ and $\Delta$ for the resulting resonance space,
action-increment radius and outside-divisor threshold.  Then
\begin{equation}\label{cl:scales}
 c_2\lambda^b\le\rho,\Delta\le\lambda^{b-o(1)},\qquad
 d_V:=\|\nu|_V\|_{V^*}\le\frac{\mu\rho}{8},\qquad
 \nu=Dh(I^0).
\end{equation}
Use $d_{\{0\}}=0$.  Since
$\|I^0-\xi\|_2\le r_0^2/2$ and $\rho\to0$, the common balls
$\|I-I^0\|_2\le2\rho$ lie in the fixed analytic and curvature
neighborhood of the finite-gap family.  Proposition~\ref{prop:hierarchy}
and Corollary~\ref{geo:cor-highunary} supply, respectively, the low
divisor hypotheses and the high normal-linear cubic and ratio bounds of
Proposition~\ref{prop:filtered} on these balls.  These estimates are uniform over the centers $I^0$, since the
analytic bounds hold on a common neighborhood of the compact finite-gap bases.

Let $C_{\rm tr}\ge1$ be the cutoff-independent constant in
\eqref{rw:transfer}, and choose $C_s=8C_{\rm tr}$.  Put
\begin{equation}\label{cl:stop-radius}
 s=C_s\lambda^{2\delta}.
\end{equation}
Stop the supplied Cartesian path at the minimum of $T_0/\eps$,
the first hitting times $B=\rho$ and $W=s$, and the first exit from
a smaller fixed positive-action chart.  Estimates are first made on
closed intervals strictly before this minimum.  Initially $B=0$
and $W=W_0\le s/C_s$.  Also $\rho=o(s)$ because $b>2\delta$,
and $s$ is below all fixed local thresholds for small $\lambda$.

Proposition~\ref{prop:filtered}, using the augmented angular/linear
jet, gives the exact increment decomposition
$\mathcal D=A_V+E_0+E_2+E_{\rm av}^{[L]}$.  On the interval $[0,T]$ set
\begin{equation}\label{cl:errors}
 \eta=\sup\|E_0+E_2+E_{\rm av}^{[L]}\|_2,\qquad
 \tau=\sup|\nu\cdot(E_0+E_2+E_{\rm av}^{[L]})|.
\end{equation}
We estimate these quantities relative to $\rho$ and $\rho^2$,
respectively.  For example, the ordinary order-$n$ terms give
\[
 \frac{\eps^n}{\Delta^{2n-1}\rho},\qquad
 \frac{\eps^n}{\Delta^{2n-2}\rho^2}
       \le\lambda^{n(1-2b)-o(1)},
\]
whereas the corresponding terms containing $s$ are at most
$\lambda^{n(1-2b)+2\delta-b-o(1)}$ after either normalization.
The full list is recorded in Table~\ref{tab:near-exponents}.
Every exponent in the table denotes an upper bound
$C\lambda^{\text{exponent}-o(1)}$; fixed cutoff polynomials are
included in the subpower factor.

\begin{table}[ht]
\centering
\small
\begin{tabular}{@{}lll@{}}
\toprule
Contribution & Vector divided by $\rho$ & Scalar divided by $\rho^2$\\
\midrule
Order-$n$ ordinary terms & $n(1-2b)$ & $n(1-2b)$\\
Order-$n$ terms with $s$ & $n(1-2b)+2\delta-b$ & $n(1-2b)+2\delta-b$\\
Order-$n$ Fourier tails & $n+1-(2n+1)b$ & $n+1-(2n+1)b$\\
Terminal remainder & $L-(2L+1)b$ & $L-(2L+1)b$\\
Finite quadratic terms & $1-2b,\ 1+2\delta-3b$ & $1-2b,\ 1+2\delta-3b$\\
Raw $\eps,e_K$ terms & $1-b$ & $1-2b$\\
Raw Taylor terms & $3\delta-b$ & $5\delta-2b$\\
Energy error & --- & $1-2b$\\
\bottomrule
\end{tabular}
\caption{Normalized error exponents for the two-parameter proof.}
\label{tab:near-exponents}
\end{table}

In particular the terminal terms satisfy
\begin{equation}\label{cl:terminal}
 \frac{\eps^L}{\Delta^{2L}\rho}
 +\frac{\eps^L}{\Delta^{2L-1}\rho^2}
       \le C\lambda^{L-(2L+1)b-o(1)}=o(1).
\end{equation}
Since $T\le T_0/\eps$, the integrated prefactor satisfies
$\eps^{n+1}T\le T_0\eps^n\le T_0\lambda^n$. All exponents are strictly positive by
\eqref{cl:order}.  For the terms containing $s$, the worst order is
$n=1$, with exponent $1+2\delta-3b>0$.
For the raw terms $s^{3/2}/\rho$ and $s^{5/2}/\rho^2$, the exponents
are $3\delta-b>0$ and $5\delta-2b>0$.
One common positive margin is
\[
 \min\{1-2b,\ 1+2\delta-3b,\ 3\delta-b,
                    5\delta-2b,\ L-(2L+1)b\}>0.
\]
The same bounds are uniform for all shorter time intervals.
We conclude
\begin{equation}\label{cl:smallerrors}
 \eta=o(\rho),\qquad \tau+2F_0\eps=o(\rho^2)
          \qquad(\lambda\to0).
\end{equation}

The energy is centered at $I^0$.  Lemma~\ref{lem:energyclosure}
therefore gives
\begin{equation}\label{cl:energy}
 \frac\mu2B^2\le2F_0\eps+d_V(B+\eta)+\tau.
\end{equation}
Using \eqref{cl:scales} and \eqref{cl:smallerrors}, its positive root
yields $B\le\rho/4+o(\rho)\le\rho/2$ for small $\lambda$.
Since the bound holds at every shorter time, the transfer estimate
applies with $\mathcal B_t\le\rho/2$ and gives
\begin{equation}\label{cl:high}
 \mathcal W_t\le C_{\rm tr}
 \left[\lambda^{2\delta}+K_{S_M}\rho/2+(2\lambda)^{2/3}\right]
       \le s/2.
\end{equation}
Indeed, $K_{S_M}=O((1+M)^{2\sigma+1/2})$ is polynomial,
$K_{S_M}\rho=o(\lambda^{2\delta})$ because $b>2\delta$, and
$\lambda^{2/3}=o(\lambda^{2\delta})$ because $\delta<1/4$.

Both bootstrap bounds are strictly improved.  The quantities $B$ and
$W$ are continuous in the supplied Cartesian topology: the quadratic
action map is continuous in its weighted $\ell^1$ norm.  By continuity, the same strict bounds hold at any
candidate first stopping time.  The positive tangential margin and
$|p|_1+\|z\|_\sigma^2\le s\to0$ exclude exit from the smaller
fixed chart as well.  Thus the estimates extend to the full supplied
interval.
We have proved
\begin{equation}\label{cl:finalW}
 \sup_{0\le t\le T_0/\eps}W(t)
       \le C_s\max\{r_0^2,\eps^{2\delta}\}
       \le C_s(r_0^2+\eps^{2\delta}).
\end{equation}
Choose $\lambda_\delta>0$ so all preceding smallness conditions hold,
then take $\eps_\delta\le\lambda_\delta$ and
$r_\delta\le\min\{r_{\rm adm},\lambda_\delta^\delta\}$.

For the weighted increments, nonnegativity gives
\[
 \sum_{j\in\mathcal N}j^{2\sigma}|J_j(t)-J_j^0|
       \le\tfrac12\bigl(\|z(t)\|_\sigma^2+W_0\bigr).
\]
Together with \eqref{cl:finalW}, this proves
\eqref{setup:near-increments}.  To compare spectra, choose any
$w_\xi$ on the finite-gap torus with actions $\xi$.  Apply
Lemma~\ref{lem:spectral} separately to $w(t)$ and $w(0)$, relative
to this same torus.  The triangle inequality gives
\[
 \|\gamma(w(t))-\gamma(w(0))\|_{\ell^2}^2
 \le2\|\gamma(w(t))-\gamma(w_\xi)\|_{\ell^2}^2
      +2\|\gamma(w(0))-\gamma(w_\xi)\|_{\ell^2}^2
 \le2C_\gamma\bigl(W(t)+W_0\bigr).
\]
This proves \eqref{setup:near-gaps} after
fixing a common $C_\delta$.  The empty tangential set is covered by
the same argument with its zero-dimensional conventions.
\end{proof}

\subsection{Time-explicit estimates on supplied intervals}
\label{subsec:long-time-budgets}

To extend the interval, one must track both the integrated remainders and
 the higher-weight transfer.  The next results are conditional on the following \emph{uniform long-interval hypotheses}: replace $[0,T_0/\eps]$ in
Assumption~\ref{ass:contract} by a supplied interval $[0,T_\eps]$, possibly
also depending on the admissible initial datum.  Require the same Cartesian
strong topology, the same two analytic perturbation levels, and the same local polar-Cartesian coordinate equations prior to exit from the coordinate tube. The coordinate neighborhoods and angular strips are fixed, with positive
radii and widths independent of $\eps$, the admissible initial tail, and
$\eps T_\eps$. For each fixed derivative order, the derivative norms and $H_+$ are bounded by
a constant independent of these parameters.
The finite averaging order is fixed before the small parameters. These
hypotheses are not implied by bounds established separately for every fixed $T_0$ with
constants depending on $T_0$.

Write
\begin{equation}\label{long:slow-time}
 \mathsf S=\eps T_\eps,\qquad \mathsf Q=1+\mathsf S.
\end{equation}
All estimates in Proposition~\ref{prop:long-budgets} are first
established on a stopped initial interval
$[0,T]\subseteq[0,T_\eps]$ on which $W\le s\le s_0$, $W_0\le s$,
and the compatible action-divisor hypotheses hold. On this
interval, use $T$ in place of $T_\eps$ in the bounds and set
$\mathsf S=\eps T$, $\mathsf Q=1+\mathsf S$. Error norms denote
suprema over $[0,T]$.

\begin{proposition}[Time-explicit observable reduction]
\label{prop:long-budgets}
Under the uniform long-interval hypotheses, the augmented construction of
Sections~\ref{sec:window}--\ref{sec:filtered} gives the exact decomposition
$\mathcal D=A_V+E_0+E_2+E_{\rm av}^{[L]}$, with $A_V\in V$, and
\begin{equation}\label{long:leakage}
 \Gamma_{S_M}(T_\eps)
       \le C\bigl(\eps\mathsf Q+\mathsf S e_K\bigr),
       \qquad e_K=e^{-aK}.
\end{equation}
Its separate vector and scalar error bounds are
\begin{equation}\label{long:window-budget}
 \begin{aligned}
 \|E_0\|_{\mathcal A_{w_*}}
       &\le C\mathsf Q^2(\eps+e_K+s^{3/2}),\\
 |\nu\cdot E_0|&\le C\mathsf Q^2(\eps+e_K+s^{5/2}),
       \qquad \nu=Dh(I^0),
 \end{aligned}
\end{equation}
\begin{equation}\label{long:quadratic-budget}
 \begin{aligned}
 \|E_2\|_2&\le C\mathfrak p(M)\eps\mathsf Q
                         (\Delta^{-1}+s\Delta^{-2}),\\
 |\nu\cdot E_2|&\le C\mathfrak p(M)\eps\mathsf Q
                         (1+s\Delta^{-1}),
       \qquad\mathfrak p(M)=(1+M)^3,
 \end{aligned}
\end{equation}
and
\begin{equation}\label{long:averaging-budget}
 \begin{aligned}
 \|E_{\rm av}^{[L]}\|_2
  &\le C_L\mathsf Q\left[
    \sum_{n=1}^L\left(\frac{\eps^n}{\Delta^{2n-1}}
        +\frac{\eps^n(s+e_K)}{\Delta^{2n}}\right)
        +\frac{\eps^L}{\Delta^{2L}}\right],\\
 |\nu\cdot E_{\rm av}^{[L]}|
  &\le C_L\mathsf Q\left[
    \sum_{n=1}^L\left(\frac{\eps^n}{\Delta^{2n-2}}
        +\frac{\eps^n(s+e_K)}{\Delta^{2n-1}}\right)
        +\frac{\eps^L}{\Delta^{2L-1}}\right].
 \end{aligned}
\end{equation}
The constants depend on the uniform analytic and geometric data and,
where indicated, on $L$, but not on $T_\eps$ or $\mathsf S$.
Moreover,  the transfer estimate \eqref{rw:transferexact}
 takes the form
\begin{equation}\label{long:critical-transfer}
 \mathcal W_{T_\eps}\le e^{A_*\mathsf S}
 \left[2^{1/3}(W_0+K_{S_M}\mathcal B_{T_\eps})
       +2H_+^{2/3}\Gamma_{S_M}(T_\eps)^{2/3}\right],
 \qquad A_*=\frac{4\sqrt2}{3}H_+C_z.
\end{equation}
\end{proposition}

\begin{proof}
We track the explicit dependence on $\mathsf S=\eps T$ in the identities of Sections 3–5. 
The pointwise Taylor bounds and \eqref{rw:rawidentity} give
\begin{equation}\label{long:raw-budget}
 \begin{aligned}
 \|E_R\|_{\mathcal A_{w_*}}&\le C\mathsf S s^{3/2},\\
 |\nu\cdot E_R|&\le C\bigl(\eps s^{3/2}
             +\mathsf S s^{5/2}+\eps\mathsf S\sqrt s\bigr).
 \end{aligned}
\end{equation}
The first scalar term comes from the two endpoints.  In
\eqref{rw:highidentity}, the endpoints of $\mathcal U$ are bounded and
its full-trajectory residual is bounded independently of the interval.
Hence both the exterior integral and the high-linear integral are
$O(\eps+\eps^2T_\eps)=O(\eps\mathsf Q)$ in the weighted action norm
and in the frozen-frequency pairing.  The angular tail contributes
$C\eps T_\eps e_K=C\mathsf S e_K$.  The proof of exterior leakage
therefore gives \eqref{long:leakage}, without any assumption that the
initial normal coordinates vanish.

Keep the frozen split \eqref{rw:frozen-diagonal}.  The integrated normal
action equation now gives
\begin{equation}\label{long:diagonal-feedback}
 \sup_{t\le T_\eps}\sum_{j>M}j^{w_*}|J_j(t)-J_j^0|
 \le C\bigl(\eps\mathsf Q+\mathsf S e_K
                         +\mathsf S s^{3/2}\bigr)
 =:C\Lambda_{\mathsf S,s}.
\end{equation}
Thus $|d_{\rm h}^{\rm inc}|+|(d_{\rm h}^{\rm inc})_\varphi|
\le C\Lambda_{\mathsf S,s}$ along the  path.  Its action drift
has only tangential components, so its integrated vector error is at
most $C\mathsf S\Lambda_{\mathsf S,s}$.  The scalar identity
\eqref{rw:diagonalidentity}, using also the crude endpoint bound $Cs$
and $\|DQ_0\|_2\le Cs$, gives
\begin{equation}\label{long:diagonal-scalar}
 |\nu\cdot E_d|\le C\bigl(\eps s
       +\mathsf S s\Lambda_{\mathsf S,s}
       +\eps\mathsf S\sqrt s\bigr).
\end{equation}
The additional factor $\mathsf S$ comes from integrating the tangential
action drift $-\eps\partial_\varphi d_{\rm h}^{\rm inc}$.
Combining these bounds with the raw, exterior, high-linear and Fourier
tail contributions proves \eqref{long:window-budget}, since $s\le1$.

The coefficient construction in Section~\ref{sec:filtered} contains no
time integration.  All its coefficient constants and divisor ratios
are therefore unchanged under the uniform hypotheses.  In
\eqref{filt:quadraticchain}, bounded endpoints contribute one factor
$\eps$, while the residual integral contributes $\eps^2T_\eps
=\eps\mathsf S$.  This proves \eqref{long:quadratic-budget}.
For \eqref{filt:eq22}, the order-$n$ observable endpoints contribute
$\eps^n$, and every order-$n+1$ residual integrates to
$\eps^{n+1}T_\eps=\mathsf S\eps^n$.  The terminal term integrates
to $\mathsf S\eps^L$.  Retaining both endpoint terms and the independent vector and scalar
coefficient bounds gives
\eqref{long:averaging-budget}.

Finally, the integrated inequality
\eqref{rw:loweraction} and the proof of \eqref{rw:transferexact} retain
their explicit dependence on $\eps t$.  They give
\eqref{long:critical-transfer} with the stated $A_*$.
\end{proof}

\subsection{A logarithmic extension with the original higher norm}
\label{subsec:long-logarithmic}

\begin{theorem}[Conditional logarithmic-time stability]
\label{thm:long-logarithmic}
Fix $0<\delta<1/4$, $\kappa>\delta$ and $C_0\ge0$.
Suppose a collection of trajectories with
$r_0\le C_0\eps^\kappa$ satisfies the uniform long-interval hypotheses
on its supplied intervals $[0,T_\eps]$.  There are constants $c>0$
and $\eps_0>0$, depending only on the uniform data and these fixed
parameters, such that, for $0<\eps<\eps_0$, the bounds with constant $1$
\begin{equation}\label{long:log-conclusion}
 \begin{gathered}
 |p(t)|_1\le\eps^{2\delta},\qquad
 \sum_{j\notin A}j^{2N+1}J_j(t)\le\eps^{2\delta},\\
 \|\gamma(w(t))-\gamma(w(0))\|_{\ell^2}\le\eps^\delta,
 \qquad 0\le t\le\widehat T_\eps,\\
 \widehat T_\eps=
 \min\{T_\eps,c\eps^{-1}\log(1/\eps)\}
 \end{gathered}
\end{equation}
hold, and the fixed coordinate tube is preserved on that interval.
Exact finite-gap data are included by taking $C_0=0$.
There is no arithmetic restriction on the base frequencies.
\end{theorem}

\begin{proof}
Choose exponents
\begin{equation}\label{long:log-orders}
 2\delta<\zeta<\min\{2\kappa,1/2\},\qquad
 \zeta<b<\min\{5\zeta/4,(1+\zeta)/3,1/2\},\qquad
 L(1-2b)>b.
\end{equation}
All these strict choices are possible; for exact data the restriction
$\zeta<2\kappa$ can be omitted.  Fix $L$, choose
$K\asymp\log(1/\eps)$ so that $e_K\le\eps$, and use the compatible
hierarchy at $I^0$ with initial scale $\eps^b$ and $M\asymp\sqrt K$.
Thus $c_1\eps^b\le\rho,\Delta\le\eps^{b-o(1)}$ and
$d_V\le\mu\rho/8$.  The full high normal-linear estimates hold on the same
action ball.  Stop at $B=\rho$, $W=s:=\eps^\zeta$, or chart exit,
no later than $\widehat T_\eps$.

In Proposition~\ref{prop:long-budgets}, use this restricted interval:
$\mathsf S\le c\log(1/\eps)$, so both $\mathsf Q$ and
$\mathsf Q^2$ are subpower factors.  After division by $\rho$ for
the vector errors and by $\rho^2$ for their scalar pairings, a common
positive lower bound for the leading exponents is
\[
 \min\{1-2b,\ 1+\zeta-3b,\ 3\zeta/2-b,\ 5\zeta/2-2b,
                                      \ L(1-2b)-b\}>0.
\]
Higher averaging orders and Fourier tails only improve the relevant
minimum.  Hence $\eta=o(\rho)$ and
$\tau+2F_0\eps=o(\rho^2)$.  Lemma~\ref{lem:energyclosure} gives
$B\le\rho/4+o(\rho)\le\rho/2$, uniformly on shorter  intervals.

On every such stopped initial interval $[0,T]$, with
$T\le\widehat T_\eps$, leakage is at most $C\eps\mathsf Q$.
The transfer estimate therefore gives
\[
 \mathcal W_T
 \le C\eps^{-A_*c}
   \left(C_0^2\eps^{2\kappa}
               +\eps^{b-o(1)}+\eps^{2/3-o(1)}\right).
\]
Choose, for example,
\begin{equation}\label{long:log-time-constant}
 0<c<\frac{\min\{b-\zeta,\ 2\kappa-\zeta,\ 2/3-\zeta\}}
                         {2\max\{1,A_*\}},
\end{equation}
omitting the initial-tail entry when $C_0=0$.  This yields $W=o(s)$.
Initially $B=0$ and $W_0=o(s)$, and the preceding estimates strictly improve
both bootstrap bounds.
The Cartesian continuity and positive-action margin argument used
above excludes the remaining exit.  Thus the estimates hold on
$[0,\widehat T_\eps]$.  Applying Lemma~\ref{lem:spectral} at both
endpoints relative to the same finite-gap torus bounds the squared
gap error by $C(W+W_0)=o(\eps^{2\delta})$, because
$\zeta>2\delta$.  This also absorbs all constants in the action bounds.
\end{proof}

In particular, supplied slow-time lengths satisfying
$\eps T_\eps=o(\log(1/\eps))$ uniformly over the admissible collection
are covered under the same uniform hypotheses.  The condition $r_0=O(\eps^{1/4})$ admits every fixed
exponent $\delta<1/4$.  In contrast, the factor $e^{A_*\mathsf S}W_0$ does not in general yield a bound $Cr_0^2$
with $C$ independent of the length of the slow-time interval.  The condition
$r_0\le c_\delta\eps^\delta$ from Corollary~\ref{cor:shrinking-data}
cannot simply be carried unchanged into this logarithmic estimate.
 
\subsection{Polynomial intervals under a bound on one additional derivative}
\label{subsec:long-polynomial}

A stronger upper moment changes the critical interpolation exponent.
Only the solution bound is strengthened; the two analytic perturbation
levels remain those of Assumption~\ref{ass:contract}.

\begin{lemma}[Transfer from the stronger upper moment]
\label{lem:long-strong-transfer}
On a supplied  interval $[0,T]$ with $W\le s$, suppose
\begin{equation}\label{long:stronger-bound}
 \sup_{0\le t\le T}\|z(t)\|_{\sigma+2}\le H_{++}.
\end{equation}
For every finite set $S$ of normal indices one then has
\begin{equation}\label{long:superlinear-transfer}
 \mathcal W_T\le W_0+K_S\mathcal B_T
  +2^{4/5}H_{++}^{2/5}
        \left(\Gamma_S(T)+C_z\eps T s^{3/2}\right)^{4/5}.
\end{equation}
\end{lemma}

\begin{proof}
Use the positive excess $J_j^+=(J_j-J_j^0)_+$ from
Proposition~\ref{prop:transfer}.  Its integrated lower moment obeys
\eqref{rw:loweraction}; on the  interval $[0,T]$, its right-hand side is at most
$\Gamma_S(T)+C_z\eps T s^{3/2}$.  Since $0\le J_j^+\le J_j$,
\[
 \sum_j j^{2\sigma+4}J_j^+\le H_{++}^2/2,\qquad
 2\sigma=\frac45(2\sigma-1)+\frac15(2\sigma+4).
\]
H\"older therefore gives
\[
 2\sum_{j\notin S}j^{2\sigma}J_j^+
 \le 2^{4/5}H_{++}^{2/5}
       \left(\sum_{j\notin S}j^{2\sigma-1}J_j^+\right)^{4/5}.
\]
Combining this with the unchanged finite-coordinate bound
$W\le W_0+K_SB+2\sum_{j\notin S}j^{2\sigma}J_j^+$ proves
\eqref{long:superlinear-transfer}.  
\end{proof}

The raw feedback in \eqref{long:superlinear-transfer} has power
$(3/2)(4/5)=6/5>1$, in contrast to the critical power
$(3/2)(2/3)=1$ in \eqref{long:critical-transfer}.  Relative to $s$,
it becomes small when $(\eps T)s^{1/4}\to0$.  This permits polynomially growing slow-time intervals for exact finite-gap data
and suitable shrinking initial families under the hypotheses below.

\begin{theorem}[Conditional polynomial-time two-parameter stability]
\label{thm:long-polynomial}
Fix parameters
\begin{equation}\label{long:polynomial-range}
 0<\delta<1/4,\qquad
 0<\alpha<\min\{\delta/2,\ 1/2-2\delta\},\qquad
 \lambda=\max\{\eps,r_0^{1/\delta}\}.
\end{equation}
Suppose the uniform long-interval hypotheses hold on
\begin{equation}\label{long:polynomial-interval}
 [0,T_{\eps,r_0}],\qquad
 T_{\eps,r_0}=\eps^{-1}\lambda^{-\alpha},
\end{equation}
and additionally \eqref{long:stronger-bound} holds there with a
constant $H_{++}$ independent of $\eps,r_0$ and the admissible datum.
There are independent positive thresholds for $r_0$ and $\eps$ and a
constant $C_{\delta,\alpha}$ such that throughout this interval
\begin{equation}\label{long:polynomial-conclusion}
 \begin{gathered}
 W(t)\le C_{\delta,\alpha}(r_0^2+\eps^{2\delta}),\\
 |p(t)|_1+\sum_{j\notin A}j^{2N+1}|J_j(t)-J_j^0|
                 \le C_{\delta,\alpha}(r_0^2+\eps^{2\delta}),\\
 \|\gamma(w(t))-\gamma(w(0))\|_{\ell^2}^2
                 \le C_{\delta,\alpha}(r_0^2+\eps^{2\delta}).
 \end{gathered}
\end{equation}
The constants depend only on the uniform hypotheses, $H_{++}$ and the
fixed parameters.  All admissible phases and resonances are included,
and the fixed coordinate tube is preserved.
\end{theorem}

\begin{proof}
Choose
\begin{equation}\label{long:polynomial-orders}
 \begin{gathered}
 2\delta<b<\min\left\{3\delta-2\alpha,\ \frac52\delta-\alpha,
          \ \frac12-\alpha,\ \frac{1+2\delta-\alpha}{3}\right\},\\
 L(1-2b)>b+\alpha.
 \end{gathered}
\end{equation}
Each upper endpoint exceeds $2\delta$ under
\eqref{long:polynomial-range}; the last difference is
$(1-4\delta-\alpha)/3>0$.  Fix $b,L$, take
$K\asymp\log(1/\lambda)$ with $e_K\le\lambda$ and
$M\asymp\sqrt K$, and apply the hierarchy at $I^0$ with initial
scale $\lambda^b$.  Proposition~\ref{prop:hierarchy} gives
$c_1\lambda^b\le\rho,\Delta\le\lambda^{b-o(1)}$ and
$d_V\le\mu\rho/8$.  Stop at $B=\rho$, $W=s:=8\lambda^{2\delta}$,
or chart exit, no later than $T_{\eps,r_0}$.

On a stopped initial interval $[0,T]$ with $T\le T_{\eps,r_0}$,
one has $\mathsf S=\eps T\le\lambda^{-\alpha}$ and
$\mathsf Q\le2\lambda^{-\alpha}$.  The time-explicit error bounds
give the normalized exponents in Table~\ref{tab:polynomial-exponents}.
As before each denotes an
upper bound by a constant times $\lambda^{\text{exponent}-o(1)}$.

\begin{table}[ht]
\centering
\small
\begin{tabular}{@{}lll@{}}
\toprule
Contribution & Vector divided by $\rho$ & Scalar divided by $\rho^2$\\
\midrule
Window error $E_0$ ($\eps,e_K$) & $1-2\alpha-b$ & $1-2\alpha-2b$\\
Window error $E_0$ (Taylor) & $3\delta-b-2\alpha$ & $5\delta-2b-2\alpha$\\
Averaging error $E_{\rm av}^{[L]}$ (order $n$) & $n(1-2b)-\alpha$ & $n(1-2b)-\alpha$\\
Order-$n$ terms with $s$ & $n(1-2b)+2\delta-b-\alpha$
                        & $n(1-2b)+2\delta-b-\alpha$\\
Terminal remainder & $L(1-2b)-b-\alpha$ & $L(1-2b)-b-\alpha$\\
Finite quadratic terms & $1-\alpha-2b,\ 1-\alpha+2\delta-3b$
                       & $1-\alpha-2b,\ 1-\alpha+2\delta-3b$\\
Energy error & --- & $1-2b$\\
\bottomrule
\end{tabular}
\caption{Normalized error exponents for polynomial-time two-parameter stability.}
\label{tab:polynomial-exponents}
\end{table}

All are positive by \eqref{long:polynomial-orders}.  The averaging
Fourier-tail exponent is $n+1-(2n+1)b-\alpha$; it exceeds the
ordinary one by $1-b>0$.  The worst nonterminal averaging order is
$n=1$.  Thus $\eta=o(\rho)$ and
$\tau+2F_0\eps=o(\rho^2)$, giving $B\le\rho/2$ by the recentered
energy inequality.  

Equation~\eqref{long:leakage} gives
$\Gamma_{S_M}(T)\le C\lambda^{1-\alpha}$.
In \eqref{long:superlinear-transfer}, the finite-coordinate and two
tail contributions, divided by $s$, have positive exponents
\begin{equation}\label{long:transfer-margins}
 b-2\delta,\qquad
 \frac45(1-\alpha)-2\delta,\qquad
 \frac45(3\delta-\alpha)-2\delta
                         =\frac25\delta-\frac45\alpha.
\end{equation}
The middle exponent is positive since $\alpha<\delta/2$ and
$\delta<1/4$; the other two follow directly from the choices.
The fixed factor $H_{++}^{2/5}$ is absorbed by taking the independent
smallness thresholds sufficiently small.  Since $W_0\le s/8$, we
obtain $W\le s/2$ on the stopped interval.  Both improved bounds hold uniformly on every shorter stopped interval.  Cartesian continuity and the positive
tangential margin exclude the first exit, as in the proof of
Theorem~\ref{thm:main}.

It follows that $W\le8\lambda^{2\delta}
\le8(r_0^2+\eps^{2\delta})$ throughout
\eqref{long:polynomial-interval}.  Nonnegativity gives the weighted
increment bound.  Lemma~\ref{lem:spectral}, applied at both endpoints
relative to the same finite-gap torus, gives the squared gap bound.
Choose a common $\lambda_*>0$ for all inequalities and then
$\eps_*\le\lambda_*$ and
$r_*\le\min\{r_{\rm adm},\lambda_*^\delta\}$ to obtain independent
thresholds.  For fixed $r_0>0$, the slow-time factor is eventually
$r_0^{-\alpha/\delta}$ as $\eps\to0$.
\end{proof}

\begin{corollary}[Bounds with constant $1$ on a polynomially longer interval]
\label{cor:long-polynomial-rates}
Fix $0<\delta<1/4$ and
\begin{equation}\label{long:polynomial-unit-range}
 0<\alpha<\min\{1/10,\ 1/2-2\delta\}.
\end{equation}
For exact finite-gap data or $r_0\le C_0\eps^{1/4}$, the
three action and gap inequalities with constant $1$ in
\eqref{long:log-conclusion} hold throughout
$[0,\eps^{-1-\alpha}]$, provided the uniform long-interval hypotheses
and the bound \eqref{long:stronger-bound} are supplied there.
The constants and thresholds may depend on $C_0,\delta,\alpha$.
\end{corollary}

\begin{proof}
Choose
\[
 \max\{\delta,2\alpha\}<\delta_1<1/4-\alpha/2.
\]
The interval is nonempty precisely when \eqref{long:polynomial-unit-range} holds.  Theorem~\ref{thm:long-polynomial}
applies with $\delta_1$ and $\alpha$.  Since $\delta_1<1/4$,
$r_0^{1/\delta_1}=o(\eps)$ uniformly, so the effective scale is $\lambda=\eps$
for sufficiently small $\eps$. The interval given by
Theorem~\ref{thm:long-polynomial} is therefore $[0,\eps^{-1-\alpha}]$.  Moreover,
$r_0^2+\eps^{2\delta_1}=o(\eps^{2\delta})$, which absorbs all
constants in the total-tail, tangential-action and gap estimates.
\end{proof}

For example, a target $\delta=1/8$ and time exponent $\alpha=1/20$
are consistent with $\delta_1=1/5$, $b=17/40$ and $L=4$ in the
proof.  Under the stronger uniform hypotheses, the spectral error is
at most $\eps^{1/8}$ and the action errors at most $\eps^{1/4}$
through time $\eps^{-21/20}$.

With only the $h^{\sigma+1}$ bound and $A_*>0$,
\eqref{long:critical-transfer} contains $e^{A_*\mathsf S}$.
Lemma~\ref{lem:long-strong-transfer} gives a different transfer estimate
under an additional $h^{\sigma+2}$ bound, while the integrated Taylor
terms still restrict the polynomial time exponents. Increasing the
fixed averaging order improves only the terminal remainder.
All longer-time conclusions require the independent strong solution
and the stated uniform bounds on their proposed intervals.

% END sections/06_closure
% BEGIN sections/07_application
\section{Examples of Hamiltonian perturbations}
\label{sec:application}

We now verify the analytic and evolution hypotheses of
Assumption~\ref{ass:contract} for several concrete classes of local and
nonlocal autonomous Hamiltonian perturbations. Throughout this section we
denote the physical potential by $u(x,t)$, and write $\|\cdot\|_0$ for the
$L^2(\T)$ norm.

Table~\ref{tab:application-classes} summarizes the main application classes, their sufficient
initial Sobolev bounds, and the resulting stability intervals. All initial
data $u_0$ are smooth and mean zero, with $I_A(u_0)\in\mathcal K$, and each
stated initial Sobolev norm is bounded uniformly over the prescribed family,
independently of $\eps$. Stability holds for sufficiently small $r_0,\eps$
and any fixed $0<\delta<1/4$.

\begin{table}[ht]
\centering
\small
\begin{tabular}{@{}>{\raggedright\arraybackslash}p{.32\linewidth}
 >{\raggedright\arraybackslash}p{.23\linewidth}
 >{\raggedright\arraybackslash}p{.37\linewidth}@{}}
\toprule
Class and reference & Initial control & Stability interval \\
\midrule
Additive or transport/smoothing
(Proposition~\ref{prop:app-transport},
Corollary~\ref{cor:app-nearby}(1))
& $H^{N+2}$
& $[0,T_0/\eps]$ for every prescribed fixed $T_0>0$ \\[3pt]
The same classes
(Corollary~\ref{cor:app-growing-horizon})
& $H^{N+2}$; $r_0=0$ or \mbox{$r_0\le C_0\eps^\kappa$}, $\kappa>\delta$
& $[0,\mathsf S(\eps)/\eps]$, with an unspecified
\mbox{$\mathsf S(\eps)\to\infty$} \\[3pt]
Analytic local densities
(Corollary~\ref{cor:app-local-density},
Corollary~\ref{cor:app-nearby}(2))
& $H^{N+2}$
& $[0,T_*/\eps]$ for some fixed $T_*>0$ \\[3pt]
Positive quartic Gardner
(Proposition~\ref{prop:app-gardner},
Theorem~\ref{thm:app-gardner-alltime})
& $H^{N+3}$
& All $t\ge0$ using the integrable Gardner structure \\
\bottomrule
\end{tabular}
\caption{Initial Sobolev bounds and stability intervals for the main perturbation classes.}
\label{tab:application-classes}
\end{table}

For fixed horizons, constants and smallness thresholds may depend on $T_0$.
 Exact finite-gap data in the prescribed compact family
automatically satisfy all the initial Sobolev bounds in the table.

\paragraph{Organization of Section~\ref{sec:application}.}
Subsection~\ref{subsec:app-common} gives the common coordinate and local
evolution criteria. Subsections~\ref{subsec:app-transport}--\ref{subsec:app-local-densities}
verify the baseline estimates for the perturbation classes in the table,
first for exact finite-gap data. Subsection~\ref{subsec:app-nearby} extends
these estimates to nearby infinite-gap data, and
Subsections~\ref{subsec:app-gardner-alltime}--\ref{subsec:app-growing-horizon}
give the all-time Gardner estimate and a nonquantitative growing slow-time
horizon for the perturbations in Proposition~\ref{prop:app-transport}, respectively.

\subsection{Common verification of the coordinate hypotheses}
\label{subsec:app-common}

The following lemma shows that independently established existence and higher
Sobolev bounds in physical variables imply the coordinate hypotheses.  Each application
supplies these bounds and the analytic mapping properties below.
All coefficients, functionals and smoothing operators in this section
are fixed independently of $\eps$.

\begin{lemma}[Physical bounds imply the coordinate hypotheses]
\label{lem:app-physical-contract}
Let $\mathscr U$ be a collection of smooth mean-zero initial data with
$I_A(u_0)\in\mathcal K$ and $r_0\le r_{\rm adm}$. For each $u_0\in\mathscr U$,
use the chart associated with $\xi=(I_A(u_0),0)$.  Let $\mathcal F$ be a
real-analytic functional on $H^N_0$ whose
Hamiltonian vector field $G=\partial_x\nabla\mathcal F$, initially
evaluated on smooth functions, has compatible real-analytic extensions
\[
 G:H^N_0\longrightarrow H^{N-1}_0,
 \qquad G:H^{N+1}_0\longrightarrow H^N_0.
\]
Assume local holomorphic extensions at the relevant real points.  Suppose
that, for some fixed $T_0>0$ and $\eps_*>0$, the perturbed equation
\begin{equation}\label{app:general-pde}
 u_t=-u_{xxx}+6uu_x+\eps G(u)
\end{equation}
has smooth solutions for the admissible data on $[0,T_0/\eps]$ for
$0<\eps\le\eps_*$, with
\begin{equation}\label{app:physical-upper}
 \sup_{u_0\in\mathscr U}
 \ \sup_{0<\eps\le\eps_*}\ \sup_{0\le t\le T_0/\eps}
                    \|u(t)\|_{H^{N+2}}<\infty.
\end{equation}
Then Assumption~\ref{ass:contract} holds uniformly for this
collection.  For sufficiently small $r_0$ and $\eps$, Theorem~\ref{thm:main}
gives \eqref{setup:near-actions}--\eqref{setup:near-gaps} with
this $T_0$.  For exact finite-gap data, or the shrinking neighborhoods
in Corollary~\ref{cor:shrinking-data}, it also gives
\eqref{setup:main-actions}--\eqref{setup:main-gaps}.
\end{lemma}

\begin{proof}
\emph{Canonical maps and represented gradients.}
Put $\sigma=N+1/2$ and denote by
$\mathcal B_s:H^s_0\to h^{s+1/2}$ the compatible normalized Birkhoff
maps, with the complex identification of Section~\ref{sec:setup}; see
\cite[Theorem~1.3]{KST2013}.  For $s=N,N+1$, the coordinate field is
\begin{equation}\label{app:general-coordinate-force}
 P_s(\mathcal Z)=D\mathcal B_{s-1}(\mathcal B_s^{-1}\mathcal Z)
              [G(\mathcal B_s^{-1}\mathcal Z)]
       :h^{s+1/2}\longrightarrow h^{s-1/2}.
\end{equation}
Here $\mathcal B_s^{-1}\mathcal Z\in H^s_0$ is included in $H^{s-1}_0$
when evaluating $D\mathcal B_{s-1}$, and $G$ takes values in
$H^{s-1}_0$.  Analytic differentiation and composition prove these two
mapping properties; compatibility identifies the coordinate components.
Write $P$ for the resulting compatible field and let $\Theta_\xi$ be the
finite polar and normal Cartesian chart of \eqref{setup:polar-map}.  Then
\begin{equation}\label{app:chartforce}
 f_\xi=\mathcal F\circ\mathcal B^{-1}\circ\Theta_\xi,
 \qquad
 X_{f_\xi}=(D\Theta_\xi)^{-1}P\circ\Theta_\xi.
\end{equation}
The fixed positive tangential action margin makes this chart and its
inverse differential analytic uniformly near the compact base family.
Canonicity, first at smooth base points on finitely supported real normal
variations and then by continuity and density, gives
\begin{equation}\label{app:representedgradient}
 D_zf_\xi(b)[v]
   =\langle-\mathcal J(X_{f_\xi})^z(b),v\rangle_{\ell^2}.
\end{equation}
The representing sequence belongs to $h^{s-1/2}$, so its ordinary
sequence pairing is continuous on $h^{s+1/2}$.  Thus
\begin{equation}\label{app:twolevels}
 \nabla_zf_\xi:h^\sigma\to h^{\sigma-1},\qquad
 \nabla_zf_\xi:h^{\sigma+1}\to h^\sigma.
\end{equation}
The scalar and finite-variable derivatives are analytic as well.
A finite cover of the compact base family supplies a common lower-level
neighborhood, angular strip and fixed-order derivative bounds.

\emph{Uniform upper neighborhoods.}
Let $\Pi_R$ retain the physical Fourier modes $|k|\le R$, $R\ge1$.
The independent estimate \eqref{app:physical-upper} gives, uniformly
over the entire collection of states,
\[
 \|(1-\Pi_R)u(t)\|_{H^{N+1}}^2
       \le CR^{-2}\|u(t)\|_{H^{N+2}}^2.
\]
Bounded low Fourier projections and uniformly small tails show that
this collection has compact closure in $H^{N+1}_0$.  By continuity of
$\mathcal B_{N+1}$, its Birkhoff image has compact closure in
$h^{\sigma+1}$; in particular $\|z(t)\|_{\sigma+1}\le H_+$ uniformly.
On a smaller time interval with a positive tangential margin, include
all radial normal Taylor segments, finite tangential segments towards
the base torus, and real angular translates used in the estimates.
These operations preserve the margin and are continuous maps on products of the preceding compact closure with compact
finite-dimensional parameter sets.  The union of their images therefore has compact closure in the upper coordinate
topology.  The analytic maps \eqref{app:chartforce}--\eqref{app:twolevels}
are defined near that closure.  A finite cover supplies a common complex
neighborhood, angular strip, and bounds for every fixed finite collection
of derivatives. 

\emph{Strong paths and the stopped chart.}
For each fixed $\eps$, the supplied smooth solution and its equation give
$u\in C([0,T_0/\eps];H^{N+1}_0)\cap
C^1([0,T_0/\eps];H^{N-2}_0)$.  The Banach chain rule for
$\mathcal B_{N-2}$, with $N-2\ge1$, and the compatible unforced conjugacy
yield
\begin{equation}\label{app:coordinatepath}
 \mathcal Z\in C([0,T_0/\eps];h^{\sigma+1})
       \cap C^1([0,T_0/\eps];h^{\sigma-2}),\qquad
 \dot{\mathcal Z}_j=i\Omega_j(I)\mathcal Z_j+\eps P_j(\mathcal Z).
\end{equation}
This is a global Cartesian statement with stronger regularity than
\eqref{setup:strong-solution}.  On every subinterval before chart exit,
the finite polar map gives $(\varphi,p)\in C^1$ and the corresponding
Hamiltonian equations.  No positive tangential margin beyond that
stopped interval is assumed.  These independent verifications establish
Assumption~\ref{ass:contract}; Theorem~\ref{thm:main} then gives the
two-parameter estimates and preserves the chart margin.  For the exact
or shrinking initial families, Corollary~\ref{cor:shrinking-data} gives
the asserted estimates with constant $1$.
\end{proof}

We use the following local existence and continuation result at high
Sobolev regularity.

\begin{lemma}[Local spatially smooth solutions and continuation]
\label{lem:app-local}
Let $\mathfrak a,\mathfrak b\in C^\infty(\T\times\R;\R)$, with
bounded derivatives on each compact range of their scalar argument.
Let $\mathcal R:L^2_0\to H^q_0$ be $C^1$ for every nonnegative integer
$q$, with its value and first derivative bounded on each $L^2$ ball.
Suppose that the expression in brackets below has zero spatial mean
for smooth mean-zero $u$.  The equation
\begin{equation}\label{app:local-equation}
 u_t=-u_{xxx}+6uu_x+
       \eps\bigl[\mathfrak a(x,u)u_x+\mathfrak b(x,u)+\mathcal R(u)\bigr]
\end{equation}
has, for every smooth mean-zero datum, a unique local solution on
$[0,t_{\rm loc}]$, where $t_{\rm loc}>0$ is uniform for $0<\eps\le1$
on bounded $H^4$ sets of such data, and
\[
 u\in\bigcap_{\substack{q\in\N\\q\ge4}}
 \left(C([0,t_{\rm loc}];H^q_0)
       \cap C^1([0,t_{\rm loc}];H^{q-3}_0)\right).
\]
Such spatially smooth solutions continue in these spaces through a
finite endpoint whenever their $H^4$ norm stays bounded up to that
endpoint.  Higher spatial regularity is propagated without an additional
smallness condition on $\eps$.   
\end{lemma}

\begin{proof}
Use spatial Fourier projections in \eqref{app:local-equation} and
project the nonlinear terms orthogonally onto the retained modes.
The projections commute with derivatives.  The cubic dispersive part
is skew in every integer Sobolev norm.  If $A(x)=\mathfrak a(x,u(x))$,
the term containing $u^{(q+1)}$ is controlled by
\[
 2\int_\T u^{(q)}A u^{(q+1)}\,dx
       =-\int_\T(\partial_xA)(u^{(q)})^2\,dx.
\]
The one-dimensional product, composition and commutator estimates give,
for every integer $q\ge4$, on $\|u\|_{H^4}\le R$,
\begin{equation}\label{app:local-energy}
 \left|\frac{d}{dt}\|u\|_{H^q}^2\right|
        \le C_{q,R}\bigl(1+\|u\|_{H^q}^2\bigr).
\end{equation}
For example, the commutator is bounded by
\[
 \|[\partial_x^q,A]u_x\|_0
 \le C_q\bigl(\|\partial_xA\|_\infty\|u\|_{H^q}
                     +\|A\|_{H^q}\|u_x\|_\infty\bigr),
\]
and $\|A\|_{H^q}\le C_{q,R}(1+\|u\|_{H^q})$ on that set.
The same calculation treats $6uu_x$ and $\mathfrak b(x,u)$; the
smoothing term is bounded in $H^q$ on the fixed $L^2$ ball.  The
constants in \eqref{app:local-energy} are independent of the projection
and of $0<\eps\le1$.

The $q=4$ estimate gives a common local time, and the estimates at each
higher order give uniform bounds on that same time interval for smooth
data.  The projected time derivatives are bounded three orders lower.
Compactness and a diagonal extraction therefore give convergence at
each lower Sobolev order.  Using one additional bounded order gives
strong continuity at every fixed order in the limit; the equation
gives $C^1$ regularity three orders lower.  

For uniqueness, the difference $w=u-v$ satisfies a linear first-order
energy inequality.  Indeed,
\[
 \mathfrak a(x,u)u_x-\mathfrak a(x,v)v_x
   =\mathfrak a(x,u)w_x
       +(\mathfrak a(x,u)-\mathfrak a(x,v))v_x.
\]
Integration by parts in the first term, bounded derivatives of the
coefficients on the solution range, and the $L^2$ Lipschitz bound for
$\mathcal R$ give $d\|w\|_0^2/dt\le C\|w\|_0^2$ on each compact
local interval.  Gronwall proves uniqueness.  If the $H^4$ norm stays
bounded up to a finite endpoint, \eqref{app:local-energy} bounds every
higher norm there.  The equation and interpolation give a spatially
smooth endpoint datum, and the same local construction restarts the
solution in the stated spaces.  When $\mathcal R$ is $C^\infty$ into
every $H^q_0$, as in the applications below, repeated differentiation
of the equation gives higher time regularity as well. \end{proof}

\subsection{Spatial transport and Hamiltonian smoothing terms}
\label{subsec:app-transport}

Fix $N\ge3$, a finite tangential set $A\subset\N$, and a compact
positive action family $\mathcal K$ as in Theorem~\ref{thm:main}.
For the applications below, the initial family consists of all finite-gap
potentials with
\begin{equation}\label{app:initial}
 I_a(u_0)=\xi_a\quad(a\in A),\qquad
 I_j(u_0)=0\quad(j\notin A),\qquad \xi\in\mathcal K,
\end{equation}
including every phase, unless an enlarged collection $\mathscr U$ is
specified.  The family may have arbitrary fixed amplitudes and resonances;
the case $A=\varnothing$ is the single zero torus.
The following criterion permits every prescribed fixed slow-time horizon.

\begin{proposition}[Stability for every fixed slow-time horizon]
\label{prop:app-transport}
Let $V\in C^\infty(\T;\R)$ and let $\mathcal F_{\rm sm}$ be a
real-analytic functional on $L^2_0$.  Suppose its Hamiltonian field
$\mathcal R=\partial_x\nabla\mathcal F_{\rm sm}$ is real analytic
from $L^2_0$ to $H^m_0$ for every nonnegative integer $m$, and that
\begin{equation}\label{app:smoothing-ball}
 \sup_{\|u\|_0\le R}
 \left(\|\mathcal R(u)\|_{H^m}
       +\|D\mathcal R(u)\|_{\mathcal L(L^2_0,H^m_0)}\right)
       \le C_{m,R}<\infty
 \qquad(m\ge0,\ R>0).
\end{equation}
Assume also that, for a fixed $C_0\ge0$,
\begin{equation}\label{app:smoothing-radial}
 2\langle u,\mathcal R(u)\rangle
                   \le C_0(1+\|u\|_0^2)\qquad(u\in L^2_0).
\end{equation}
Then the Hamiltonian perturbation
\begin{equation}\label{app:transport-hamiltonian}
 \mathcal F(u)=\frac12\int_\T V(x)u(x)^2\,dx+\mathcal F_{\rm sm}(u),
 \qquad G(u)=\partial_x(Vu)+\mathcal R(u)
\end{equation}
satisfies Assumption~\ref{ass:contract} for the initial family
\eqref{app:initial} on $[0,T/\eps]$ for every fixed $T>0$.
The physical solutions exist and are smooth there for $0<\eps\le1$ and
have uniform bounds at every fixed Sobolev order.  In particular,
\begin{equation}\label{app:strong}
 \begin{gathered}
 u\in C([0,T/\eps];H^{N+1}_0)
       \cap C^1([0,T/\eps];H^{N-2}_0),\\
 \sup_{0\le t\le T/\eps}\|u(t)\|_{H^{N+1}}\le C_+.
 \end{gathered}
\end{equation}
For each $0<\delta<1/4$, the unique smooth solution satisfies, for
sufficiently small $\eps$,
\begin{align}
 \sum_{a\in A}|I_a(u(t))-\xi_a|&\le\eps^{2\delta},&
 \sum_{j\notin A}j^{2N+1}I_j(u(t))&\le\eps^{2\delta},
 \label{app:actions}\\
 \|\gamma(u(t))-\gamma(u_0)\|_{\ell^2}&\le\eps^\delta,
 &&0\le t\le T/\eps.\label{app:gaps}
\end{align}
The constants may depend on $V,\mathcal R,T,N,A,\mathcal K,\delta$
and the stated bounds, but are uniform in the initial phases.
No sign condition is imposed on $V$.
\end{proposition}

\begin{proof}
The scalar perturbation and its field have the two analytic physical
mapping properties of Lemma~\ref{lem:app-physical-contract}.
In particular, $\partial_x(Vu):H^s_0\to H^{s-1}_0$ is a bounded
linear map.  
Lemma~\ref{lem:app-local} supplies the local smooth solution.  We now
establish its independent bounds while $\eps t\le T$.

Periodic integration by parts gives
\begin{equation}\label{app:transport-mass}
 \frac{d}{dt}\|u\|_0^2
   =\eps\int_\T V_xu^2\,dx
         +2\eps\langle u,\mathcal R(u)\rangle
   \le C\eps(1+\|u\|_0^2).
\end{equation}
Thus $1+\|u(t)\|_0^2\le(1+\|u_0\|_0^2)e^{C\eps t}$.
The state remains in a fixed $L^2$ ball on this slow-time interval,
so \eqref{app:smoothing-ball} bounds $\mathcal R(u(t))$ in each fixed
Sobolev norm.

For $m\ge1$, put $X_m=\|\partial_x^m u\|_0^2$ and normalize the
$m$-th KdV conservation law as
$\mathscr E_m(u)=X_m+\mathscr P_m(u)$.  Here
$\mathscr P_1(u)=2\int_\T u^3\,dx$, and for $m\ge2$ the lower part
has the form
\[
 \mathscr P_m(u)=\int_\T
 \left(c_m u(\partial_x^{m-1}u)^2
       +Q_m(u,\ldots,\partial_x^{m-2}u)\right)\,dx,
\]
where $c_m$ is real and $Q_m$ is a polynomial; see
\cite[Remark~3.3 and equation~(4.21)]{HK2014}.
Write $u^{(m)}=\partial_x^m u$ and $L_Vu=\partial_x(Vu)$.
The leading quadratic derivative satisfies
\begin{equation}\label{app:transport-leading}
 \begin{split}
 2\langle u^{(m)},\partial_x^m L_Vu\rangle
    ={}&(2m+1)\int_\T V_x(u^{(m)})^2\,dx\\
       &+2\sum_{j=2}^{m+1}\binom{m+1}{j}
          \int_\T u^{(m)}V^{(j)}u^{(m+1-j)}\,dx.
 \end{split}
\end{equation}
It is bounded by $C_{m,V}(X_m+\|u\|_{H^{m-1}}^2)$.  Pairing the
leading derivative with $\mathcal R(u)$ gives at most $X_m+C$.

For $r=m-1\ge1$, the distinguished lower term has the cancellation
\begin{equation}\label{app:transport-lower}
 \begin{split}
 D\!\left(\int_\T u(u^{(r)})^2\,dx\right)[L_Vu]
    ={}&2(r+1)\int_\T V_xu(u^{(r)})^2\,dx\\
       &+2\sum_{j=2}^{r+1}\binom{r+1}{j}
          \int_\T u\,u^{(r)}V^{(j)}u^{(r+1-j)}\,dx.
 \end{split}
\end{equation}
Indeed, the terms containing $V u_x(u^{(r)})^2$ cancel after
integration by parts.  The right side is bounded on each bounded
$H^{m-1}$ set.  In particular, for $m=2$($r=1$), the cancellation of $V u_x^3$ yields a bound depending only on $\|u\|_{H^1}$.  Differentiating the remaining
polynomial $Q_m$ along $L_Vu$ produces at most one factor of order
$m-1$, while its other factors have order at most $m-2$.  The former
is in $L^2$ and the latter in $L^\infty$ under the same preceding
Sobolev bound.  Derivatives of the lower terms along $\mathcal R(u)$
are bounded by \eqref{app:smoothing-ball}.

At $m=1$, the mass bound gives
$|\mathscr P_1|\le X_1/2+C$, and
\[
 D\mathscr P_1(u)[L_Vu]=4\int_\T V_xu^3\,dx,
 \qquad
 |D\mathscr P_1(u)[L_Vu]|\le C(1+X_1).
\]
The smoothing contribution $6\int u^2\mathcal R(u)$ is bounded.
Induction therefore gives constants independent of $\eps$ with
\begin{equation}\label{app:transport-hierarchy}
 \tfrac12X_m-A_m\le\mathscr E_m(u)\le2X_m+A_m,
 \qquad
 \mathscr E_m'(t)\le\eps C_m(1+X_m).
\end{equation}
For $Z_m=\mathscr E_m+A_m+1$, this implies
\begin{equation}\label{app:transport-gronwall}
 Z_m'\le2\eps C_mZ_m,
 \qquad X_m(t)\le2Z_m(0)e^{2C_mT}\quad(\eps t\le T).
\end{equation}
The constants may depend on the preceding bounds and on $m,T$;
no uniformity as $m\to\infty$ is claimed.  These bounds prevent
finite-time loss of smoothness before $T/\eps$, by
Lemma~\ref{lem:app-local}, and give, for every fixed integer $m\ge4$,
\begin{equation}\label{app:slowbound}
 \sup_{0\le t\le T/\eps}\|u(t)\|_{H^m}
       \le C_m(T,V,\mathcal R,\|u_0\|_{H^m}),
\end{equation}
uniformly for $0<\eps\le1$ on bounded sets of smooth initial data.
The constants include the stated bounds on $\mathcal R$; no
order-dependent smallness threshold is used to propagate smoothness.
The family \eqref{app:initial} is compact in every $H^m_0$: its
Birkhoff image has fixed finite support and compact parameters, and the
inverse Birkhoff maps are continuous at every order
\cite[Theorem~1.3]{KST2013}.  Thus the bounds are uniform at every
fixed order.  Lemma~\ref{lem:app-local} gives \eqref{app:strong},
and \eqref{app:physical-upper} holds.  Lemma~\ref{lem:app-physical-contract}, with
$\mathscr U$ equal to the family \eqref{app:initial}, and then
Corollary~\ref{cor:shrinking-data} with $r_0=0$ give the conclusions.
\end{proof}

\paragraph{Additive forcing.}
Let $a\in C^\infty(\T;\R)$ be fixed independently of $\eps$.
Taking $V=0$ and $\mathcal F_{\rm sm}=F_a$ gives
\begin{equation}\label{app:hamiltonian}
 \mathcal H_\eps(u)=\mathcal H_{\rm KdV}(u)+\eps F_a(u),
 \qquad F_a(u)=\int_\T a(x)u(x)\,dx,
 \qquad X_{F_a}=a_x,
\end{equation}
and hence the additive equation
\begin{equation}\label{app:pde}
 u_t=-u_{xxx}+6uu_x+\eps a_x(x),\qquad u(0)=u_0.
\end{equation}
The constant field $\mathcal R(u)=a_x$ is analytic from $L^2_0$ to
every $H^m_0$, its derivative vanishes, and
\[
 2\langle u,a_x\rangle\le\|u\|_0^2+\|a_x\|_0^2.
\]
Thus \eqref{app:smoothing-ball}--\eqref{app:smoothing-radial} hold,
and Proposition~\ref{prop:app-transport} gives the stated arbitrary-horizon
bounds for additive forcing.  Spatial analyticity of $a$ is unnecessary.

\subsection{Concrete nonlinear and nonlocal examples}
\label{subsec:app-examples}

Each of the following classes meets the sufficient conditions of
Proposition~\ref{prop:app-transport}.  Thus every fixed $T>0$ is
allowed, with constants depending on that horizon.  They may also be
combined with each other by finite sums and with the original force
$a_x$.

\paragraph{A spatially varying quadratic potential.}
For any fixed real $V\in C^\infty(\T)$, take $\mathcal R=0$:
\begin{equation}\label{app:variable-potential}
 \mathcal F_V(u)=\tfrac12\int_\T Vu^2\,dx,
 \qquad X_{\mathcal F_V}=Vu_x+V_xu.
\end{equation}
For example, $V(x)=\cos(2\pi x)$ gives
$\cos(2\pi x)u_x-2\pi\sin(2\pi x)u$.  This is a spatially
inhomogeneous first-order perturbation.   Adding $\int au$ to the functional adds $a_x$
to the field.  Neither $a$ nor $V$ needs to be analytic in $x$.

\paragraph{Nonlinear finite-rank Hamiltonians.}
Let $a_1,\ldots,a_r\in C^\infty(\T;\R)$ and set
$\ell_j(u)=\int_\T a_ju\,dx$.  For a real-analytic
$\Phi:\R^r\to\R$ satisfying $|\nabla\Phi(y)|\le C(1+|y|)$, put
\begin{equation}\label{app:finite-rank}
 \mathcal F_\Phi(u)=\Phi(\ell_1(u),\ldots,\ell_r(u)),
 \qquad
 X_{\mathcal F_\Phi}(u)=\sum_{j=1}^r
                     \partial_j\Phi(\ell(u))\,a_j'.
\end{equation}
The finite-dimensional argument ranges over a compact set on each
$L^2$ ball, giving all fixed-order local analytic bounds.  Moreover,
$\|X_{\mathcal F_\Phi}(u)\|_{H^m}\le C_m(1+\|u\|_0)$, which
implies \eqref{app:smoothing-radial}.  A nonlinear example is
\begin{equation}\label{app:finite-rank-sine}
 \mathcal F(u)=1-\cos\!\left(\int_\T au\,dx\right),
 \qquad
 X_{\mathcal F}(u)=\sin\!\left(\int_\T au\,dx\right)a_x.
\end{equation}
An arbitrary
superlinear $\nabla\Phi$ is not covered by this growth argument without
an independent replacement for \eqref{app:smoothing-radial}.

\paragraph{Smooth-kernel interactions.}
For a fixed real smooth kernel $\mathsf K(x,y)=\mathsf K(y,x)$ on
$\T^2$, set
\begin{equation}\label{app:smooth-kernel}
 \begin{split}
 \mathcal F_{\mathsf K}(u)
    &=\tfrac12\int_\T\int_\T
                         \mathsf K(x,y)u(x)u(y)\,dx\,dy,\\
 X_{\mathcal F_{\mathsf K}}(u)(x)
    &=\int_\T\partial_x\mathsf K(x,y)u(y)\,dy.
 \end{split}
\end{equation}
The field is a bounded linear map $L^2_0\to H^m_0$ for every $m$,
so \eqref{app:smoothing-ball}--\eqref{app:smoothing-radial} follow.
The kernel may have infinite rank; its symmetry gives the stated
Hamiltonian gradient.

\paragraph{Nonlinear spatially smoothed functionals.}
Let $S$ be a fixed real self-adjoint convolution operator on $L^2(\T)$
which is bounded into $H^q(\T)$ for every nonnegative integer $q$.
For example, take the heat smoother $S=e^{\ell^2\partial_x^2}$ with
fixed $\ell>0$.  For a real-analytic $P:\R\to\R$, define
\begin{equation}\label{app:smoothed-density}
 \mathcal F_{S,P}(u)=\int_\T P(Su)\,dx,
 \qquad X_{\mathcal F_{S,P}}(u)=\partial_xS[P'(Su)].
\end{equation}
On each $L^2$ ball the range of $Su$ is in a compact real interval,
where $P$ has uniform local complex extensions and fixed-order
derivative bounds.  The two smoothing factors give
\eqref{app:smoothing-ball} and analyticity of the functional and field.
Self-adjointness and commutation with $\partial_x$ give
\begin{equation}\label{app:smoothed-mass}
 \langle u,X_{\mathcal F_{S,P}}(u)\rangle
      =-\langle(Su)_x,P'(Su)\rangle
      =-\int_\T\partial_xP(Su)\,dx=0.
\end{equation}
  Thus the Hamiltonian flow generated by $\mathcal{F}_{S,P}$ preserves the $L^2$-norm of $u$, so \eqref{app:smoothing-radial} holds with $C_0=0$ even for superlinear $P'$.
In particular,
\[
 \mathcal F(u)=\tfrac14\int_\T(Su)^4\,dx,
 \qquad X_{\mathcal F}(u)=\partial_xS[(Su)^3]
\]
is a nonlinear infinite-dimensional example.  Spatial inhomogeneity is also permitted, for instance
\[
 \mathcal F(u)=\int_\T a(x)(1-\cos(Su(x)))\,dx,
 \qquad X_{\mathcal F}(u)=\partial_xS[a\sin(Su)].
\]
For real $u$, the vector field $X_{\mathcal F}(u)$ is bounded in every $H^m$
by a fixed constant,
and its derivative has the required bounds on $L^2$ balls.  This
verifies the radial condition without the cancellation
\eqref{app:smoothed-mass}.

\subsection{A local nonlinear Gardner perturbation}
\label{subsec:app-gardner}

The next example has a local nonlinear first-order field.  Its
independent bounds use an integrable transformation rather than the
smoothing criterion.

\begin{proposition}[Global bounds for the positive quartic Gardner perturbation]
\label{prop:app-gardner}
The Hamiltonian perturbation
\begin{equation}\label{app:gardner-pde}
 \mathcal F_{\rm G}(u)=\tfrac12\int_\T u^4\,dx,
 \qquad
 u_t=-u_{xxx}+6uu_x+6\eps u^2u_x
\end{equation}
has smooth solutions for the initial family \eqref{app:initial}
for all $t\ge0$ and $0<\eps\le\eps_*$, for a fixed $\eps_*>0$.
The solutions have bounds at every fixed Sobolev order uniform in
$t\ge0$, $0<\eps\le\eps_*$, and the initial family.
More generally, for every $R_*>0$ and any collection $\mathscr U$ of
smooth mean-zero data with $I_A(u_0)\in\mathcal K$ and
$\sup_{u_0\in\mathscr U}\|u_0\|_{H^{N+3}}\le R_*$, there is a fixed
$\eps_*>0$ for which the solutions exist and are smooth for all $t\ge0$ and
\[
 \sup_{u_0\in\mathscr U}\ \sup_{0<\eps\le\eps_*}\ \sup_{t\ge0}
                         \|u(t)\|_{H^{N+2}}<\infty.
\]
\end{proposition}

\begin{proof}
Lemma~\ref{lem:app-local} applies with
$\mathfrak a(x,u)=6u^2$ and the other perturbative terms zero.
On its smooth lifespan, the solution conserves $\|u\|_0^2$ and
\begin{equation}\label{app:gardner-energy}
 H_\eps(u)=\tfrac12\|u_x\|_0^2+\int_\T u^3\,dx
                              +\tfrac\eps2\int_\T u^4\,dx.
\end{equation}
Since the initial mass is bounded, interpolation and Young's inequality give
$|\int u^3|\le\|u_x\|_0^2/4+C$.  The quartic term is nonnegative,
and the initial energy is bounded uniformly for $0<\eps\le1$.
There is therefore a uniform $H^1$ and $L^\infty$ bound on the
available lifespan.  Choose a fixed $\eps_*>0$ sufficiently small that
\begin{equation}\label{app:gardner-positive}
 1+2\eps u(x,t)\ge\tfrac12\qquad(0<\eps\le\eps_*).
\end{equation}

Put $\lambda=\sqrt\eps$ and
\begin{equation}\label{app:gardner-transform}
 q=u+\lambda u_x+\lambda^2u^2.
\end{equation}
This is the Gardner-to-KdV transformation with the present signs; see
\cite[Section~10, equations~(120)--(123)]{HeremanGoktas2024}.
A direct computation yields the differential identity
\begin{equation}\label{app:gardner-intertwining}
 q_t+q_{xxx}-6qq_x
   =(1+\lambda\partial_x+2\lambda^2u)
\cdot
       \bigl(u_t+u_{xxx}-6uu_x-6\lambda^2u^2u_x\bigr).
\end{equation}
Thus $q$ solves ordinary KdV on the same interval.  Its initial values
are uniformly bounded in each $H^m$, using the compact initial family
for $u$ at order $m+1$.  Its mean is
$c=\eps\int_\T u_0^2\,dx$, not necessarily zero.  The function
$q(x-6ct,t)-c$ solves mean-zero KdV.  The conserved KdV hierarchy and
its successive coercive bounds therefore give uniform, time-independent
$H^m$ bounds for $q$, including the bounded range of $c$.

For $m\ge1$, differentiate \eqref{app:gardner-transform}:
\begin{equation}\label{app:gardner-recovery}
 \begin{split}
 q^{(m)}&=(1+2\eps u)u^{(m)}
             +\sqrt\eps\,u^{(m+1)}+\eps R_m,\\
 R_m&=\sum_{j=1}^{m-1}\binom mj u^{(j)}u^{(m-j)}.
 \end{split}
\end{equation}
Pairing with $u^{(m)}$ removes the derivative term by periodic
integration by parts.  At $m=2$, the remainder is $R_2=2u_x^2$, and
its pairing also vanishes exactly:
\[
 \int_\T R_2u_{xx}\,dx
   =2\int_\T u_x^2u_{xx}\,dx
   =\frac23\int_\T\partial_x(u_x^3)\,dx=0.
\]
Using \eqref{app:gardner-positive} in the remaining pairing gives
\[
 \tfrac12\|u_{xx}\|_0^2
   \le\langle q_{xx},u_{xx}\rangle
   \le\|q_{xx}\|_0\|u_{xx}\|_0,
 \qquad \|u_{xx}\|_0\le2\|q_{xx}\|_0.
\]
For $m\ge3$, the preceding $H^{m-1}$ bound controls $R_m$ in $L^2$:
one factor has order at most $m-2$ and is in $L^\infty$, while the
other is in $L^2$.  The lower bound \eqref{app:gardner-positive} then
yields uniform bounds for $u$ at each fixed order, successively.  These bounds hold throughout the interval of smooth existence; hence
Lemma~\ref{lem:app-local} extends the solution to all $t\ge0$.
For the enlarged collection, the initial mass and Hamiltonian are
uniformly bounded on the $H^{N+3}$ ball.  The same positive quartic
energy gives the common $H^1$ bound and one fixed $\eps_*>0$ in
\eqref{app:gardner-positive}.  The transformed initial potential
$u_0+\sqrt\eps\,\partial_xu_0+\eps u_0^2$ is uniformly bounded in
$H^{N+2}$.  The KdV hierarchy through that order, the constant-mean
correction, and \eqref{app:gardner-recovery} give the uniform all-time
$H^{N+2}$ bound for $u$.  Lemma~\ref{lem:app-local} gives global
smooth existence and propagates higher regularity separately for each
datum.
\end{proof}

The perturbing field is a polynomial map $H^s_0\to H^{s-1}_0$ for
$s=N,N+1$.  These analytic mapping properties and the physical bounds,
by Lemma~\ref{lem:app-physical-contract}, verify
Assumption~\ref{ass:contract} on every fixed $[0,T/\eps]$ for the
exact initial family and, upon restricting to $r_0\le r_{\rm adm}$,
for the enlarged collection.

\subsection{General analytic local densities on a fixed short slow-time interval}
\label{subsec:app-local-densities}

Let $P(x,v)$ be real-valued, smooth and periodic with period one in $x$, and
real analytic in $v\in\R$.  More precisely, for every compact real interval in $v$
and every fixed finite number of $x$ derivatives, assume a common
holomorphic extension in $v$ with uniform bounds for those derivatives.
Finite polynomials in $v$ with smooth spatial coefficients, and
$a(x)(1-\cos v)$, satisfy this condition.  Set
\begin{equation}\label{app:local-density}
 \mathcal F_P(u)=\int_\T P(x,u(x))\,dx,
 \qquad
 G_P(u)=P_{xv}(x,u)+P_{vv}(x,u)u_x.
\end{equation}
This is the local-density perturbation class considered in
\cite[equations~(1.4)--(1.7)]{KM2021}, with analyticity in the phase
variable required here.  The following independent existence and Sobolev estimate is uniform over the
prescribed family and all initial phases.

\begin{corollary}[A fixed short slow-time horizon]
\label{cor:app-local-density}
For each fixed $N\ge3$, $P$ as above and the initial family
\eqref{app:initial}, there exists $T_*>0$, independent of
$0<\eps\le1$, such that the equation with perturbation $G_P$ has
smooth solutions on $[0,T_*/\eps]$ satisfying the uniform bound
\eqref{app:physical-upper} with $T_0=T_*$.
Assumption~\ref{ass:contract} therefore holds, and for each
$0<\delta<1/4$ the conclusions \eqref{app:actions}--\eqref{app:gaps}
hold with $T=T_*$ and sufficiently small $\eps$.  
\end{corollary}

\begin{proof}
The assumed holomorphic bounds and Sobolev multiplication give the
compatible analytic maps $G_P:H^s_0\to H^{s-1}_0$ for $s=N,N+1$,
and analyticity of $\mathcal F_P$.  Since $G_P=\partial_xP_v(x,u)$,
it preserves the mean.  Local smooth solutions follow from
Lemma~\ref{lem:app-local} with $\mathfrak a=P_{vv}$,
$\mathfrak b=P_{xv}$ and $\mathcal R=0$.

Fix $m=N+2\ge5$, put $\mathscr E_0=\|u\|_0^2$, and use the
normalized laws $\mathscr E_1,\ldots,\mathscr E_m$ above.  Choose
$M_r\ge0$ bounding $|\mathscr E_r(u_0)|$ on the compact initial
family.  The closed window
\begin{equation}\label{app:local-window}
 |\mathscr E_r(u)|\le M_r+1\qquad(0\le r\le m)
\end{equation}
implies $\|u\|_{H^m}\le R_*$ for a fixed $R_*$.  Indeed,
$\mathscr E_0$ bounds $L^2$, the cubic interpolation estimate and
$\mathscr E_1$ bound $H^1$, and the lower polynomial in each
$\mathscr E_r$ is bounded on bounded $H^{r-1}$ sets for $r\ge2$.
This bounds all subsequent norms by induction.

On this $H^m$ ball there is a fixed constant $C_*$ such that
\begin{equation}\label{app:local-drift}
 |D\mathscr E_r(u)[G_P(u)]|\le C_*
                  \qquad(0\le r\le m).
\end{equation}
For the highest term, write $A(x)=P_{vv}(x,u(x))$.  In
$\partial_x^rG_P$, only $A u^{(r+1)}$ could exceed the available
order.  Its energy pairing is
\[
 2\int_\T u^{(r)}A u^{(r+1)}\,dx
      =-\int_\T\bigl(P_{xvv}(x,u)+P_{vvv}(x,u)u_x\bigr)
                                  (u^{(r)})^2\,dx.
\]
The commutator and composition estimates from
Lemma~\ref{lem:app-local} bound the other terms using only $H^m$.
The derivatives of the lower polynomial in $\mathscr E_r$ likewise
involve no more than $r\le m$ spatial derivatives of $u$ after
substituting $G_P$.  Sobolev multiplication bounds them on the same
ball; in particular $u_x\in L^\infty$ since $m\ge5$.
This proves \eqref{app:local-drift}.

Along the actual smooth solution of the PDE, the unforced contribution vanishes, so
\[
 \frac{d}{dt}\mathscr E_r(u(t))
       =\eps D\mathscr E_r(u(t))[G_P(u(t))].
\]
Choose $T_*=[2(1+C_*)]^{-1}$.  Before a first exit from
\eqref{app:local-window}, and for $t\le T_*/\eps$,
\begin{equation}\label{app:local-improvement}
 |\mathscr E_r(u(t))-\mathscr E_r(u_0)|
       \le\eps t C_*\le\frac{C_*}{2(1+C_*)}<\frac12.
\end{equation}
Every window inequality improves strictly.  The $H^m$ bound and
Lemma~\ref{lem:app-local} continue the solution through $T_*/\eps$
and propagate higher smoothness for each datum and fixed $\eps$.
The uniform bound at $m=N+2$ verifies \eqref{app:physical-upper};
Lemma~\ref{lem:app-physical-contract} and
Corollary~\ref{cor:shrinking-data} with $r_0=0$ give the asserted
stability estimates.
\end{proof}

For example, the corollary applies to
\begin{equation}\label{app:local-examples}
 \begin{aligned}
 P(x,v)=\tfrac14a(x)v^4
   &:\quad G_P=a_xu^3+3a u^2u_x,\\
 P(x,v)=a(x)(1-\cos v)
   &:\quad G_P=a_x\sin u+a\cos u\,u_x.
 \end{aligned}
\end{equation}
The arbitrary-horizon results of
Propositions~\ref{prop:app-transport} and~\ref{prop:app-gardner}
use additional structure beyond the analytic mapping property of a
general local density.

\subsection{Nearby initial data and infinitely many open gaps}
\label{subsec:app-nearby}

The finite-gap restriction in the preceding evolution arguments was
used to obtain uniform high initial norms.  The estimates themselves
apply to smooth nearby data with the corresponding initial bounds.
\begin{corollary}[Applications to nearby and infinite-gap data]
\label{cor:app-nearby}
Fix $N\ge3$, $A$, $\mathcal K$, and $R_*>0$.  Let $\mathscr U$
be any collection of smooth mean-zero initial potentials with
$I_A(u_0)\in\mathcal K$.  Use the associated finite-gap projection
$\xi=(I_A(u_0),0)$ and put $r_0=\|z(0)\|_\sigma$.
The following independent initial bounds suffice for the two-parameter
conclusions \eqref{setup:near-actions}--\eqref{setup:near-gaps},
with constants uniform over $\mathscr U$ whenever $r_0$ and $\eps$
are sufficiently small.
\begin{enumerate}
\item For additive forcing and the transport/smoothing class of
Proposition~\ref{prop:app-transport}, assume
$\sup_{u_0\in\mathscr U}\|u_0\|_{H^{N+2}}\le R_*$.
Every prescribed fixed $T_0>0$ is permitted.
\item For the local densities of Corollary~\ref{cor:app-local-density},
the same $H^{N+2}$ initial bound gives a fixed $T_*>0$, depending
on $N,R_*$ and the density but not on $\eps$, with $T_0=T_*$.
\end{enumerate}
No finite-support condition is imposed on the initial normal actions.
If $r_0\le c_\delta\eps^\delta$ with a sufficiently small uniform
prefactor, Corollary~\ref{cor:shrinking-data} gives the bounds
\eqref{setup:main-actions}--\eqref{setup:main-gaps} with constant~$1$.
\end{corollary}

\begin{proof}
For the first class, the mass estimate and the conservation-law
induction through order $N+2$ require only a bound on the initial
$H^{N+2}$ norm, the fixed coefficients, and the chosen $T_0$.
The initial values of the finitely many conservation laws are bounded
on that Sobolev ball by their displayed differential-polynomial
structure.  Thus
\eqref{app:transport-hierarchy}--\eqref{app:transport-gronwall}
give a uniform $H^{N+2}$ bound on $[0,T_0/\eps]$ for
$0<\eps\le1$.  For each smooth datum, the local tame estimates
propagate its higher regularity over this interval.

For the second class, take $m=N+2$.  The initial $H^m$ bound supplies
uniform numbers $M_r$ for $0\le r\le m$ in
\eqref{app:local-window}.  The same window then gives a uniform
$R_*$-dependent drift bound and the same strictly positive slow-time
horizon.  The strict improvement and $H^4$ continuation proof do not
use finite support of the initial actions.

These bounds verify \eqref{app:physical-upper} for the enlarged
collection.  Restricting to $r_0\le r_{\rm adm}$,
Lemma~\ref{lem:app-physical-contract}, Theorem~\ref{thm:main}, and
Corollary~\ref{cor:shrinking-data} give the stated conclusions.
\end{proof}

For a concrete infinite-gap family, prescribe arbitrary tangential
phases with actions $\xi\in\mathcal K$ and, for $j\notin A$, put
\[
 z_j(0)=\eps^\kappa e^{-j}
                   (\cos\vartheta_j,\sin\vartheta_j),\qquad \kappa>0.
\]
Every $J_j^0=\eps^{2\kappa}e^{-2j}/2$ is positive for $\eps>0$.
The exponential envelope makes the union of these initial Birkhoff
sequences, including all normal phases and $0\le\eps\le1$,
compact in $h^q$ for every fixed $q\in\R$: finite projections have compact
parameter sets and the tails converge uniformly. By the compatible inverse
Birkhoff maps its physical image is compact in $H^s_0$ for every fixed $s\ge0$.  Thus the initial-bound conditions in
Corollary~\ref{cor:app-nearby} and
Theorem~\ref{thm:app-gardner-alltime} hold, and $r_0=O(\eps^\kappa)$.
Any fixed $\kappa>\delta$ gives the bounds
\eqref{setup:main-actions}--\eqref{setup:main-gaps} for these infinite-gap
data.

\subsection{All-time stability of the original KdV actions for Gardner}
\label{subsec:app-gardner-alltime}

The positive quartic example admits a stronger conclusion than the
general averaging argument.  Its Gardner transform differs from the
original potential at order $\sqrt\eps$, but that leading difference
is an infinitesimal spatial translation.  Removing the translation
before comparing actions improves the difference to order $\eps$.

\begin{theorem}[All-time Gardner action and gap estimates]
\label{thm:app-gardner-alltime}
Fix $N\ge3$, a finite set $A$, a compact positive-action family
$\mathcal K$, and $R_*>0$.  Let $\mathscr U$ be a collection of smooth
mean-zero data with $I_A(u_0)\in\mathcal K$ and
$\sup_{u_0\in\mathscr U}\|u_0\|_{H^{N+3}}\le R_*$.
There are constants $r_*>0$, $\eps_*>0$, and $C<\infty$ such that
solutions of the positive quartic Gardner equation
\eqref{app:gardner-pde}, with $r_0\le r_*$ and $0<\eps\le\eps_*$,
exist and are smooth for all $t\ge0$ and satisfy
\begin{equation}\label{app:gardner-alltime-actions}
 \sup_{t\ge0}\sum_{j\ge1}j^{2N+1}
                 |I_j(u(t))-I_j(u_0)|\le C\eps,
\end{equation}
\begin{equation}\label{app:gardner-alltime-normal}
 \begin{gathered}
 \sup_{t\ge0}\|z_u(t)\|_\sigma\le r_0+C\eps,\\
 \sup_{t\ge0}\sum_{j\notin A}j^{2N+1}
                |I_j(u(t))-I_j(u_0)|\le C\eps(r_0+\eps),
       \qquad\sigma=N+1/2,
 \end{gathered}
\end{equation}
and
\begin{equation}\label{app:gardner-alltime-gaps}
 \sup_{t\ge0}\|\gamma(u(t))-\gamma(u_0)\|_{\ell^2}
                  \le C(r_0+\eps).
\end{equation}
Here $z_u$ is the normal part of the real canonical Birkhoff coordinates,
so that $\|z_u\|_\sigma^2=2\sum_{j\notin A}j^{2N+1}I_j(u)$.
Constants are uniform in $\mathscr U$, its phases and its admissible
initial tails.  The solution remains in the fixed finite-gap coordinate
tube.  In particular, for exact finite-gap data,
\begin{equation}\label{app:gardner-exact-alltime}
 \begin{gathered}
 \sup_{t\ge0}|p(t)|_1\le C\eps,\qquad
 \sup_{t\ge0}\sum_{j\notin A}j^{2N+1}I_j(u(t))\le C\eps^2,\\
 \sup_{t\ge0}\|\gamma(u(t))-\gamma(u_0)\|_{\ell^2}\le C\eps.
 \end{gathered}
\end{equation}
\end{theorem}

\begin{proof}
Proposition~\ref{prop:app-gardner} applies to the present $H^{N+3}$
initial bound.  Its independent evolution argument gives global smooth
solutions, a fixed $\eps_*>0$ satisfying
\eqref{app:gardner-positive}, and a uniform all-time $H^{N+2}$ bound.

Put $h=\sqrt\eps$, $q=u+hu_x+h^2u^2$,
$c=\eps\int_\T u_0^2\,dx$ and $v=q-c$.  Equation
\eqref{app:gardner-intertwining} shows that $q$ solves KdV and
$q(x-6ct,t)-c$ solves mean-zero KdV.  Spatial translations preserve
all KdV actions: translation conjugates the Hill operators unitarily,
and the actions are spectral invariants, as in the action formulas of
\cite[Section~3]{KST2013}.  Therefore
\begin{equation}\label{app:gardner-conserved-comparison}
 I_j(v(t))=I_j(v(0))\qquad(j\ge1,\ t\ge0).
\end{equation}
Let $\tau_hu(x)=u(x+h)$.  The integral Taylor formula in $H^N$ gives
\begin{equation}\label{app:gardner-translation-remainder}
 \tau_hu-u-hu_x=\int_0^h(h-y)\tau_yu_{xx}\,dy,\qquad
 \|\tau_hu-u-hu_x\|_{H^N}
                 \le\tfrac12h^2\|u\|_{H^{N+2}}.
\end{equation}
The mass is conserved, so $c=\eps\overline{u(t)^2}$, where the bar
means spatial average.  Hence
\begin{equation}\label{app:gardner-close-translate}
 v(t)-\tau_hu(t)
    =\eps\bigl(u(t)^2-\overline{u(t)^2}\bigr)
                         -\bigl(\tau_hu(t)-u(t)-hu_x(t)\bigr),
 \quad\sup_{t\ge0}\|v(t)-\tau_hu(t)\|_{H^N}\le C\eps.
\end{equation}
The Sobolev product estimate and the uniform $H^{N+2}$ bound justify
the last line.

We next transfer this comparison through the global Birkhoff map.
The family of translated $u(t)$ is bounded in $H^{N+2}$ and the family
of $v(t)$ is bounded in $H^{N+1}$.  Their closures are compact in
$H^N$.  The analytic Birkhoff map
$\mathcal B:H^N_0\to h^\sigma$ therefore has a uniform derivative
bound on sufficiently small neighborhoods of these compact sets;
see \cite[Theorem~1.3]{KST2013}.  A finite cover and the smallness in
\eqref{app:gardner-close-translate} put each connecting segment in a
neighborhood with such a bound.   

Use real canonical pairs $\mathbf Z(w)=\mathcal B(w)$ in this
paragraph, with $I_j(w)=|\mathbf Z_j(w)|^2/2$.  Equation
\eqref{app:gardner-close-translate} yields
\begin{equation}\label{app:gardner-birkhoff-comparison}
 \|\mathbf Z(v(t))-\mathbf Z(\tau_hu(t))\|_\sigma\le C\eps.
\end{equation}
Writing $I_j(\mathbf Z)=|\mathbf Z_j|^2/2$ for sequences,
weighted Cauchy--Schwarz gives
\begin{equation}\label{app:weighted-action-lipschitz}
 \sum_jj^{2\sigma}|I_j(\mathbf Z)-I_j(\widetilde{\mathbf Z})|
 \le\tfrac12(\|\mathbf Z\|_\sigma+\|\widetilde{\mathbf Z}\|_\sigma)
                        \|\mathbf Z-\widetilde{\mathbf Z}\|_\sigma.
\end{equation}
The sequence norms are uniformly bounded on the compact images.
Compare $u(t)$ to $v(t)$ by translation invariance, use
\eqref{app:gardner-conserved-comparison}, and then compare $v(0)$
to $u_0$.  This proves \eqref{app:gardner-alltime-actions}.

Restrict \eqref{app:gardner-birkhoff-comparison} to the normal indices.
The normal norm of $\mathbf Z(v(t))$ is independent of $t$ by
\eqref{app:gardner-conserved-comparison}; translations preserve the
normal norm of $\mathbf Z(u(t))$.  The triangle inequality at times
$t$ and $0$ therefore gives
\[
 \|z_u(t)\|_\sigma
  \le\|z_v(t)\|_\sigma+C\eps
  =\|z_v(0)\|_\sigma+C\eps
  \le r_0+2C\eps.
\]
The same weighted product estimate, restricted to normal indices,
bounds each normal action comparison by $C\eps(r_0+\eps)$.
Adding the comparisons at the two endpoints proves
\eqref{app:gardner-alltime-normal}, after fixing a common constant.

The tangential drift is at most $C\eps$, while the normal size is
at most $r_0+C\eps$.  Thus sufficiently small $r_*,\eps_*$ put the
entire trajectory in the fixed positive-action tube independently of
a stopping argument.  Let $u_\xi$ be any point of the finite-gap torus
with $\xi=(I_A(u_0),0)$.  Use the sharper first inequality in
\eqref{setup:spectral-transfer} at both endpoints, before replacing
$|p|_1^2$ by $|p|_1$.  It gives
\[
 \begin{aligned}
 \|\gamma(u(t))-\gamma(u_0)\|_{\ell^2}^2
 &\le 2\|\gamma(u(t))-\gamma(u_\xi)\|_{\ell^2}^2
        +2\|\gamma(u_0)-\gamma(u_\xi)\|_{\ell^2}^2\\
 &\le C\bigl(|p(t)|_1^2+\|z_u(t)\|_{1/2}^2
                                      +\|z_u(0)\|_{1/2}^2\bigr)\\
 &\le C(r_0+\eps)^2.
 \end{aligned}
\]
This proves \eqref{app:gardner-alltime-gaps}.  If $r_0=0$, the normal norm is bounded by $C\eps$, the weighted normal action
sum by $C\eps^2$, and \eqref{app:gardner-exact-alltime} follows.  The argument also covers
$A=\varnothing$ by omitting the tangential coordinates.
\end{proof}

For general perturbations,
Theorem~\ref{thm:long-polynomial} still requires its separately supplied
uniform $h^{\sigma+2}$ bound on the longer interval.

\subsection{A nonquantitative growing slow-time horizon}
\label{subsec:app-growing-horizon}

The applications valid for every fixed slow-time horizon also admit the
following nonquantitative extension, even without uniform analytic constants
on an explicitly specified growing interval.  It is weaker than the logarithmic and
polynomial conditional results above.

\begin{corollary}[An unspecified growing slow-time factor]
\label{cor:app-growing-horizon}
Fix an additive or transport/smoothing perturbation covered by
Corollary~\ref{cor:app-nearby}, together with its uniform initial
$H^{N+2}$ bound, compact tangential family and $0<\delta<1/4$.
Restrict to exact finite-gap data or to a fixed shrinking family
$r_0\le C_0\eps^\kappa$ with $\kappa>\delta$.
There are $\eps_*>0$ and a positive step function
$\mathsf S(\eps)\to\infty$ as $\eps\downarrow0$ such that the
action and gap estimates with constant $1$ hold on
\begin{equation}\label{app:growing-interval}
 [0,\mathsf S(\eps)/\eps]\qquad(0<\eps\le\eps_*),
\end{equation}
uniformly over that family.  
\end{corollary}

\begin{proof}
For each integer $n\ge1$, the arbitrary-horizon physical estimates
and Corollary~\ref{cor:shrinking-data} give a threshold
$\eps_0(n)>0$ for the desired conclusion through $n/\eps$.
For shrinking data, $C_0\eps^\kappa\le c_\delta(n)\eps^\delta$
holds uniformly below a further positive threshold, which is
incorporated in $\eps_0(n)$.  Choose a strictly decreasing sequence
$e_n>0$, with $e_0=1$, such that
\[
 e_n\le\min\{2^{-n},e_{n-1}/2,\eps_0(n)/2\}.
\]
Set $\mathsf S(\eps)=n$ when $e_{n+1}<\eps\le e_n$.
Then $\mathsf S(\eps)\to\infty$ and $\eps<\eps_0(n)$ on each
such interval.  The conclusion for that fixed integer $n$ proves
\eqref{app:growing-interval}.  This construction uses the theorem
separately at each horizon; it does not iterate the action bound
along the trajectory.
\end{proof}

Because the admissible $\eps$-thresholds and higher-norm bounds may
depend on the fixed horizon, the diagonal choice gives no quantitative
rate for $\mathsf S(\eps)$ and does not verify the uniform long-interval
hypotheses. It requires arbitrary fixed horizons; general local
densities have only the short-horizon conclusion of
Corollary~\ref{cor:app-local-density}.

All applications above are autonomous and have at most the stipulated
one-derivative loss.  A density involving $u_x$ need not belong to this
class: $\frac12\int a(x)u_x^2\,dx$ has field
$-\partial_x^2(a u_x)$, which is generally third order.

% END sections/07_application

\appendix
% BEGIN sections/08_finite_gap_hessian
\section{The Hessian at an arbitrary finite-gap action}
\label{sec:finitegap-hessian}

We prove the coercivity statement used in the physical energy estimate
without a smallness restriction on the finite-gap base.  Throughout this
appendix, the potentials have period one and zero mean.  Write
\begin{equation}
 h(I)=\sum_{n\ge1}(2\pi n)^3I_n-V(I),\qquad
 V=-h_{\mathrm{nl}},\qquad
 \omega_n^*(I)=\omega_n(I)-(2\pi n)^3.
 \label{fg:conventions}
\end{equation}
The function $V$ has a real-analytic extension to a neighborhood of the
positive cone in $\ell^2_{\R}$, and
\begin{equation}
 H(I):=D^2V(I)=-D\omega^*(I),\qquad H(0)=6\operatorname{Id}.
 \label{fg:hessian}
\end{equation}
Here $H(I)$ is a bounded self-adjoint operator.  The full dispersive sum
in~\eqref{fg:conventions} need not converge for an arbitrary
$I\in\ell^2$; in this appendix, we use the full Hamiltonian only on
finite-dimensional finite-gap families.  The analytic extension and the normalization at
zero are established in \cite{KMMT2016,KappelerMolnar2018}.  Global weak
convexity of $V$ on $\ell^2_+$ follows from
\cite[Corollary~1.2]{KorotyaevKuksin2011}.

\begin{theorem}[Finite-gap coercivity]\label{thm:finitegap-coercivity}
For every $\xi\in\ell^2_+$ with finite support, there exists $m_\xi>0$
such that
\begin{equation}
 \langle D^2V(\xi)u,u\rangle\ge m_\xi\|u\|_2^2
 \qquad(u\in\ell^2_{\R}).
 \label{fg:coercivity}
\end{equation}
The constant is uniform when $\xi$ ranges over a compact subset of
$\ell^2_+$ consisting of finite-gap actions.
\end{theorem}

The proof gives an exact diagonal-plus-finite-rank formula for $H(\xi)$.
The only external nondegeneracy result used concerns the frequency map on the
stratum determined by the open gaps; the transverse directions, including the opening
of a collapsed gap, are treated below.

\subsection{The reduced curve and normalized differentials}

Fix $\xi$, set $S=\{s:\xi_s>0\}$ and $g=|S|$, and first suppose $g>0$.
Let $e_0$ be the lowest spectral endpoint and $e_s^-<e_s^+$ the endpoints
of each open gap.  On the reduced curve define
\begin{equation}
 Y(\lambda)^2=(\lambda-e_0)
       \prod_{s\in S}(\lambda-e_s^-)(\lambda-e_s^+),
 \qquad p=iF,\qquad a=2p'=\frac{P}{Y}.
 \label{fg:curve}
\end{equation}
The infinity sheet is chosen so that $Y\sim\lambda^{g+1/2}$ and
$a\sim\lambda^{-1/2}$.  Thus $P$ is a monic polynomial of degree $g$.
Its simple zeros $c_1,\ldots,c_g$ lie one in each open gap and are not
branch points.  The symbol $F$ denotes the standard primitive of the
spectral differential used in \cite{KappelerMolnar2018}; $p=iF$ is the
quasimomentum in the present convention.

Let $\Gamma_s$ be a positively oriented contour around the $s$-th gap.
For every action index $n$ put
\begin{equation}
 q_n=\frac{i\psi_n}{\sqrt[c]{\Delta^2-4}},\qquad
 \oint_{\Gamma_s}q_n\,d\lambda=2\pi i\delta_{ns}.
 \label{fg:normalized}
\end{equation}
Here $\Delta$ is the Hill discriminant, $\psi_n$ is the normalized
entire numerator, and the superscript $[c]$ specifies the canonical
square root.  These spectral identities, their finite-product forms,
and the expansion at infinity follow from
\cite[equations~(29), (33), and (52)--(54),
arXiv:1605.06690v1]{KappelerMolnar2018}; the action interpretation is
that of \cite[Theorem~1.3 and Section~3]{KST2013}.
For $n\in S$, write $q_n=N_n/Y$, where $N_n$ is a real polynomial of
degree at most $g-1$.  For $n\notin S$, $N_n$ is a real rational function whose only finite pole is a simple pole at the closed
node $t_n=\tau_n$, where $\tau_k$ denotes the midpoint of the $k$-th gap.
On the selected sheet the residue of $q_n$ at
$t_n$ is $+1$; it is $-1$ on the other sheet.

\begin{lemma}[Finite period normalization]\label{fg:normalization}
The matrix
\begin{equation}
 \mathcal A=\left(
   \oint_{\Gamma_s}\frac{\lambda^\ell}{Y(\lambda)}\,d\lambda
                 \right)_{s\in S,\ 0\le\ell<g}
 \label{fg:period-matrix}
\end{equation}
is invertible.  More generally, adjoining any finite set of collapsed
gaps gives an invertible finite normalization matrix at the collapsed
base.  In a fixed family with active indices $E\supset S$, the
normalized differentials depend analytically on its finite action
parameters, including at zero actions.
\end{lemma}

\begin{proof}
For a real polynomial $B$ of degree less than $g$, each gap period of
$B/Y$ is a nonzero constant imaginary factor times an integral of $B$
against a strictly positive weight on that gap.  If all periods vanish,
$B$ has a zero in each of the $g$ disjoint open gaps and hence is zero.
For complex coefficients, remove the fixed row phases and apply the
same argument to real and imaginary parts.  This proves invertibility
of~\eqref{fg:period-matrix}.

At a base with additional collapsed gaps indexed by $E\setminus S$,
the unreduced expression is
\[
 Y_E=\pm Y\prod_{k\in E\setminus S}(\lambda-t_k).
\]
For a numerator polynomial of degree at most $|E|-1$, a zero period
around a collapsed gap forces its value at $t_k$ to vanish.  Cancel
these factors.  The remaining polynomial has degree at most $g-1$ and
zero original periods, so the first argument applies.  This proves
invertibility at the collapsed base.   

For the analytic assertion, fix positively oriented contours $\Gamma_k$,
$k\in E$, enclosing the corresponding gaps and their sufficiently small
openings, and continue the square-root branches from the base. All action
derivatives of period integrals are taken on these fixed contours. Write
\begin{equation}
 Y_E(I,\lambda)=\sqrt{\lambda-e_0(I)}
    \prod_{k\in E}
       \sqrt{(\lambda-\tau_k(I))^2-\gamma_k(I)^2/4}.
 \label{fg:unreduced}
\end{equation}
The symmetric spectral variables in this formula are analytic in the
finite actions.  At a zero action this can also be seen by taking real
Birkhoff amplitudes: an isospectral analytic quantity is even in each
amplitude by rotation invariance, and therefore descends analytically
to its squared amplitude, hence to its action.  The square roots on the
contours are separated from zero, so the same assertion holds for the
contour integrals.  Finite matrix inversion now gives the unique
polynomial $C_n(I,\lambda)$, of degree at most $|E|-1$, satisfying
\begin{equation}
 \oint_{\Gamma_k}\frac{C_n(I,\lambda)}{Y_E(I,\lambda)}\,d\lambda
     =2\pi i\delta_{kn}\qquad(k\in E).
 \label{fg:finite-normalization}
\end{equation}
Its coefficients are analytic.  The exact finite-product identities
and normalization identify $q_n=C_n/Y_E$ for $n\in E$.
\end{proof}

We next derive the action-variation identity.
First take every action indexed by $E$ strictly positive.  In the
Poisson-bracket convention of \cite[Lemma~2.5 and the proof of
Lemma~3.2]{KappelerMolnar2018},
\[
 \{h,\theta_n\}_{\rm src}=\omega_n,\qquad
 \{\theta_n,F\}_{\rm src}
 =\frac{\{\theta_n,\Delta\}_{\rm src}}
        {\sqrt[c]{\Delta^2-4}},\qquad
 2\{\theta_n,\Delta\}_{\rm src}=\psi_n,
\]
where  $\theta_n$ denotes the Birkhoff angle in the convention of \cite{KappelerMolnar2018}.  Since $F$ is
isospectral, the first convention gives
$\{\theta_n,F\}_{\rm src}=-\partial_{I_n}F$.  Consequently
\begin{equation}
 \partial_{I_n}F=-\frac{\psi_n}{2\sqrt[c]{\Delta^2-4}},
 \qquad -2\partial_{I_n}p=q_n.
 \label{fg:action-variation}
\end{equation}
On fixed contours, the primitive $p$ depends analytically on the finite
actions, including at collapsed gaps, by evenness in the Birkhoff
amplitudes. Lemma~\ref{fg:normalization} gives the same dependence for
$q_n$. Thus \eqref{fg:action-variation} extends to the collapsed base by
continuity of $\partial_{I_n}p$ and $q_n$.

We shall apply this lemma with $E=S\cup\{n,j\}$.  Consequently,
$\partial_{I_j}q_n$ at the base is the product of a rational function and $1/Y$, with possible
singularities only at the original branch points and the selected
closed nodes, and is $O(\lambda^{-3/2})$ at infinity.  

For a nonbranch point $t$ on a selected local sheet, define
\begin{equation}
 K_t(\lambda)=
 \frac{Y(t)+(\lambda-t)Y'(t)}{(\lambda-t)^2Y(\lambda)},\qquad
 D_t=K_t+\frac{L_t}{Y},\quad \deg L_t\le g-1,
 \qquad \oint_{\Gamma_s}D_t\,d\lambda=0.
 \label{fg:second-kind}
\end{equation}
Lemma~\ref{fg:normalization} determines $L_t$ uniquely.  Expanding $Y$
at $t$ shows that the differential $D_t\,d\lambda$ has double principal
part $+(\lambda-t)^{-2}d\lambda$ and no residue on that sheet, and the
opposite principal part on the other sheet.  It is holomorphic
elsewhere, including at the branch points in their square-root local
coordinates and at infinity.  Uniqueness follows from the same period
matrix.  For $t=c_k$, the contours are kept outside a local disk about
$c_k$; passing a contour across this double pole does not alter a
period because its residue is zero.
Set
\begin{equation}
 b_k=[\lambda^{-5/2}]D_{c_k}(\lambda),\qquad
 \beta_k=\frac{8b_k}{a'(c_k)},
 \label{fg:beta}
\end{equation}
where brackets denote the indicated coefficient in the expansion on
the infinity sheet.

\subsection{The complete two-index variation}

Let $e_{nj}=1$ when $n=j\notin S$ and $e_{nj}=0$ otherwise.  Terms
involving $t_n$ below are omitted when $e_{nj}=0$.  Define
\begin{align}
 W_{nj}(\lambda)
   &=\frac{q_n(\lambda)q_j(\lambda)
                    -e_{nj}/[2(\lambda-t_n)^2]}{a(\lambda)},
 \label{fg:W}\\
 r_{nj,k}
   &=\frac{q_n(c_k)q_j(c_k)
                    -e_{nj}/[2(c_k-t_n)^2]}{a'(c_k)}.
 \label{fg:residue}
\end{align}

The global infinity normalization remains fixed when the chosen local
lift of $c_k$ is moved to the other sheet.  The values $q_n(c_k)$ and
$a'(c_k)$ then change sign.  Both $Y(c_k)$ and $Y'(c_k)$ in
\eqref{fg:second-kind} change sign, so uniqueness of the zero-period
normalization also changes $D_{c_k}$ and $b_k$ to their negatives.
The numerator of $r_{nj,k}$ is unchanged and its denominator changes
sign.  Consequently $\beta_k$, $\beta_kq_n(c_k)q_j(c_k)$, and
$r_{nj,k}D_{c_k}$ are independent of this local sheet choice.

\begin{lemma}[Two-index variation]\label{fg:variation}
For all $n,j\ge1$, with differentiation in the finite action family
$S\cup\{n,j\}$,
\begin{equation}
 \partial_{I_j}q_n
    =-\partial_\lambda W_{nj}
       -\sum_{k=1}^g r_{nj,k}D_{c_k}.
 \label{fg:variation-identity}
\end{equation}
\end{lemma}

\begin{proof}
We check all possible singularities and then determine the global
remainder by its periods.  Write $\partial_j=\partial_{I_j}$.

At an original branch point $e$, the numerators of $a=P/Y$ and
$q_j=N_j/Y$ are regular, and $P(e)\ne0$.  Since
$\partial_j a=-q_j'$, comparison of the leading terms gives
\begin{equation}
 \partial_j e=\frac{N_j(e)}{P(e)}.
 \label{fg:endpoint-motion}
\end{equation}
Thus the leading term of $\partial_jq_n$ is
$N_n(e)N_j(e)/[2P(e)Y(\lambda)(\lambda-e)]$, the same as that of
$-\partial_\lambda(q_nq_j/a)$.  Their difference, multiplied by
$d\lambda$, is regular in the coordinate $\sqrt{\lambda-e}$.  The
self correction in~\eqref{fg:W} also gives a regular differential
there.

If $n\notin S$ and $j\ne n$, the $n$-th gap remains closed along this
variation.  Differentiating $p(t_n)=n\pi$ and using
$\partial_jp=-q_j/2$ gives
\begin{equation}
 \partial_jt_n=\frac{q_j(t_n)}{a(t_n)}.
 \label{fg:node-motion}
\end{equation}
The residue of $q_n$ remains equal to one, so $\partial_jq_n$ has a double
pole with this coefficient and no simple pole.  These are the principal
parts of $-\partial_\lambda(q_nq_j/a)$ at $t_n$.

If $j\notin S$ and $n\ne j$, set $v_j=\gamma_j^2/4$.  Comparison of
the double-pole terms in $\partial_ja=-q_j'$, equivalently comparison
in the local gap factor, gives
\begin{equation}
 \partial_jv_j=\frac{2}{a(t_j)}.
 \label{fg:gap-opening}
\end{equation}
Differentiating the new square-root denominator then gives a double
pole of $\partial_jq_n$ with coefficient $q_n(t_j)/a(t_j)$.  Its 
residue vanishes by differentiation of the $j$-period
in~\eqref{fg:finite-normalization}.  These are again the principal
parts of $-\partial_\lambda(q_nq_j/a)$.  For distinct closed indices
$n,j$, these two cases verify the principal parts at both $t_n$ and $t_j$.

It remains to consider $n=j\notin S$.  Put $x=\lambda-t_n(\xi)$,
and take action derivatives at fixed $\lambda$. At the base, write
\begin{equation}
 q_n=x^{-1}+q_0+O(x),\qquad a=a_0+a_1x+O(x^2).
 \label{fg:self-expansion}
\end{equation}
Along the opening family, set $t_n(I)=\tau_n(I)$ and
$v_n(I)=\gamma_n(I)^2/4$.  There is a local factorization
\begin{equation}
 p(I,\lambda)-n\pi
  =M(I,\lambda)
       \sqrt{(\lambda-t_n(I))^2-v_n(I)}.
 \label{fg:local-factorization}
\end{equation}
Indeed, $(p-n\pi)^2$ is analytic locally and vanishes at the two
endpoints.  Division by their quadratic factor gives a nonzero analytic
quotient, whose local analytic square root is $M$.  At the base,
$M(\xi,\lambda)=a_0/2+(a_1/4)x+O(x^2)$.
Differentiating at the base gives
\[
 \partial_np=-\frac{a_0\partial_nv_n}{4x}
             -\frac{a_0\partial_nt_n}{2}
             -\frac{a_1\partial_nv_n}{8}+O(x).
\]
Comparison with $\partial_np=-q_n/2$ gives
\begin{equation}
 \partial_nv_n=\frac2{a_0},\qquad
 \partial_nt_n=\frac{q_0}{a_0}-\frac{a_1}{2a_0^2}.
 \label{fg:self-motion}
\end{equation}
Locally, write the normalized differential as
$B(I,\lambda)/\sqrt{(\lambda-t_n(I))^2-v_n(I)}$, where at the base
$B=1+q_0x+O(x^2)$.  The coefficients of $x^{-3}$ and $x^{-2}$ in $\partial_nq_n$ are
$(\partial_nv_n)/2=1/a_0$ and $\partial_nt_n+q_0(\partial_nv_n)/2$,
respectively.  Its residue is zero  because the normalized $n$-period of $q_n$ is constant.  Hence
\begin{equation}
 \partial_nq_n=
 \frac1{a_0x^3}
  +\left(\frac{2q_0}{a_0}-\frac{a_1}{2a_0^2}\right)x^{-2}
  +O(1).
 \label{fg:self-principal}
\end{equation}
In comparison,
\[
 -\partial_\lambda(q_n^2/a)=
 \frac2{a_0x^3}
   +\left(\frac{2q_0}{a_0}-\frac{a_1}{a_0^2}\right)x^{-2}
   +O(1).
\]
The missing principal part is exactly
\begin{equation}
 -\frac1{a_0x^3}+\frac{a_1}{2a_0^2x^2}
   =\operatorname{pp}_{x=0}
        \partial_\lambda\frac1{2a(\lambda)x^2}.
 \label{fg:self-correction}
\end{equation}
At a critical point $c_k$, the actual derivative $\partial_jq_n$ is
regular.  The function $W_{nj}$ has residue $r_{nj,k}$, so
the coefficient of $(\lambda-c_k)^{-2}$ in $-\partial_\lambda W_{nj}$ is $+r_{nj,k}$.  It is
canceled by $-r_{nj,k}D_{c_k}$, with no simple pole.  All these
statements hold on both sheets because the differentials are odd.

It follows that
\[
 \mathcal R=\partial_jq_n+\partial_\lambda W_{nj}
                +\sum_k r_{nj,k}D_{c_k}
\]
becomes a holomorphic odd differential after multiplication by
$d\lambda$.  At infinity, the finite-action construction and the
formulas for $W_{nj}$ and $D_{c_k}$ give
$\mathcal R(\lambda)=O(\lambda^{-3/2})$.  In the local coordinate
$s=\lambda^{-1/2}$,
\[
 d\lambda=-2s^{-3}\,ds,\qquad
 \mathcal R(\lambda)\,d\lambda=O(1)\,ds,
\]
so the differential is regular there as well.  More explicitly, the
finite construction in Lemma~\ref{fg:normalization} expresses $\mathcal R$ as the product of a rational function and $1/Y$, and the preceding principal-part checks imply
$\mathcal R=P_0/Y$ for a polynomial of degree at most $g-1$.  Its
original gap periods vanish: the first term has zero periods, obtained by
differentiating the constant periods of $q_n$; the second is the derivative of
a single-valued meromorphic function; and the remaining differentials have zero
periods by definition.  Invertibility of~\eqref{fg:period-matrix} forces $P_0=0$.
This proves~\eqref{fg:variation-identity}.
\end{proof}

\subsection{The Hessian formula and its operator meaning}

\begin{proposition}[Complete Hessian formula]\label{fg:hessian-formula}
At the fixed finite-gap base, for every $n,j\ge1$,
\begin{equation}
 H_{nj}=\sum_{k=1}^g\beta_kq_n(c_k)q_j(c_k)+\delta_{nj}d_n,
 \qquad
 d_n=\begin{cases}
 0,&n\in S,\\[2pt]
 6-\displaystyle\frac12\sum_{k=1}^g
       \frac{\beta_k}{(c_k-t_n)^2},&n\notin S.
 \end{cases}
 \label{fg:entries}
\end{equation}
The right-hand side defines a bounded diagonal operator plus an
operator of rank at most $g$ on $\ell^2$.
\end{proposition}

\begin{proof}
At a finite-gap base, quasimomentum has a convergent expansion at
infinity.  In the conventions of \cite[Lemma~2.5(v), equation~(33), and the
proof of Lemma~3.2]{KappelerMolnar2018},
$F^4=\lambda^2-H_0-h/(4\lambda)+O(\lambda^{-2})$.  Since $p=iF$,
\begin{equation}
 p=\sqrt\lambda-
       \frac{H_0}{4\lambda^{3/2}}
       -\frac{h}{16\lambda^{5/2}}+O(\lambda^{-7/2}),
 \qquad H_0=\sum_{n\ge1}2\pi nI_n.
 \label{fg:infinity-p}
\end{equation}
The coefficients are analytic in the fixed finite family used in
Lemma~\ref{fg:variation}.  Differentiation gives
\begin{equation}
 q_n=\pi n\lambda^{-3/2}
     +\frac{\omega_n}{8}\lambda^{-5/2}+O(\lambda^{-7/2}),
 \qquad
 [\lambda^{-5/2}]\partial_jq_n=-\frac{H_{nj}}8.
 \label{fg:infinity-q}
\end{equation}
Now $-\partial_\lambda(q_nq_j/a)=O(\lambda^{-7/2})$, whereas
\begin{equation}
 \partial_\lambda\frac1{2a(\lambda)(\lambda-t_n)^2}
     =-\frac34\lambda^{-5/2}+O(\lambda^{-7/2}).
 \label{fg:infinity-self}
\end{equation}
Extracting the $\lambda^{-5/2}$ coefficient
in~\eqref{fg:variation-identity} and multiplying by $-8$
proves~\eqref{fg:entries}.  In particular, for $n\notin S$, the diagonal term $d_n$ contains the universal
term $6$ before its critical-point corrections.

To justify the operator assertion, fix $c_k$.  For a large closed
index $n$, write
\[
 q_n(\lambda)=\frac{B_n(\lambda)}{(\lambda-t_n)Y(\lambda)},
 \qquad \deg B_n\le g.
\]
The vanishing periods around the $g$ originally open gaps determine its lower
coefficients in
terms of its leading coefficient.  After multiplying these equations
by $-t_n$, their coefficient matrix converges
 to~\eqref{fg:period-matrix}.  Consequently the resulting monic
polynomial $A_n$ has uniformly bounded coefficients.  The residue
normalization gives
\[
 B_n(\lambda)=\frac{Y(t_n)}{A_n(t_n)}A_n(\lambda).
\]
Since $Y(t_n)\sim t_n^{g+1/2}$, $A_n(t_n)\sim t_n^g$, and
$t_n\sim\pi^2n^2$, the prefactor is $O(n)$.  At $\lambda=c_k$,
$|c_k-t_n|$ has order $n^2$ and $Y(c_k)\ne0$.  We obtain
\begin{equation}
 q_n(c_k)=O_\xi(n^{-1}),\qquad
 d_n=6+O_\xi(n^{-4})\quad(n\notin S,\ n\to\infty).
 \label{fg:row-decay}
\end{equation}
All remaining denominators are nonzero: a closed node cannot lie in an
open gap.  Thus each of the finitely many evaluation rows belongs to
$\ell^2$ and $(d_n)$ is bounded.  Formula~\eqref{fg:entries} first
gives equality on finite sequences.  Both sides are bounded operators,
so density proves equality on all of $\ell^2$.
\end{proof}

\subsection{Nondegeneracy on the open stratum}

\begin{lemma}[Tangential positivity]\label{fg:tangential-positivity}
For the open support $S$, the matrix $H_{SS}(\xi)$ is positive definite.
\end{lemma}

\begin{proof}
For every finite set $S\subset\N$, the action--frequency map
\[
 \omega_S:(0,\infty)^S\longrightarrow\R^S,
 \qquad I_S\longmapsto
       \bigl(\omega_s(I_S,0)\bigr)_{s\in S},
\]
is a local diffeomorphism by
\cite[Lemma~2.1]{KappelerKramer2024}. This is the finite-gap
nondegeneracy result of \cite{BikbaevKuksin1993}; see also
\cite[Appendix~6]{Kuksin2000}. The cited formulation is already in
Birkhoff actions for the period-one KdV equation and applies to
arbitrary finite gap labels, with no nonresonance assumption.
Since the dispersive part of $h$ is linear in the actions,
\[
 H_{SS}(\xi)=-D_{I_S}\omega_S(\xi_S),
\]
so $H_{SS}(\xi)$ is nonsingular.

Global convexity gives $H\ge0$ on the physical cone, including
its boundary. To see the boundary assertion directly, take a finite
test direction and a nearby action with all its relevant coordinates
strictly positive. Convexity gives a nonnegative quadratic form there;
continuity gives the same at the base, and density treats every
$\ell^2$ direction. Thus $H_{SS}\ge0$, and its nonsingularity makes
it positive definite.
\end{proof}

\subsection{Strict signs and the uniform energy defect}

\begin{proof}[Proof of Theorem~\ref{thm:finitegap-coercivity}]
At a real critical point in an open gap, $Y(c_k)$ is purely imaginary
and the normalized numerators are real.  Write
\[
 q_n(c_k)=iR_{kn},\qquad R_{kn}\in\R.
\]
The numerators $N_s$, $s\in S$, form a basis of the polynomials of
degree at most $g-1$, by their normalized periods.  Evaluation at the
$g$ distinct points $c_k$ is invertible.  Hence the real matrix
$R_S=(R_{ks})_{k,s}$ is invertible.  Restricting~\eqref{fg:entries}
to $S$ gives
\begin{equation}
 H_{SS}=-R_S^{\mathsf T}
             \operatorname{diag}(\beta_1,\ldots,\beta_g)R_S.
 \label{fg:congruence}
\end{equation}
Inversion first shows that each $\beta_k$ is real.  By
Lemma~\ref{fg:tangential-positivity}, the congruence then implies
$\beta_k<0$ for every $k$.

Put $\alpha_k=-\beta_k>0$ and define the bounded map
$T:\ell^2_{\R}\to\R^g$ by
\begin{equation}
 (Tu)_k=\sqrt{\alpha_k}\sum_{n\ge1}R_{kn}u_n.
 \label{fg:T}
\end{equation}
Its restriction $T_S$ is invertible.  The boundedness was proved
in~\eqref{fg:row-decay}.  The exact operator formula becomes
\begin{equation}
 \langle Hu,u\rangle=
   \sum_{n\notin S}
    \left(6+\frac12\sum_{k=1}^g
                  \frac{\alpha_k}{(c_k-t_n)^2}\right)u_n^2
      +\|Tu\|_{\R^g}^2.
 \label{fg:positive-factorization}
\end{equation}
In particular the normal diagonal is at least $6$.
For $g>0$, let $D=\operatorname{diag}(d_n)_{n\notin S}$ on the normal
subspace $\ell^2_{\R}(S^c)$, and write $T=(T_S,T_{S^c})$.
The block form of \eqref{fg:positive-factorization} gives
$H_{SS}=T_S^*T_S$, $H_{SS^c}=T_S^*T_{S^c}$, and
$H_{S^cS^c}=D+T_{S^c}^*T_{S^c}$.  Since $T_S$ is invertible,
$T_S(T_S^*T_S)^{-1}T_S^*=\operatorname{Id}_{\R^g}$; hence the normal
Schur complement is
\[
 H_{S^cS^c}-H_{S^cS}H_{SS}^{-1}H_{SS^c}
   =D\ge6\operatorname{Id}_{\ell^2_{\R}(S^c)}.
\]
For an explicit lower bound let
\begin{equation}
 C=\|T_S^{-1}\|,\qquad M=\|T_{S^c}\|,\qquad
 m_\xi=\left[\max\left\{2C^2,
                        \frac{1+2C^2M^2}{6}\right\}\right]^{-1}.
 \label{fg:explicit-bound}
\end{equation}
Writing $u=(x,y)$ in $S\oplus S^c$ gives
\[
 \|x\|^2\le2C^2\|Tu\|^2+2C^2M^2\|y\|^2,
 \qquad
 \langle Hu,u\rangle\ge\|Tu\|^2+6\|y\|^2.
\]
These inequalities prove~\eqref{fg:coercivity}.  When $g=0$, the
normalization~\eqref{fg:hessian} gives $H=6\operatorname{Id}$ and
$m_\xi=6$.

For the uniformity statement, the map
$\xi\mapsto H(\xi)$ is continuous in operator norm.  The infimum of
the quadratic form over the unit sphere is a $1$-Lipschitz function of
a bounded self-adjoint operator.  It is strictly positive at every
finite-gap base by the preceding argument, so its minimum on a compact
set of such bases is strictly positive.  This also covers compact
families in which gaps close; the explicit constants
in~\eqref{fg:explicit-bound} need not be continued across a change of
support.
\end{proof}

\begin{corollary}[Uniform physical energy defect]
\label{cor:finitegap-tube}
Let $\mathcal K\subset\ell^2_+$ be compact and consist entirely of
finite-gap actions.  There exist $m_{\mathcal K}>0$ and $r_{\mathcal K}>0$
such that for every $\xi\in\mathcal K$ and every real
$v\in\ell^2$ with $\|v\|_2<r_{\mathcal K}$ and $\xi+v\ge0$,
\begin{equation}
 V(\xi+v)-V(\xi)-DV(\xi)v
       \ge\frac{m_{\mathcal K}}4\|v\|_2^2.
 \label{fg:energy-defect}
\end{equation}
The Hessian at every $\xi\in\mathcal K$ is bounded below by
$m_{\mathcal K}\operatorname{Id}$.
\end{corollary}

\begin{proof}
The last conclusion follows from
Theorem~\ref{thm:finitegap-coercivity}.  Analyticity and compactness give
a common tube around $\mathcal K$, within the domain of the extension,
on which $D^2V\ge(m_{\mathcal K}/2)\operatorname{Id}$.  Choose
$r_{\mathcal K}$ so that every segment $\xi+tv$, $0\le t\le1$, in the
statement stays in this tube.  Taylor's integral formula gives
\[
 V(\xi+v)-V(\xi)-DV(\xi)v
    =\int_0^1(1-t)\langle D^2V(\xi+tv)v,v\rangle\,dt,
\]
which proves~\eqref{fg:energy-defect}.
\end{proof}

 % END sections/08_finite_gap_hessian

\begingroup
\small
\setlength{\baselineskip}{11.4pt}
\setlength{\bibsep}{0pt}

\endgroup
\end{document}